%% file: main.tex
\documentclass[opre,nonblindrev]{informs3} 

\OneAndAHalfSpacedXI 

\usepackage{endnotes}
\let\enotesize=\normalsize
\def\notesname{Endnotes}%
\def\makeenmark{$^{\theenmark}$}
\def\enoteformat{\rightskip0pt\leftskip0pt\parindent=1.75em
  \leavevmode\llap{\theenmark.\enskip}}

\usepackage{tikz}

\usepackage{natbib}
 \bibpunct[, ]{(}{)}{,}{a}{}{,}%
 \def\bibfont{\small}%
 \def\newblock{\ }%

\input{preamble}

\TheoremsNumberedThrough     
\ECRepeatTheorems

\EquationsNumberedThrough    

\begin{document}

\RUNAUTHOR{Freund, Hssaine, and Zhao}
\RUNTITLE{On the Power of Delayed Flexibility}  
\TITLE{On the Power of Delayed Flexibility:\\Balls, Bins, and a Few Opaque Promotions}

\ARTICLEAUTHORS{%
\AUTHOR{Daniel Freund}
\AFF{Massachusetts Institute of Technology, Cambridge, MA\\ \EMAIL{dfreund@mit.edu}}
\AUTHOR{Chamsi Hssaine}
\AFF{University of Southern California, Marshall School of Business, Los Angeles, CA\\ \EMAIL{hssaine@usc.edu}}
\AUTHOR{Jiayu (Kamessi) Zhao}
\AFF{Massachusetts Institute of Technology, Cambridge, MA\\ \EMAIL{kamessi@mit.edu}}
} 

\ABSTRACT{
\input{abstract}
}

\maketitle

 \input{intro.tex}

 \input{preliminaries.tex}

\input{bib-results.tex}

\input{opaque-2024}

\input{opaque-numerics}

\input{conclusion.tex}

\SingleSpacedXI
\bibliographystyle{plainnat}

\input{main.bbl}
\begin{APPENDICES}{}
\OneAndAHalfSpacedXI
\input{appendix.tex}

\end{APPENDICES}

\end{document}

%% file: preamble.tex
\usepackage{url}            

\usepackage{tcolorbox}
\usepackage{bbm}
\usepackage{dsfont}
\usepackage{amsmath,amsfonts,amssymb,commath,hyperref}
\usepackage{natbib}
\usepackage{caption}
\usepackage[colorinlistoftodos,prependcaption,textsize=small]{todonotes}
\usepackage[shortlabels]{enumitem}
\usepackage{xspace}
\usepackage[multiple]{footmisc}

\usepackage{algorithm}
\usepackage[noend]{algpseudocode}

\usepackage[capitalize]{cleveref}

\usepackage[suppress]{color-edits}
\addauthor{ch}{blue}
\addauthor{df}{orange}
\addauthor{kz}{purple}

\usepackage{natbib}
 \bibpunct[, ]{(}{)}{,}{a}{}{,}%
 \def\bibfont{\small}%
 \def\newblock{\ }%
\usepackage{booktabs} 

\usepackage{subcaption}
\usepackage{mathtools}
\usepackage{float}

\newcommand{\perperiodrev}{\mathcal{R}}

\newcommand{\Gap}{Gap}
\newcommand{\Gapnf}{Gap^{nf}}
\newcommand{\Gapp}{Gap^{\pi}}
\newcommand{\Gapa}{Gap^{a}}
\newcommand{\Gaps}{Gap^{s}}
\newcommand{\Gapd}{Gap^{d}}
\newcommand{\Gapf}{Gap^{d}}

\newcommand{\Gaprestart}{Gap^{\text{unload}}}

\newcommand{\zsemi}{z^{d}}
\newcommand{\zreal}{x^{\pi}}
\newcommand{\zfic}{x^{\text{fic}}}
\newcommand{\zrestart}{x^{\text{unload}}}
\newcommand{\xp}{x^{\mathcal{\pi}}}
\newcommand{\xnf}{x^{nf}}

\newcommand{\xstatic}{x^{s}}
\newcommand{\xsemi}{x^{d}}
\newcommand{\xfull}{x^{d}}
\newcommand{\mofp}{M^{\pi}}
\newcommand{\pin}{\pi^{nf}}
\newcommand{\pia}{\pi^{a}}
\newcommand{\pis}{\pi^s}
\newcommand{\pid}{\pi^d}
\newcommand{\pif}{\pi^d}

\newcommand{\binstatic}[1]{\mathcal{A}^s(#1)}

\newcommand{\binfull}[1]{\mathcal{A}^d(#1)}
\newcommand{\allo}{\mathcal{A}}

\newcommand{\typeballs}[1]{\theta({#1})}

\newcommand{\policycst}{a_{p}}
\newcommand{\staticcst}{a_{{s}}}
\newcommand{\semicst}{a_{{d}}}

\newcommand{\Tstar}{T^\star}
\newcommand{\Ta}{T_a}
\newcommand{\R}{R}
\newcommand{\Rn}{R^{nf}}
\newcommand{\Ra}{R^{a}}

\newcommand{\Rd}{R^{d}}

\newcommand{\thetan}{\theta_{nf}}
\newcommand{\thetaa}{\theta_{a}}

\newcommand{\thetad}{\theta_{d}}

\newcommand{\cstsemi}{c_d}
\newcommand{\cstaf}{c_{a}}

\newcommand{\Tballs}{T}

\newcommand{\Nballs}{N}
\newcommand{\Ninv}{N}
\newcommand{\qballs}{q}

\newcommand{\rballs}{r}

\newcommand{\Sinv}{S}

\newcommand{\stateinv}{z}

\newcommand{\historyballs}[1]{\sigma(#1)}
\newcommand{\historysetballs}[1]{\Sigma(#1)}
\newcommand{\history}[1]{\sigma(#1)}
\newcommand{\historypath}[1]{\sigma'(#1)}
\newcommand{\historyset}[1]{\Sigma(#1)}

\newcommand{\preferredballs}[1]{P(#1)}
\newcommand{\preferred}[1]{P(#1)}

\newcommand{\flexsetballs}[1]{\mathcal{F}(#1)}
\newcommand{\flexset}[1]{\mathcal{F}(#1)}
\newcommand{\flex}[1]{f(#1)}

\newcommand{\flexball}[1]{f(#1)}
\newcommand{\flextype}[1]{f(#1)}

\newcommand{\flexactionballs}{\omega}
\newcommand{\flexaction}{\omega}

\newcommand{\EE}[1]{\mathbb{E}\left[#1\right]}
\newcommand{\PP}[1]{\mathbb{P}\left(#1\right)}
\newcommand{\That}{\widehat{T}}
\newcommand{\Ttilde}{\widetilde{T}}

\newcommand{\parenthesis}[1]{\left(#1\right)}
\newcommand{\bracket}[1]{\left\{#1\right\}}
\newcommand{\squarebracket}[1]{\left[#1\right]}

\newcommand{\rev}{\mathcal{R}}

\usepackage[capitalize]{cleveref}

%% file: abstract.tex
Effective load balancing lies at the heart of many applications in operations. Frequently tackled via the balls-into-bins paradigm, seminal results established the power of two choices in load balancing: a limited amount of costly flexibility goes a long way in order to maintain an approximately balanced load throughout the decision-making horizon. In many applications, however, balance across time may be too stringent a requirement; rather, the only desideratum is approximate balance at the {\it end} of the horizon. Motivated by this observation, in this work we design ``delayed-flexibility'' algorithms tailored to such settings. For the canonical balls-into-bins problem, we show that a simple policy that begins exerting flexibility toward the end of the time horizon  --- namely, when $\Theta\left(\sqrt{T\log T}\right)$ periods remain --- suffices to achieve an approximately balanced load, i.e., a maximum load within $\mathcal{O}(1)$ of the average load. Moreover, with just a small amount of adaptivity, a threshold policy achieves the same result, while only exerting flexibility in $\mathcal{O}\left(\sqrt{T}\right)$ periods, thus matching a natural lower bound. We leverage these results to study the design of opaque selling strategies in retail settings, a topic recently identified as a key application of the power of two choices paradigm. For this problem, we prove that late-stage opaque selling strategies achieve the optimal trade-off between exerting costly flexibility --- i.e., offering the opaque product at a discount --- and achieving inventory cost savings through load balancing. We demonstrate the robustness of our insights via extensive numerical experiments, for a variety of customer choice models.

%% file: intro.tex
\section{Introduction}\label{sec:intro}

A key question in operations is how to effectively address supply-demand imbalances in dynamic settings. When a decision-maker has access to different resources with which to satisfy demand, this question is often tackled through the lens of \emph{load balancing}.
The canonical model of load balancing is the balls-into-bins paradigm, in which balls (demand) are sequentially placed into bins (resources) according to an allocation scheme  \citep{richa2001power}.
This model is used to understand how a decision-maker can maintain an approximately balanced load across bins --- i.e., an approximately equal number of balls in each bin --- over time. Seminal studies of this topic have established the {\it power of two choices} for balls-into-bins: \emph{a limited amount of flexibility goes a long way toward load balancing}. Specifically, an algorithm that places each ball in the minimally loaded of two randomly chosen bins in each period ensures that the gap of the system, measured by the difference between the maximum and average load across all bins, is independent of the number of balls thrown \citep{azar1994balanced,mitzenmacher1996power,peres2010}. The power of these simple algorithms has been far-reaching, particularly in computing settings where (i) the decision-maker seeks to balance the load {\it throughout} time in order to minimize metrics such as average delay, and (ii) querying the load of all bins in each period can be costly.

{This work is motivated by {three} important observations that apply to a number of load balancing settings. First, maintaining a balanced load across all resources over {\it all time} may be an unnecessarily stringent requirement. Instead, balance may only be required at the {\it end} of the decision-making horizon. Second, in many settings it is not the act of querying the loads of bins that is costly; rather, it is the act of placing a ball into a bin for which it was not destined originally. Another way to view this is that each arriving ball has a {preferred} bin, and exerting flexibility by placing it in a lesser-preferred bin is costly to the decision-maker. {Lastly}, in many applications, the cost of this flexibility is approximately homogeneous across bins, as we discuss at the end of Section \ref{sec:prelim}. Examples where this is the case include the following:

\smallskip 

{
\paragraph{Online retail warehouse operations.} In many e-commerce settings, the platform operator (e.g., Amazon) offers fulfillment services to third-party sellers \citep{amazonFBA}. Sellers bear the cost of shipping their goods (balls) into one of the warehouses operated by the platform (bins), and therefore prefer to send their shipments to the closest warehouse. The platform operator, on the other hand, maintains end-of-week shipment volume targets for each warehouse. {These fixed targets, based on pre-determined processing capacities and expected arrival volumes, represent the ideal fraction of shipments each warehouse should receive to balance the load across the network}  \citep{hssaine2024target}. While these targets need not be respected throughout the week, it is important that they be respected approximately at the end of each week. If the load across warehouses becomes so imbalanced that the operator may not achieve its targets, the platform operator can incentivize sellers to send their shipments to a different warehouse {within a small radius} by compensating them for the difference in shipping costs.

\smallskip 

\paragraph{Delivery windows in e-commerce fulfillment.} Consider a setting in which customers living in the same zip code place orders on an e-commerce website. Each customer (ball) has a preferred delivery window (bin), but the e-commerce company {can achieve greater efficiency by having a similar number of orders associated with each delivery window.} If too many customers choose the same delivery window, the e-commerce company may assign customers placing orders closer to the delivery date to a lesser-loaded delivery window that is within a few hours of their preferred window. However, the company wants to do this as infrequently as possible, as it represents a costly inconvenience to the customer. 

\smallskip

\paragraph{Dynamic workforce scheduling.} In settings with flexible workforces (e.g., volunteers in nonprofit organizations \citep{escallon2025sharing}, gig economy workers in online retail fulfillment \citep{amazonFlexibleJobs}), workers (balls) may begin signing up for their preferred shifts (bins) well in advance of the shift date. As in the delivery window example, these organizations prefer to have a balanced supply of workers across shifts. Closer to the shift date, the organization may start to assign shifts to workers (or close off over-subscribed shifts). This{, however,} comes at the cost of a worse experience for flexible workers, who place a high value on the ability to choose their schedule. 

\smallskip 

\paragraph{Inventory management via opaque selling.} Consider a retailer that sells a large number of a few different products (bins), and jointly restocks them all once the stock of any one product is depleted by customers (balls). For inventory costs to be minimized, the retailer wants \emph{all} items to be close to depletion at the time restocking occurs; however, imbalances in remaining inventory do not affect the retailer's supply costs at other points in time. In practice, to address these inventory imbalances, retailers may offer opaque products to customers \citep{elmachtoub2015retailing, elmachtoub2019value}. These products give the retailer the right to choose which item is allocated to the customer (thereby giving the retailer more control over her inventory levels), at the cost of a discount. \Cref{fig:opaque_example} presents two examples of this practice. For instance, \Cref{fig:opaque-stapler} shows an example in which a seller on Amazon.com offers a randomly-colored stapler at a discount relative to staplers of a specific color. Similarly, \Cref{fig:opaque-socks} shows a retailer selling a ``Grab Bag'' of calf sleeves of unknown color for \$45, compared to calf sleeves of a specific color selling for \$75. Common to both examples is that the retailer offers a subset of products of similar quality and decides which specific item from this subset to allocate to the customer after the purchase has been made.

\begin{figure}[t]
    \centering
    \subfloat[\centering Opaque stapler sold on Amazon.com]{{\includegraphics[width=0.50\textwidth]{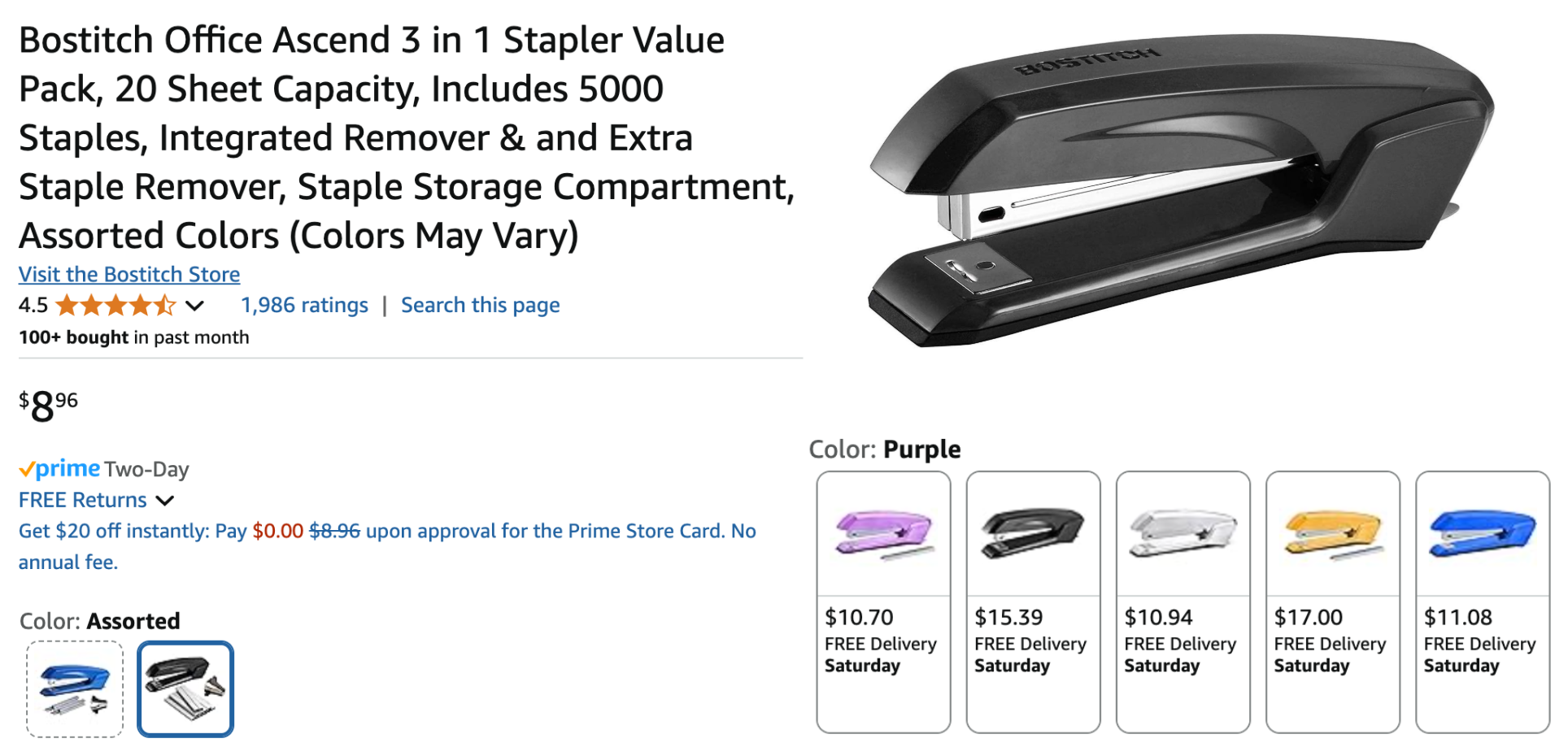} }\label{fig:opaque-stapler}}%
    \hspace{0cm}
    \subfloat[\centering Opaque calf sleeves sold by PRO Compression]{{\includegraphics[width=0.47\textwidth]{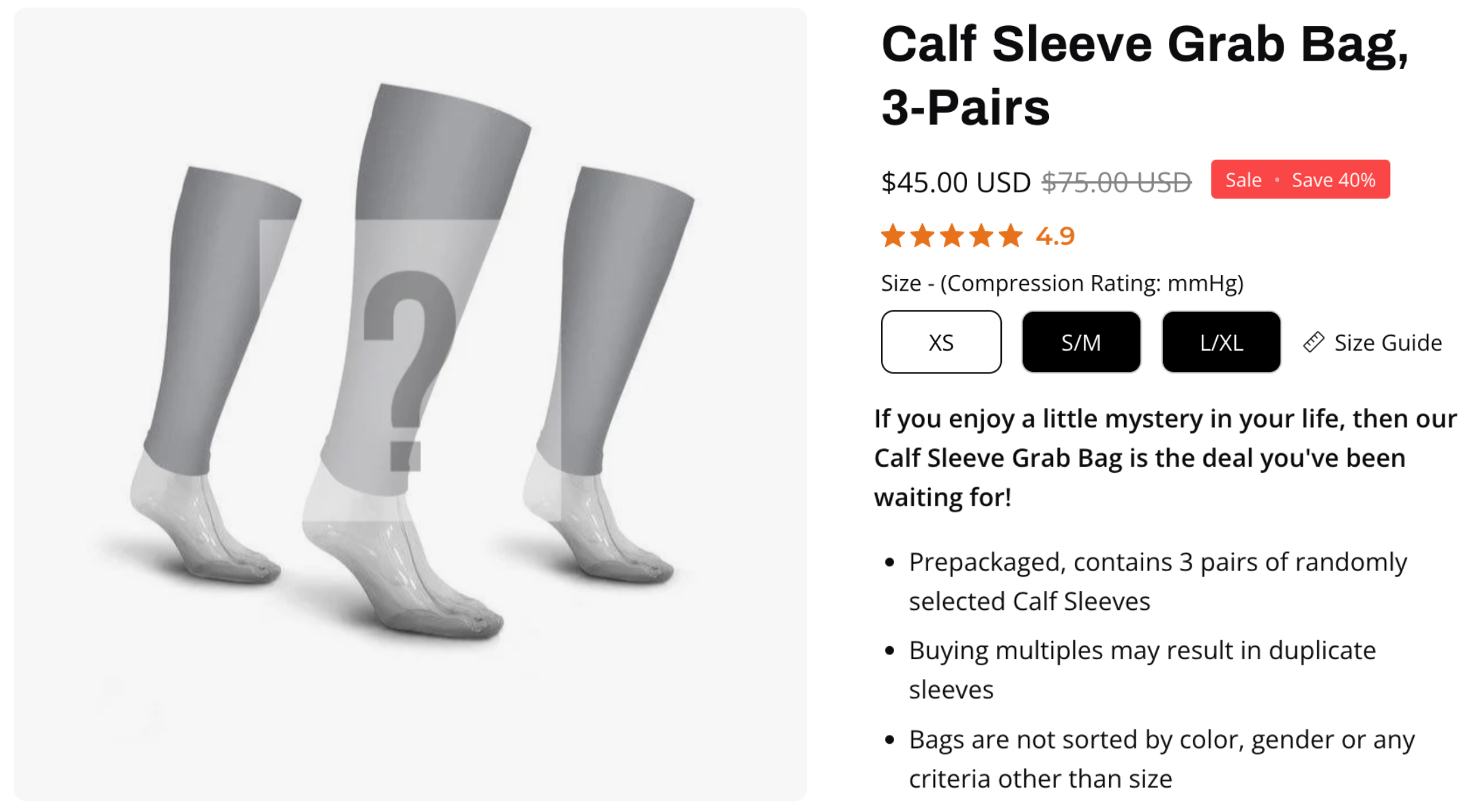} }\label{fig:opaque-socks}}%
    \caption{\centering Examples of opaque selling in retailing platforms}
    \label{fig:opaque_example}
\end{figure}

Within the context of the food industry, restaurants have partnered with online applications such as Too Good To Go to reduce food waste in recent years  \citep{toogoodtogo}. While restaurants can list any individual food item (bin) for purchase by a customer (ball) on these apps, at the end of the day they begin offering heavily discounted surprise bags filled with surplus items, the content of which the customer has no control over. These surprise bags allow restaurants to ensure that the remaining stock of food across all products is even by the end of the day, thereby avoiding excess of any one food item; this however comes at the cost of a steep discount. 
}

\smallskip

{While existing ``power of two choices'' load balancing algorithms would achieve balance across resources by the end of the horizon in all of these settings, the fact that these algorithms attempt to allocate a demand unit to a lesser-preferred resource in each period comes at a significant cost to the decision-maker. Moreover, this cost may be incurred unnecessarily, since the decision-maker does not require balance {throughout} the horizon in these settings. 
Against this backdrop, our work tackles the following questions:
\smallskip
\begin{center}
\emph{What is the minimum amount of flexibility a decision-maker needs to exert to achieve end-of-horizon balance? Do there exist simple load balancing algorithms that achieve this lower bound?}
\end{center}
}

\subsection{Our Contributions}

{
\subsubsection*{Late-stage flexibility in the canonical balls-into-bins model.} Toward answering this question, we first study the power of late-stage flexibility in the classical balls-into-bins model. We consider a discrete decision-making horizon of length $T$. In each period, a ball arrives; this ball has a preferred bin to which it wants to be allocated, drawn uniformly at random from $N$ available bins. The ball is moreover {\it flexible} with some known probability. If the ball is flexible, the decision-maker may draw a subset of two bins uniformly at random and allocate the ball to the lesser-loaded of these two bins; we say that she {\it exerts flexibility} in this case. The decision-maker's goal is to achieve {\it approximate end-of-horizon balance} --- i.e., ensure that the expected gap between the maximum and average loads across all bins is a constant independent of $T$ --- in the minimum number of flexible throws.

We first establish an intuitive lower bound of $\Omega(\sqrt{T})$ on the minimum number of flexible actions required to achieve approximate balance in expectation (\Cref{thm:ball_lb}); this is a natural consequence of the Central Limit Theorem. In our first main contribution, we then design a non-adaptive policy that starts exerting flexibility when $\Theta\left(\sqrt{T \log T}\right)$ periods remain in the horizon. We prove that such a policy achieves approximate end-of-horizon balance, thereby achieving the lower bound up to a factor of $\Theta(\sqrt{\log T})$ (\Cref{thm:ball_static}). The analysis of this policy is tight: we show that no policy that starts exerting flexibility at a deterministic point in time can close this gap to $\Theta(\sqrt{T})$ (\Cref{prop:static_tight}). Motivated by this fact, we propose a {\it semi-dynamic}, ``point-of-no-return'' style policy that begins exerting flexibility the first time the gap between the maximum and average loads across bins exceeds a carefully designed time-varying threshold that intuitively represents the expected number of opportunities left to rectify the current imbalance of the system. If the gap exceeds this threshold, this indicates that approximate end-of-horizon balance may not be achieved unless flexibility is exerted immediately, thereby justifying the use of a costly diversion. In our second main contribution, we demonstrate that this slightly more adaptive policy closes the gap of the non-adaptive policy: it achieves approximate balance in $\mathcal{O}(\sqrt{T})$ flex throws, in expectation (\Cref{thm:ball_semi2,thm:ball_semi}).
}

{Our analysis of these algorithms is based on the following intuition:} over the course of the entire time horizon, if the decision-maker never exerted flexibility, the gap between the maximum and average loads across bins would scale as $\Theta\left(\sqrt{T}\right)$. If each time the decision-maker exerted flexibility that gap was reduced by $1$, then she would only need to do so $\mathcal{O}\left(\sqrt{T}\right)$ times in order to achieve approximate balance. {The main issue with this intuition as a formal argument, however, is that} though exerting flexibility always reduces the {\it instantaneous} gap, it does not always reduce the gap as measured {\it in hindsight}. To see this, consider the setting where $N = 2$. Consider moreover a sample path over which the decision-maker decided to exert flexibility early on in the horizon by diverting a ball to bin 1, away from bin 2. If, later on in the horizon, more balls landed in bin 1 than in bin 2, exerting flexibility in this early period would have actually {\it increased}, rather than decreased, the gap at the end of the horizon. Proving that these ``mistakes'' are not too costly in hindsight is one of the main technical challenges {we} overcome in the analysis of our two algorithms. 
This difficulty is further compounded in the analysis of our semi-dynamic policy, under which flexing begins at a {\it random} time. The key then is to show that the threshold condition is constructed carefully enough that the now-random number of rounds remaining suffices to control the accumulated imbalance across bins.

\subsubsection*{Application to inventory management via opaque selling.} We next leverage the algorithmic insights derived for the canonical balls-into-bins model to address the more complex problem of designing late-stage opaque selling strategies for inventory management.\footnote{The connection between opaque selling and balls-into-bins was first made in \citet{elmachtoub2019value}.} {In the model we consider, a retailer sells $N$ horizontally differentiated products; customers are risk-neutral, and behave according to the Salop circle model \citep{salop1979monopolistic}. The retailer sets a single price for each ``traditional'' product, sold individually; she may also offer the opaque product at a discount. In addition to the revenue generated from sales, the retailer incurs holding and replenishment costs associated with her inventory decisions.

While opaque selling strategies have been an active area of research in recent years, their appeal has been rooted in their ability to act as a market segmentation strategy, thereby generating higher revenues under certain behavioral models \citep{fay2008probabilistic}. Our main insight here is that there exist a variety of horizontally differentiated settings for which there exists {\it no} revenue benefit from offering an opaque product; rather, the {\it only} benefit of opaque selling is in inventory cost savings due to the additional control the retailer has over which products she allocates to customers. This clear trade-off between the revenue loss due to opaque discounts and the resulting inventory cost savings then motivates the design of late-stage opaque selling policies that achieve a better revenue-inventory cost trade-off than a naive policy that offers the opaque product in each period. We validate this intuition by adapting the semi-dynamic policy {from balls-into-bins} to this setting; in short, our late-stage opaque selling policy begins offering the opaque product the first time the gap between the average and minimum inventory levels exceeds a similarly designed threshold. We show that this policy achieves the optimal revenue-inventory cost trade-off in a large-inventory regime parameterized by a base-stock inventory level $S$: it achieves a long-run average revenue loss of $\mathcal{O}\left(\frac{1}{\sqrt{S}}\right)$ relative to the revenue-optimal policy that {\it never} offers the opaque discount, all the while incurring the minimum possible inventory costs of any opaque selling policy (\Cref{thm:dynamic_objective}). These results moreover imply that, in settings for which offering the opaque product provides no revenue benefit, 
{late-stage opaque (i) yields greater profit than naive opaque selling, and (ii) is beneficial relative to traditional selling when inventory costs are substantial (\Cref{cor:comparison}).} From a technical perspective, the challenge presented by this setting is that the effective ``end-of-horizon'' in this case corresponds to the first time that a product's inventory is depleted, which is random and endogenous to the opaque selling policy. The introduction of this moving target requires tighter probabilistic bounds on the gap of the system, as compared to the bounds on the gap in the balls-into-bins model, which are only required to hold in expectation.

{
Finally, we demonstrate the robustness of our theoretical insights via extensive computational experiments. Namely, we study the performance of the semi-dynamic policy under {heterogeneous} customer preferences, and {observe for this more general setting that it usually outperforms both the always-flex policy, which offers the opaque product in each period, and traditional selling.} Our experiments show that these gains are crucially due to the {\it strategic} timing of the opaque offering late in the replenishment cycle (as opposed to simply its {\it infrequent} offering) in order to achieve significant inventory cost savings. We {also} 
uncover the unintuitive fact that, even in settings where opaque selling can {\it increase} revenue due to an increase in customers' purchase probabilities, the semi-dynamic policy can actually achieve {\it higher} profits than the always-flex policy. 
This phenomenon is due to the fact that the retailer must hold additional inventory to meet the increased demand, which results in higher per-period inventory costs in the long run (despite the always-flex policy's ability to balance inventory much more aggressively). {Lastly,} 
we demonstrate the semi-dynamic policy's robust performance to a wide variety of opaque discounts, as well as risk-averse and risk-seeking customer behavior. In short, our theory and experiments yield the important managerial insight that the semi-dynamic policy efficiently interpolates between traditional selling and the always-flex policy, {making} it an attractive strategy for retailers to trade off between revenue and inventory cost considerations.
}
}

\subsubsection*{Paper organization.} We review the related literature in the remainder of this section. We present the canonical balls-into-bins model in \Cref{sec:prelim}, and design and analyze optimal late-stage flexibility algorithms for this model in \Cref{sec:balls}. In \Cref{sec:opaque_salop} we apply our insights from the balls-into-bins model to the opaque selling problem. We complement our theoretical results with extensive numerical experiments in \Cref{sec:numerics}, and conclude in \Cref{sec:conclusion}.
}

{
\subsection{Related Work}
Our work relates to three separate streams of literature: studies of balls-into-bins processes, literature studying the revenue and inventory implications of opaque selling practices, and, more generally, the value of flexibility in operations. We survey the most closely related papers for each of these lines of work below.

\subsubsection*{Balls-into-bins processes.} The balls-into-bins model is one of the most fundamental models in applied probability. This paradigm gained traction early on as being well-suited for a variety of computing applications, such {as} hashing, shared memory emulation, dynamic task assignment to servers, and virtual circuit routing \citep{richa2001power}. We refer the reader to \citet{richa2001power} for an exhaustive survey on this line of work, focusing our discussion on the most closely related technical results.

\citet{raab1998balls} analyze a basic model where $m$ balls are sequentially and randomly assigned to $n$ bins, and derive high-probability upper and lower bounds on the maximum load across bins, showing that the gap between the most-loaded bin and the average load is $\Theta\left(\sqrt{\frac{m\log n}{n}}\right)$. For the special case where \mbox{$m = n$}, \citet{azar1994balanced} and \citet{mitzenmacher1996power} established the celebrated {\it power of two choices}, showing that the load balancing strategy that allocates each ball to the lesser-loaded of two randomly chosen bins yields an exponential decrease in the maximum load across all bins. \citet{peres2010} (whose arrival model we follow) later generalized these results to the case where $m \gg n$; they showed that if a ball is assigned to a random bin with probability $q \in (0,1]$, and to the lesser-loaded of two random bins with probability $1 - q$, the expected gap is a constant independent of $m$ (though dependent on $n$). 

At a high level, the underlying motivation behind the power of two choices framework is the idea that, in many computing applications, load balancing requires costly queries of the load of each bin in each period. Our contribution to this line of work is to push this idea to its limit, by showing that to achieve an approximately balanced load at the end of the horizon, not only does it suffice for a decision-maker to load balance with {\it two} (as opposed to $n$) bins; the decision-maker need only do so a vanishingly small fraction of the time, at the very end of the horizon. In this sense, our work can be viewed as even further evidence of the power of two choices in balls-into-bins processes.

\subsubsection*{On the power of flexibility in opaque selling.} The practice of offering opaque products has been widely studied in the operations literature. From a revenue perspective, opaque selling offers two key benefits: (i) it can boost overall demand, and (ii) it enables better capacity utilization when demand and supply are misaligned~\citep{gallego2004revenue}. Much of the focus on opaque selling has been on its ability to effectively {price discriminate} between customers with different willingnesses-to-pay, thereby generating gains in revenue~\citep{jiang2007price, fay2008probabilistic, jerath2010revenue, zhang2015probabilistic}. The majority of these latter works focus on uncapacitated settings with horizontally differentiated products, with customers behaving according to the Hotelling or Salop circle models, as is the case in our model. \citet{elmachtoub2021power} extend this line of work by identifying conditions under which opaque selling outperforms both discriminatory and uniform pricing for exchangeable valuation distributions. More recently, \citet{housni2025price} study pricing and assortment optimization with opaque products under MNL choice, and \citet{fu2025optimal} design distribution-free mechanisms for this problem. 

Relative to this literature, our contribution is to show that even in settings where opaque selling generates revenue losses, {it} 
remains an appealing practice due to the inventory cost savings resulting from the retailer having the power to choose which good to allocate to the customer. A few prior works have studied this inventory aspect of opaque selling. For instance, \citet{xiao2014evaluating} analyze a model wherein a retailer sells two products over a finite selling period, in addition to an opaque product. They propose a dynamic programming approach to compute the optimal opaque selling decision at each state, and moreover illustrate the positive effects of inventory pooling in their setting. \citet{fay2015timing} study a stylized model wherein the retailer selling two goods needs to time the opaque selling decision before or after observing demand, in addition to deciding how much of each good to order. More recently, \citet{feldmanapproximation} study a retailer's finite-horizon dynamic pricing problem in the presence of opaque products. Less closely related to our work is \citet{ren2022opaque}, who propose a stylized game-theoretic model of a retailer offering two vertically differentiated products over two periods, and using opaque selling as an inventory-clearance strategy in the second period. The goal of the retailer is to set optimal prices and inventory levels when customers strategically time their purchases.

We highlight two works upon which our work builds: \citet{elmachtoub2015retailing} and \citet{elmachtoub2019value}. In the model considered by \citet{elmachtoub2015retailing}, a retailer sells two products over an infinite horizon, with inventory dynamics as in our setting (and, in particular, in a large-market regime parameterized by order-up-to-level $S$). The authors show that when demand is symmetric across both products, the inventory cost savings are on the order of $\Omega\left(\frac{1}{\sqrt{S}}\right)$ as long as the per-period opaque purchase probability $q \in \Omega\left(\frac{1}{\sqrt{S}}\right)$. When demand follows a Hotelling model, they demonstrate the existence of an opaque discount that induces $q \in \Theta\left(\frac{1}{\sqrt{S}}\right)$ and yields a revenue loss on the order of $\mathcal{O}\left(\frac{1}{S}\right)$. Since $N = 2$ is a special case of our setting, our work adds to their result by proposing an {\it adaptive} opaque selling policy with the same guarantees. \citet{elmachtoub2019value} later extended this first work by considering $N \geq 2$ horizontally differentiated products. Focusing only on the inventory management problem in this case, they leverage an elegant connection to the balls-into-bins model to show that opaque selling yields relative cost savings of $\Theta\left(\frac{1}{\sqrt{S}}\right)$ when $q \in \Omega\left(\frac{1}{\sqrt{S^{1-\delta}}}\right)$, for $\delta > 0$. They introduce revenue considerations in numerical experiments, exhibiting a number of demand models under which opaque selling yields strong profit gains relative to traditional selling.\footnote{These findings were later echoed in \citet{zhang2024less}, who extend their results to non-uniform demand across products.} We bridge the theoretical gap between \citet{elmachtoub2015retailing} and \citet{elmachtoub2019value} by showing that even for demand models for which opaque selling yields {\it no} revenue gains, late-stage opaque selling strategies generate the same inventory savings as those that offer the opaque product in each period, all the while losing a limited amount in revenue due to opaque discounts.}

\subsubsection*{General flexible processes.}  Finally, we note that the power of flexibility has been extensively studied in operations applications such as manufacturing \citep{jordan1995principles}, queuing systems \citep{tsitsiklis2013power,tsitsiklis2017flexible}, workforce cross-training \citep{wallace2005staffing}, bin packing \citet{tan2024online}, and warehouse operations \citep{hssaine2024target}. All of these works demonstrate that a small amount of supply-side flexibility suffices to realize most of the benefits of full flexibility. We refer the reader to \citet{wang2021review} for an excellent survey of more recent results on the power of flexibility in operations.

Moreover, our opaque selling contribution relates to broader research on how demand-side flexibility can improve supply utilization and reduce costs. For instance, \citet{zhou2021supply} and \citet{freund2021pricing} show that time-flexible demand improves utilization in scheduled services and ride-hailing, respectively. {Relatedly, in the resource allocation setting, \citet{golrezaei2021online} show the value of demand flexibility when both time-flexible and time-inflexible customers seek a service with periodic replenishments. \citet{zhu2022performance}} study a flexible version of the classical network revenue management problem, where the service provider chooses which combination of resources to allocate to each customer.

%% file: preliminaries.tex
\section{Preliminaries}\label{sec:prelim}

{In this section we present the classical balls-into-bins model \citep{mitzenmacher2001power,peres2010}.} For clarity of exposition, we defer a description of the opaque selling model to \Cref{sec:opaque_salop}.

\subsubsection*{The balls-into-bins process.} The vanilla balls-into-bins model evolves over a discrete, finite-time horizon in which $\Tballs$ balls are sequentially allocated into $\Nballs \geq 2$ bins. At the beginning of each period \mbox{$t \in \{1,2,\ldots,T\}$}, a ball arrives; the ball has a {\it preferred bin}, drawn uniformly at random from $\{1,2,\ldots,N\}$ and denoted by $\preferredballs{t}$. In addition to this, the ball is a {\it flexible}, or {\it flex}, ball with probability $\qballs \in (0,1]$. We let $f(t)$ be the indicator variable denoting whether the ball is flexible, with $f(t) = 1$ if it is, and $f(t) = 0$ otherwise. If the ball is not flexible, the decision-maker places it in its preferred bin. If it is flexible, on the other hand, the decision-maker has more control over the bin to which the ball will be allocated. In particular, the decision-maker may choose to draw a random subset of two bins and may place the ball in any of these two random bins. We refer to this random subset of bins as the {\it flex set}, denoted by $\flexsetballs{t}$, and refer to the act of choosing the bin to which to allocate the ball as {\it exerting flexibility}, or exercising a {\it flex throw}. If the decision-maker chooses not to exert flexibility, the ball is placed in its preferred bin. In this latter case, or if the ball is not a flex ball, we write $\flexsetballs{t} = \phi$. For $t \in \{1,2,\ldots,T\}$ we let $\typeballs{t}{=\left(\flexball{t}, \preferredballs{t}, \flexsetballs{t}\right)}$, and define the {history} of balls at time $t$ to be $\historyballs{t} = (\typeballs{1},\ldots,\typeballs{t-1})$. Finally, let $\historysetballs{t}$ be the set of all possible histories at time~$t$.

\subsubsection*{Objective.} 
Let $\pi$ denote a policy that maps the number of balls in each bin to the decision to exert flexibility, denoted by $\omega^\pi(t)$ (with $\flexactionballs^\pi(t) = 1$ if exercised, and 0 otherwise). For $i \in [N]$, we use $x_i^\pi(t)$ to denote the number of balls in bin $i$ at the end of period $t$ under policy $\pi$, and henceforth refer to this quantity as the {\it load} of bin $i$. If a ball is flexible,  we assume the decision-maker always places it in the bin with the smallest load in its flex set. Formally, letting $\allo^\pi(t)$ denote the bin to which the $t$-th ball is allocated, we have:
$$
\allo^\pi(t) = \begin{cases}
    \arg\min_{j\in\flexsetballs{t}} x_j^\pi(t) &\quad \text{if } \flexball{t}\flexactionballs^\pi(t) = 1 \\
    \preferredballs{t} &\quad \text{if } \flexball{t}\flexactionballs^\pi(t) = 0,
    \end{cases}
$$
where we define the $\arg\min$ with a lexicographic tie-breaking rule that returns the smallest value~$i$ of all bins with the smallest number of balls. Finally, we let $\mofp$ denote the number of times that the decision-maker chooses the bin that a flex ball goes into, i.e., $\mofp = \sum_{t=1}^Tf(t)\omega^\pi(t).$

The decision-maker's goal is to ensure that the load across bins is approximately balanced at the end of the time horizon, all the while minimally exerting flexibility. To formalize this two-fold objective, we define the {\it gap} of the system in period $t$ under $\pi$ {as} the difference between the maximum load across all bins and the average load {after the $t$-th ball arrives}, given by $t/N$. Letting $\Gap^\pi(t)$ denote this gap, we have:
\begin{align*}
    \Gap^{\pi}(t) = \max_{i \in [\Nballs]} x_i^\pi(t) - \frac{t}{\Nballs} \quad \forall \, t \in \{1,2,\ldots, T\}.
\end{align*}
We say that the system is approximately balanced under $\pi$ if $\mathbb{E}\left[\Gapp(T)\right] \in \mathcal{O}(1)$, where the Big-O notation is with respect to the time horizon $\Tballs$. 

In order to formalize the desideratum of achieving approximate balance in a minimal number of flex throws, we present two simple policies that have previously been analyzed in the literature.  Consider first the {\it no-flex} policy, denoted by superscript $nf$, that never exerts flexibility (i.e., $M^{nf} = 0$ almost surely). It is known that under this policy, the load across bins is imbalanced by the end of the horizon, with \mbox{$\mathbb{E}\left[\Gapnf\right] \in \Theta\left(\sqrt{T}\right)$} \citep{raab1998balls}. On the other hand, the {\it always-flex} policy, denoted by superscript $a$, always exerts flexibility, with $\omega^a(t) = 1$ for all $t \in \{1,2,\ldots, T\}$. This policy has been shown to achieve an approximately balanced load at the end of the horizon, with \mbox{$\mathbb{E}\left[\Gapa\right] \in \mathcal{O}(1)$ in $\EE{M^{a}}{=Tq}\in\Theta(T)$} flex throws in expectation \citep{peres2010}. Given these benchmarks, our goal is to design policies that achieve two desiderata: (i) ensure that the load is balanced at the end of the horizon, with $\mathbb{E}[\Gap^\pi(T)] \in \mathcal{O}(1)$, and (ii) doing this with $o(T)$ flexes, in expectation.

{
\subsubsection*{Discussion of modeling assumptions.} We conclude this section with a discussion of the modeling assumptions upon which the classical balls-into-bins model relies. Specifically, the model we consider here is most appropriate for settings {in which} (i) demand for resources is approximately homogeneous, and (ii) the cost of diverting a demand unit from its preferred resource is approximately the same across all lesser-preferred resources (otherwise, the number of flexes $M^\pi$ ceases to be a meaningful metric of costly flexibility). The first assumption is not critical for our results; at the end of \Cref{sec:balls} we discuss how our results easily extend to settings in which demand is non-uniform across resources. It therefore remains to justify the assumption of uniform costly diversion. We highlight instances of the examples provided in \Cref{sec:intro} for which this assumption is well-justified below.

For the warehouse operations example, this model is most appropriate for ``regionalized'' networks that have ex-ante been partitioned into warehouses within the same geographic region \citep{oneill2023amazon}. For these smaller regions, it is sensible to assume that sellers' shipping costs do not vary wildly across warehouses that are at a slightly further distance from their preferred warehouse. For the opaque selling examples, the uniform cost of diversion follows from the practice of offering a discount for the opaque product. Finally, within the context of delivery windows and dynamic workforce scheduling, one can assume that the decision-maker partitions windows/shifts ex-ante (e.g., according to peak and non-peak hours). As in the regionalization example, for all intents and purposes, these partitions form different load balancing systems, and diversions within these systems are similarly inconvenient to customers/workers.

Finally, we note that the trade-off between balance and flexibility we consider in this work can be tackled by formulating an appropriate stochastic optimization problem tailored to each of the aforementioned applications. However, the popularity of the balls-into-bins model lies in its generality, and the fact that it is able to parsimoniously capture this fundamental trade-off. We will see that this parsimony gives rise to intuitive and easy-to-implement algorithms.
}

{
\subsubsection*{Technical notations.} Throughout the paper we use $\text{Ber}(p)$ to denote a Bernoulli random variable with probability $p$, $\text{B}(n, p)$ to denote a Binomial random variable with probability $p$ and $n$ trials, and $U(a,b)$ to denote a uniform random variable with support $[a,b].$ Finally, $\Phi(\cdot)$ denotes the CDF of the standard normal distribution $\mathcal{N}(0,1)$.
}

%% file: bib-results.tex
\section{Optimality of Late-Stage Flexing for Balls-into-Bins}\label{sec:balls}

In this section we tackle the task of designing policies that achieve the two desiderata described above. We first study the question of how many flex throws are {\it required} to achieve approximate balance at the end of the horizon. We leverage this lower bound to design an approximately optimal policy that balances the system up to a factor of $\Theta(\sqrt{\log T})$ of the minimum number of flexes. We then use the analysis of this policy to derive a simple optimal policy that closes this gap.

\input{bib-static}

\subsection{Optimality of a Semi-Dynamic Threshold Policy} \label{app:ball_semi}

Notice that our static policy does not exactly meet the lower bound from \cref{thm:ball_lb} due to the additional $\mathcal{O}\left(\sqrt{\log T}\right)$ factor in the number of flexes. This is not an artifact of our analysis, nor is it due to our definition of~$\staticcst$. Instead, it is a general fact about non-adaptive policies: we show in Appendix \ref{app:static_tight} that no non-adaptive policy can achieve $\mathcal{O}(1)$ gap at $T$ while exerting flexibility $\mathcal{O}\left(\sqrt{T}\right)$ times. At a high level, this follows from the fact that static policies ignore that {\it not all sample paths are created equal}: while the loads under certain sample paths require a larger number of flexes to achieve balance, on others, a smaller number of flexes suffices. This naturally leads us to the idea that, while no {static} policy can achieve an approximately balanced load with $\mathcal{O}\left(\sqrt{T}\right)$ flexes, it may be that an {\it adaptive} policy can.

Thus motivated, we consider an adaptive modification of our static policy, referred to as the {\it semi-dynamic} policy $\pi^d$, which {begins flexing the first time the gap of the system exceeds a pre-specified threshold, and keeps flexing from then onwards}. Specifically, the threshold (also referred to as the {\it flexing}) condition that our policy verifies is given by:
\begin{align}\label{eq:dynamic-threshold}
\Gapd(t) \geq \frac{a_d(T-t)q}{N},
\end{align}
where the superscript $d$ is used for all quantities induced by the semi-dynamic policy, and $a_d \in (0,1]$ is a constant tuning parameter. We let $\Tstar := \inf\left\{t:\Gapf(t) \geq \frac{\semicst(T-t)q}{N}\right\}$ be the period in which the flexing condition is triggered; that is, the first period in which the semi-dynamic policy flexes is $\Tstar + 1$. We provide a complete description of the semi-dynamic policy in \Cref{alg:ball_full}.

\begin{algorithm}[t]
\caption{Semi-Dynamic Policy $\pif$}\label{alg:ball_full}
\begin{algorithmic}[1]
\Require $\xfull_i(0)=0\;\forall \, i \in [N]$, $\flexaction^d(t) = 0\;\forall \ t \in [T]$, constant $\semicst > 0$
 \For{$t \in [T]$}
    \If{$\flexaction^d(t)\flextype{t} = 1$}
    \State Sample flex set $\mathcal{F}(t)$.
    \State Allocate ball to least-loaded bin in $\mathcal{F}(t)$, i.e., set $\binfull{t} = \argmin_{i\in\flexset{t}} \xfull_i(t-1)$.
    \Else \ Allocate ball to preferred bin, i.e., set $\binfull{t} = \preferred{t}$.
    \EndIf
    \For{$i \in [N]$}
    \State Update load of each bin:
    \[\xfull_i(t)= \begin{cases}
    \xfull_i(t-1)+1 &\quad \text{if } i = \binfull{t} \\
    \xfull_i(t-1) &\quad \text{otherwise}.
    \end{cases}\]
    \EndFor
    \If{$\Gapf(t) \geq \frac{\semicst(T-t)q}{N}$}
    \State Exercise the flex option from $t+1$ onwards, i.e., set $\flexaction^d(t') = 1 \; \forall \ t' > t$.
    \EndIf
\EndFor
\end{algorithmic}
\end{algorithm}

The threshold condition specified in \Cref{eq:dynamic-threshold} can be interpreted as verifying whether the system has attained a ``point of no return,'' at each time $t$. Namely, the quantity $\frac{a_d(T-t)q}{N}$ represents, for any bin $i$, a lower bound on the number of remaining opportunities to flex a ball that prefers $i$ into another bin. If the gap of the system exceeds this estimate, since exerting a flex throw in any given period reduces the gap of the system by at most one, the semi-dynamic policy assumes that the only way it can balance the system by the end of the horizon is if it starts load balancing immediately. Having established that it is at a point of no return, our policy flexes from then onwards. \Cref{fig:bib-sample-paths} illustrates the difference between the times at which the static and semi-dynamic policies begin flexing for two sample paths, and provides intuition as to the challenge in analyzing the random time at which the semi-dynamic policy begins flexing. In particular, while we observe in \Cref{fig:bib-sample-2} that it indeed avoids unnecessary flexing over some sample paths, with $\Tstar > \That$ if the gap at $\That$ is low, \Cref{fig:bib-sample-1} also illustrates that it may begin flexing {\it well before} $\That$. The main concern would then be that the threshold condition in \Cref{eq:dynamic-threshold} was chosen too conservatively, resulting in early starts happening more frequently than late starts, as depicted in \Cref{fig:bib-sample-paths}. \Cref{thm:ball_semi2} below establishes that our threshold condition in fact guarantees that late starts occur frequently enough, and moreover are within $\mathcal{O}(\sqrt{T})$ of the end of the horizon, in expectation.

\begin{figure}
    \centering
    \subfloat[\centering Early start: $\Tstar < \That$]{{\includegraphics[width=0.45\textwidth]{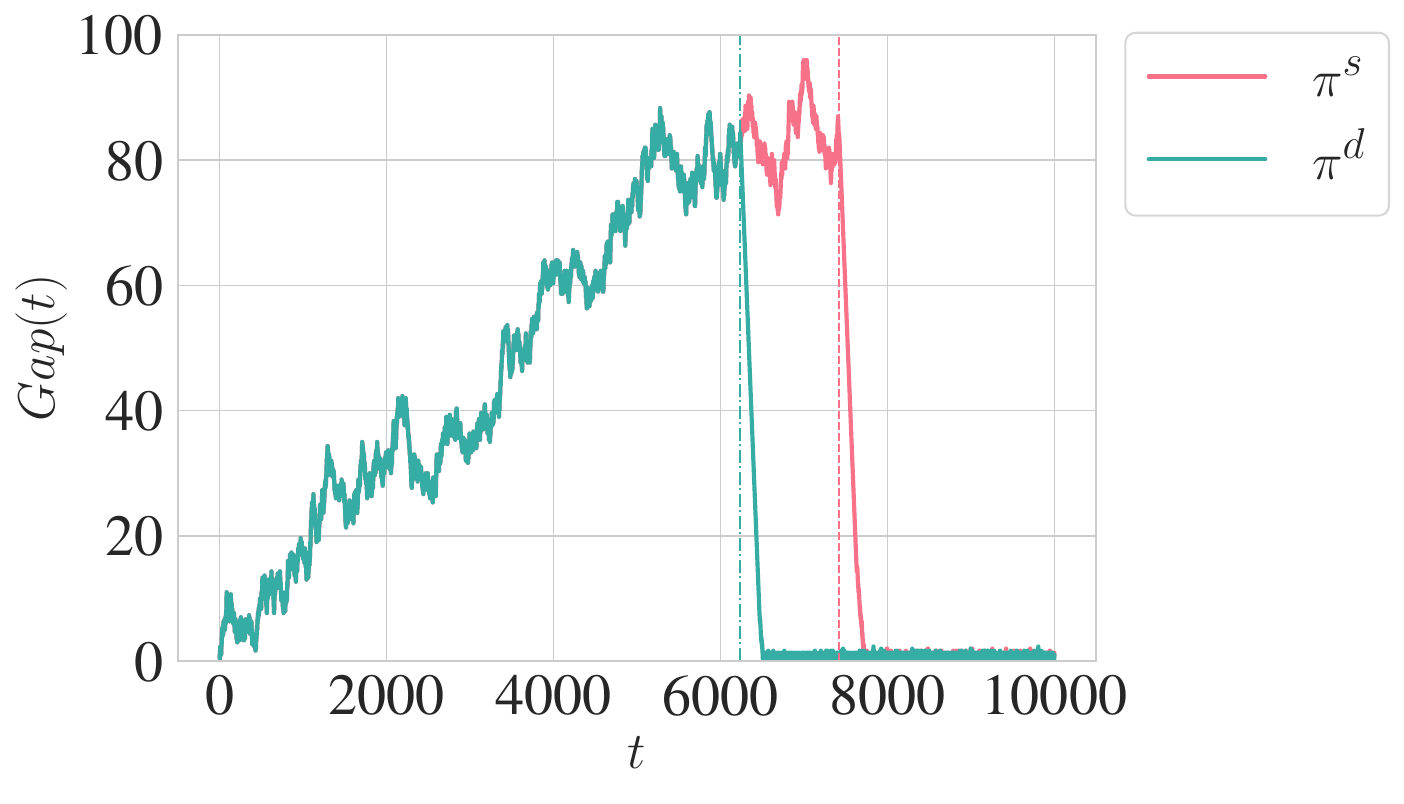}} \label{fig:bib-sample-1}}%
    \hfill 
    \subfloat[\centering Late start: $\Tstar > \That$]{{\includegraphics[width=0.45\textwidth]{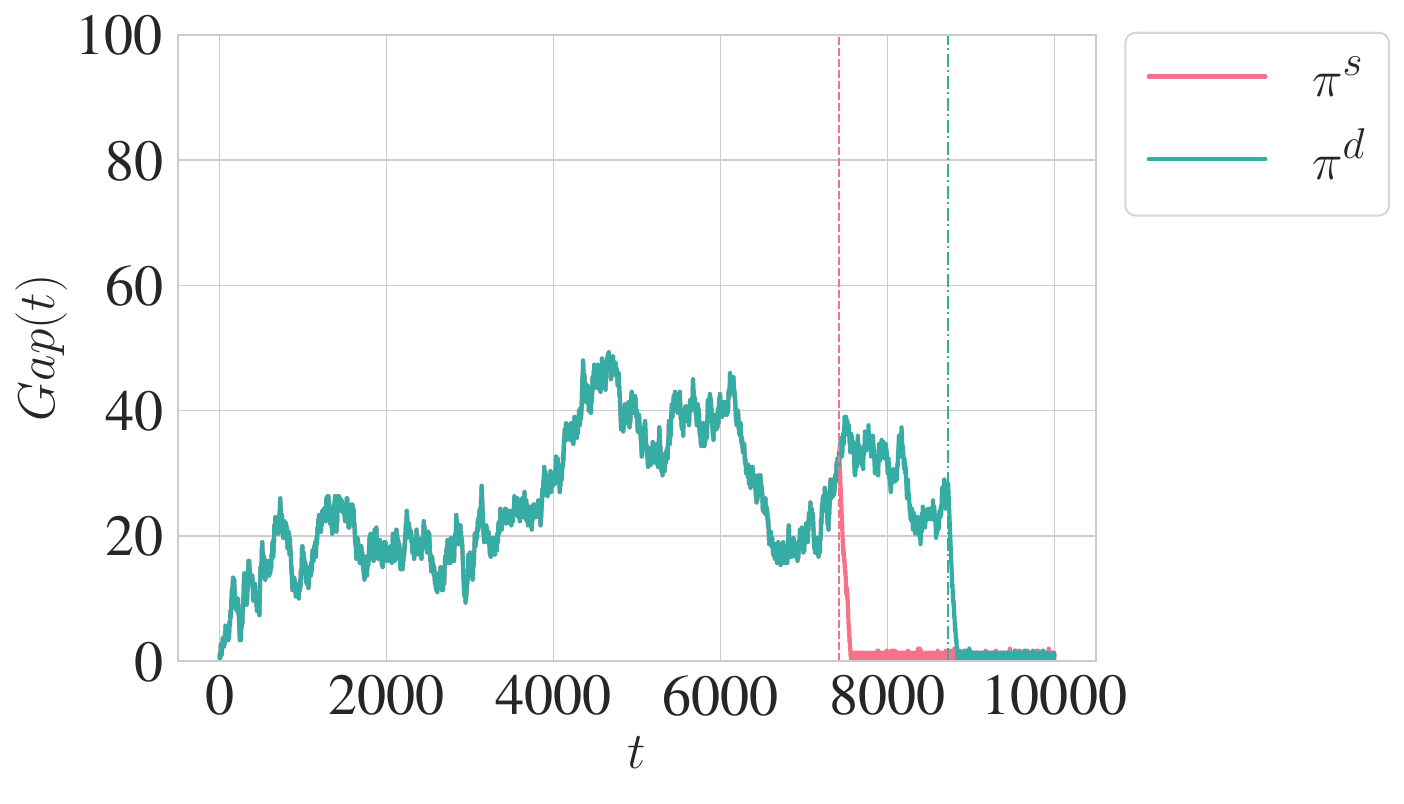} }\label{fig:bib-sample-2}}%
    \caption{\centering $Gap(t)$ versus $t$ under the static and semi-dynamic policies, for $N = 3, q = 1$, $a_s = {2\sqrt{2}{N\choose 2}}$, $a_d = \frac{1}{5{N\choose 2}}$.}
    \label{fig:bib-sample-paths}
\end{figure}

\begin{theorem} \label{thm:ball_semi2}
For any constant $a_d \in (0,1]$, the semi-dynamic policy flexes $\mathcal{O}(\sqrt{T})$ times in expectation, i.e., $\EE{M^d} \in \mathcal{O}\left(\sqrt{T}\right).$
\end{theorem}

To prove \Cref{thm:ball_semi2}, we use the threshold condition to show that bounding $\mathbb{E}[T-\Tstar]$ reduces to obtaining tight bounds on the gap of the no-flex policy at any time $t$. We provide a proof sketch of our approach below, deferring a complete proof to Appendix \ref{app:ball_proof_semi_2}. 

{
\paragraph{Proof sketch.}
Algebraic arguments yield the following useful bound on $\mathbb{E}[T-\Tstar]$:
\begin{align}
\mathbb{E}[T-\Tstar] \leq \sqrt{T}\sum_{k=1}^{\lceil\sqrt{T}\rceil}\mathbb{P}(T-\Tstar \geq (k-1)\sqrt{T}).
\end{align}
Namely, it suffices to show that the probability $T-\Tstar$ exceeds $(k-1)\sqrt{T}$, for some $k \in \{1,\ldots,\lceil\sqrt{T}\rceil\}$, is upper bounded by a constant.

Since the threshold condition is satisfied at $\Tstar$, the event \mbox{$T-\Tstar \geq (k-1)\sqrt{T}$} implies \mbox{$\Gapd(\Tstar) \geq \frac{a_d(k-1)\sqrt{T}q}{N} \geq \frac{a_d(k-1)\sqrt{\Tstar}q}{N}$}. Noting that \mbox{$\Gap^d(t) = \Gap^{nf}(t)$} for all $t \leq \Tstar$, since the policy only begins flexing after $\Tstar$ by construction, it then suffices to bound $\mathbb{P}\left(\Gap^{nf}(\Tstar) \geq c(k-1)\sqrt{\Tstar}\right)$, where $c = a_dq/N$. The result then follows by applying the Berry-Esseen theorem to bound the tail of the load of each bin under the no-flex policy by a standard normal distribution, for all $k \in \{1,\ldots,\lceil{\sqrt{T}}\rceil\}$. Summing over all $k$ completes the proof of the result.
\hfill\Halmos
}

\smallskip 

It remains to argue that beginning to exert flexibility only when absolutely necessary suffices to recoup the imbalance accumulated before $\Tstar$. \Cref{thm:ball_semi} establishes this fact below.
\begin{theorem} \label{thm:ball_semi}
For any $a_d \leq \frac{1}{5 {N\choose2}}$, the semi-dynamic policy achieves $\mathbb{E}\left[\Gapf(T)\right] \in \mathcal{O}(1).$
\end{theorem}

We defer the formal proof {of \cref{thm:ball_semi} to Appendix \ref{app:ball_proof_semi_1}, and provide a proof sketch below}.
{
\paragraph{Proof sketch.} 
The proof follows similar lines as that of \Cref{thm:ball_static}. Namely, we analyze the gap under two events: (i) the event $E^1$, under which the maximally and minimally loaded bins in period $T$, respectively denoted by $i$ and $j$, never had the same load between $\Tstar$ and $T$, and (ii) the event $E^2$, under which these two bins had the same load for some $\tau \in \{\Tstar,\ldots,T-1\}$. We then have the following bound on $\mathbb{E}[\Gapd(T)]$:
\begin{align*}
\mathbb{E}\left[\Gapd(T)\right] \leq \mathbb{E}\left[\Gapd(T) \mid E^1\right]\mathbb{P}(E^1) + \mathbb{E}\left[T-\tau\mid E^2\right]\mathbb{P}(E^2).
\end{align*}
While our analysis of the second term is identical to that of the static policy, with the likelihood that $E^2$ occurs being exponentially decreasing in $T-\tau$, the main challenge in this setting is bounding the first term, since the time at which the semi-dynamic policy begins flexing is random, with no explicit characterization. We tackle this obstacle by leveraging the threshold condition to first upper bound $\Gapd(T)$. In particular, we use the fact that $\Gapd(T) \leq \Gapd(\Tstar-1) + 1 + T-\Tstar$, since no more than $T-\Tstar+1$ balls can go into any bin between $\Tstar$ and $T$. Moreover, it must be that \mbox{$\Gapd(\Tstar-1) < \frac{a_d(T-\Tstar+1)q}{N}$}, by definition of $\Tstar$. Putting these two facts together, we have that $\Gapd(T)\leq \left(1 + T-\Tstar\right)\left(1+\frac{a_dq}{N}\right)$. Not only is this fact useful to bound the expected gap at time $T$, given $E^1$, but it will help us bound the likelihood that $E^1$ occurs. 

To bound this latter probability, we use the threshold condition to observe that $E^1$ occurs if the total number of balls that land in bin $j$ during the flexing horizon cannot correct an already accumulated imbalance of $a_dq(T-\Tstar)+a$ with respect to bin $i$, for some constant $a > 0$. We similarly bound this via two ``bad'' events: (a) that the number of random balls thrown into bin $i$ during the flexing horizon exceeds those thrown into bin $j$ by $a_dq(T-\Tstar)+a$, and (b) that the number of flexible balls thrown into bin $j$ exceeds those thrown into bin $i$ by no more than $2a_dq(T-\Tstar)+a$. 

As before, the analysis of the random throws is a simple application of Hoeffding's inequality. To analyze the flexible throws, we rely on the useful ``lazy'' allocation rule from the proof of \Cref{thm:ball_static}, and show that there exists a constant $t_0 > 0$ for which the probability that event (b) occurs is exponentially decreasing in $T-\Tstar$, as long as $T-\Tstar \geq t_0$, for small enough $a_d > 0$. For values of $\Tstar$ such that $T-\Tstar < t_0$, $\Gapd(T)$ is trivially upper bounded by a constant, thereby completing the proof of the theorem.
\hfill\Halmos
}

{
\begin{remark}\label{rem:non-uniform-bins}
In order to illustrate the fundamental power of late-stage flexibility, for simplicity we focus on the setting where the preferred bin of each arriving ball is drawn uniformly at random (i.e., with probability $1/N$). This assumption is most appropriate for settings where resource preferences or requirements are homogeneous. Our algorithms and analyses easily extend however to the heterogeneous setting, wherein each random ball is placed into bin $i \in [N]$ with constant probability $p_i \in (0,1)$. In this latter setting, the appropriate notion of balance is one in which each bin $i$ receives exactly $Tp_i$ balls by the end of the horizon.\footnote{This models the practical reality that, in settings where certain resources are known to be preferred over others, their capacities (e.g., staffing levels, in the workforce / delivery scheduling examples; initial inventory levels in the opaque selling example) are scaled according to the demand that they are expected to receive.}
Then, to achieve approximate end-of-horizon balance in this setting, one can define the {\it normalized} load of a bin under any policy $\pi$ to be $\bar{x}_i^\pi(t) = \tfrac{x_i^\pi(t)}{{N p_i}}$, and the gap of the system under this policy to be \mbox{$\overline{\Gap}^\pi(t) = \max_{i \in [N]}\bar{x}_i^\pi(t)-\frac{t}{N}$}. (Note that by letting $p_i = 1/N$ for all $i \in [N]$, we exactly recover that $\bar{x}_i^\pi(t) = x_i^\pi(t)$ and $\overline{\Gap}^\pi(t) = \Gap^\pi(t)$  in the homogeneous setting.) Our semi-dynamic policy can then be adapted by placing any flexible ball into the bin with the smallest {\it normalized} load, and letting the threshold condition be $\overline{\Gap}^\pi(t)\geq \frac{\bar{a}_d(T-t)q}{N}$ for an appropriately defined tuning parameter $\bar{a}_d$ depending on $(p_1,\ldots,p_N)$. Leveraging our analysis for the homogeneous setting, it is easy to show that for heterogeneous $p_i$ all results would continue to hold. We omit this analysis for brevity.
\end{remark}
}

%% file: bib-static.tex
\subsection{Warm-Up: Approximate Optimality of Non-Adaptive Late-Stage Flexing}\label{eq:bib-lb}

\cref{thm:ball_lb} below establishes that, in order to achieve an approximately balanced load, a policy must exert flexibility $\Omega(\sqrt{T})$ times in expectation.
\begin{proposition}\label{thm:ball_lb}
Consider any policy $\pi$ such that $\mathbb{E}\left[\Gap^\pi(T)\right] \in \mathcal{O}(1)$. Then, $\mathbb{E}\left[\mofp\right] \in \Omega\left(\sqrt{T}\right).$
\end{proposition}
{This lower bound is quite intuitive: every time the decision-maker exerts flexibility, the gap at time $T$ decreases by at most 1. Since the expected gap when balls are randomly thrown into the bins is well-known to scale as $\Theta\left(\sqrt{T}\right)$, closing this gap to $\mathcal{O}(1)$ would require the decision-maker to exert flexibility at least $\Theta\left(\sqrt{T}\right)$ times. We defer a formal proof of this fact to Appendix \ref{app:ball_lb}.}

We next design two policies that strive to achieve this lower bound. By the above intuition, to close the gap at time $T$ to $\mathcal{O}(1)$ it should suffice to begin exercising the flex option with~$\widetilde{\mathcal{O}}\left(\sqrt{T}\right)$ periods remaining in the horizon. Thus motivated, the first policy we consider, referred to as the {\it static} policy $\pis$, is non-adaptive, and starts {exerting flexibility when} $\Theta\left(\sqrt{T\log T}\right)$ periods remain. 
{Specifically, $\pis$ fixes a time \mbox{$\That = \lfloor T- \staticcst\sqrt{T\log T}\rfloor$}, where 
$\staticcst > 0$ is a constant tuning parameter. This simple policy exerts flexibility for all $t \geq \That$ (i.e., it places flex balls in the minimally loaded bin within the flex set), but not before. Throughout the section, we use superscript $s$ to refer to all quantities induced by this static policy.

\Cref{thm:ball_static} establishes that the intuition underlying the design of this naive policy is correct: exerting flexibility $\Theta(\sqrt{T\log T})$ times suffices to achieve a balanced load by the end of the horizon.
\begin{theorem} \label{thm:ball_static}
For any constant $\staticcst \geq \frac{2\sqrt{2}{N\choose 2}}{q}$, the static policy $\pi^s$ achieves $\mathbb{E}\left[\Gaps(T)\right] \in \mathcal{O}\left(1\right).$
\end{theorem}

{We will build on the analysis of the static policy to analyze the optimal semi-dynamic policy in \Cref{app:ball_semi}. Hence, we provide a detailed proof sketch below, deferring a complete proof of \Cref{thm:ball_static} to Appendix \ref{app:ball_proof_static}.
\paragraph{Proof sketch.}
 {The sample paths where $\Gaps(T) > 0$ can be partitioned into two events:} 
(i) the event $E^1$, under which the maximally and minimally loaded bins in period $T$ never had the same load between $\widehat{T}$ and $T$, and (ii) the event $E^2$, under which these two bins had the same load {at some ``intersection time''} $\tau \in \{\widehat{T},\ldots,T-1\}$. Noting that the gap of the system can trivially be bounded by $T$ under event $E^1$, and by $T-\tau$ under event $E^2$ (which would occur if all $T-\tau$ balls went into the maximally loaded bin between $\tau+1$ and $T$), we obtain the following bound on $\mathbb{E}[\Gap^s(T)]$:
\begin{align}\label{eq:main-decomp}
\mathbb{E}[\Gap^s(T)] \leq T\mathbb{P}(E^1) + \mathbb{E}\left[T-\tau \mid E^2\right]\mathbb{P}(E^2).
\end{align}

The key steps of the proof then lie in showing that (i) $\mathbb{P}(E^1) \in \mathcal{O}\left(\That^{-1}\right)+e^{-\Omega(T-\That)}$, and (ii) given {the last intersection time $\tau$}, $\mathbb{P}(E^2 \mid \tau) \in e^{-\Omega(T-\tau)}.
$
To prove each of these two facts, we let $i$ and $j$ respectively denote the maximally and minimally loaded bins at time $T$. 

By definition, $x_i^s(t) > x_j^s(t)$ for all $t \in \{\That, \ldots, T\}$ under event $E^1$. Intuitively, $E^1$ would occur if the flex balls thrown into bin $j$ between $\That+1$ and $T$ do not sufficiently correct the imbalances created by the random balls thrown into bin $i$ throughout the entire time horizon. We formalize this intuition via a union bound over two bad events: (a) that the number of random balls thrown into bin $i$ over all $T$ periods exceeds those thrown into bin $j$ by $\frac{q}{2{N\choose 2}} (T-\That)$, where $\frac{q}{{N\choose 2}}(T-\That)$ represents the expected number of times $\{i,j\}$ is chosen as the flex set, and (b) that the number of flex balls allocated to bin $j$ between $\That+1$ and $T$, denoted by $Y_j(T-\That)$, exceeds the number of flex balls allocated to bin $i$ between $\That+1$ and $T$, denoted by $Y_i(T-\That)$, by no more than $\frac{q}{2{N\choose 2}}(T-\That)$.

Noting that the number of random balls allocated to bins $i$ and $j$ are respectively binomially distributed, the bound on event (a) follows from a standard application of Hoeffding's inequality. Using the fact that \mbox{$T-\That \geq a_s\sqrt{T\log T}$}, this gives rise to the $\mathcal{O}(\That^{-1})$ term in (i), for large enough $a_s$. The main challenge lies in bounding event (b), due to the state-dependent nature of our allocation rule. In particular, characterizing the excess of flex balls placed into bin $j$ over bin $i$ requires not only keeping track of the loads of these two bins, but also the loads of all other bins with which $j$ and $i$ could have been included in a flex set. Keeping track of the loads of these other bins is important, given that bin $i$ may have a smaller load than some bin $k \not\in \{i,j\}$ at some point in the flexing horizon, in which case the policy would place a flex ball in bin $i$ and have further contributed to the imbalance between $i$ and $j$.

We overcome this challenge by showing that, under event $E^1$, our policy makes the same decisions as a ``lazy'' allocation rule that is biased toward bin $j$, and is conceptually easier to analyze. For this policy, we show that the difference between the number of flex balls allocated to bin $j$ and those allocated to bin $i$ form a submartingale, since the lazy allocation rule is biased toward bin $j$. This then allows us to leverage Azuma-Hoeffding's inequality to show that the probability that $Y_j(T-\That)-Y_i(T-\That) \leq \frac{q}{2{N\choose 2}}(T-\That)$ is exponentially decreasing in $T-\That$, thereby completing the analysis of $\mathbb{P}(E^1)$.

The exponential bound on $\mathbb{P}(E^2 \mid \tau)$ follows similar lines. The key difference is that, given that the loads of bins $i$ and $j$ intersected in some period $\tau \geq \That$, any imbalance between bins $i$ and $j$ would be caused by the random balls thrown between $\tau+1$ and $T$, as opposed to {\it all} random balls thrown throughout the horizon, as was the case in the analysis of $E^1$. A Hoeffding's bound over this source of imbalance yields a term that is exponentially decreasing in $T-\tau$. We then similarly leverage the ``lazy'' allocation rule to show that the probability that the imbalance caused by incorrect flex throws is linear in $T-\tau$ is also exponentially decreasing in $T-\tau$. Applying the bounds on $\mathbb{P}(E^1)$ and $\mathbb{P}(E^2\mid\tau)$ to \eqref{eq:main-decomp}, we obtain the theorem.
\hfill\Halmos
}

%% file: opaque-2024.tex
\section{Application: Late-Stage Opaque Selling for Inventory Management}\label{sec:opaque_salop}

In this section we extend the insights {from \Cref{sec:balls} to an  inventory control problem.} In particular, recent work leveraged the balls-into-bins paradigm to quantify the inventory benefits of offering opaque products in retail settings \citep{elmachtoub2015retailing,elmachtoub2019value}. When opaque selling also improves a retailer's revenue through price discrimination, always offering an opaque product is clearly an appealing strategy. In this section, however, we show that there exist settings in which opaque selling may instead be {\it detrimental} to a retailer from a revenue perspective. We use this as motivation to build upon our contributions for the balls-into-bins model and design simple opaque selling strategies that yield strong inventory cost savings, all the while losing a small amount in revenue.

\subsection{Basic Setup}\label{ssec:opaque-setup}

\subsubsection*{Choice model.} We consider a retailer offering $\Ninv \geq 2$ horizontally differentiated (also referred to as {\it homogeneous}) product types to customers over an infinite horizon. For instance, these products may differ in color or style but are not inherently quality-differentiated (e.g., differently colored staplers or socks, as in \Cref{fig:opaque_example}). A customer arrives in each period and evaluates the products offered to her by the retailer. Generalizing the Hotelling model considered in \citet{elmachtoub2015retailing}, we assume that customers value all products according to the Salop circle model \citep{salop1979monopolistic}, a standard framework for modeling customer choice amongst horizontally differentiated products. {Specifically, each product $i \in [N-1]$ is defined by a numerical characteristic $x_i = i/N$; we use the convention that $x_N = 0$ for the $N$th product.} Each arriving customer has (i) a deterministic baseline value for obtaining any product, denoted by $\bar{v}$, and (ii) an idiosyncratic preferred product characteristic $X \sim U[0,1]$. Her valuation for product \mbox{$i \in [N]$} is then given by \mbox{$V_i = \bar{v}-\gamma\cdot d(X,x_i)$}, where \mbox{$d(X, x_i) = \min\left\{|X - x_i|, 1 - |X - x_i|\right\}$} denotes the arc distance between the customer's preferred product characteristic and that of product $i$, and $\gamma \in [0,\bar{v}\cdot N]$ is a known constant.\footnote{The upper bound on $\gamma$ ensures that $\overline{v} - \gamma/N \geq 0$, i.e., the products are sufficiently attractive that a customer may be incentivized to purchase at least one non-preferred product. \citet{elmachtoub2015retailing} similarly makes such an assumption on the corresponding Hotelling model, for $N = 2$.} Given that products are homogeneous, we assume the retailer sets a single revenue-maximizing price $\hat{p}$ for each individual product $i\in [N]$; we derive $\hat{p}$ later on in \Cref{lem:purchase_probabilities}. Finally, without loss of generality we assume that the marginal cost of selling each unit is zero.\footnote{This assumption is without loss of generality for our theoretical results, as we will see that customers purchase exactly one product with probability one in this setting. While marginal costs may impact the effectiveness of our policies when the demand lift from the opaque discount also increases production costs, our numerical experiments show that our insights continue to hold even in settings where (i) there is a positive marginal cost and (ii) customers may choose not to purchase any item.}

\begin{figure}[t]
    \centering
    \includegraphics[width=0.35\textwidth]{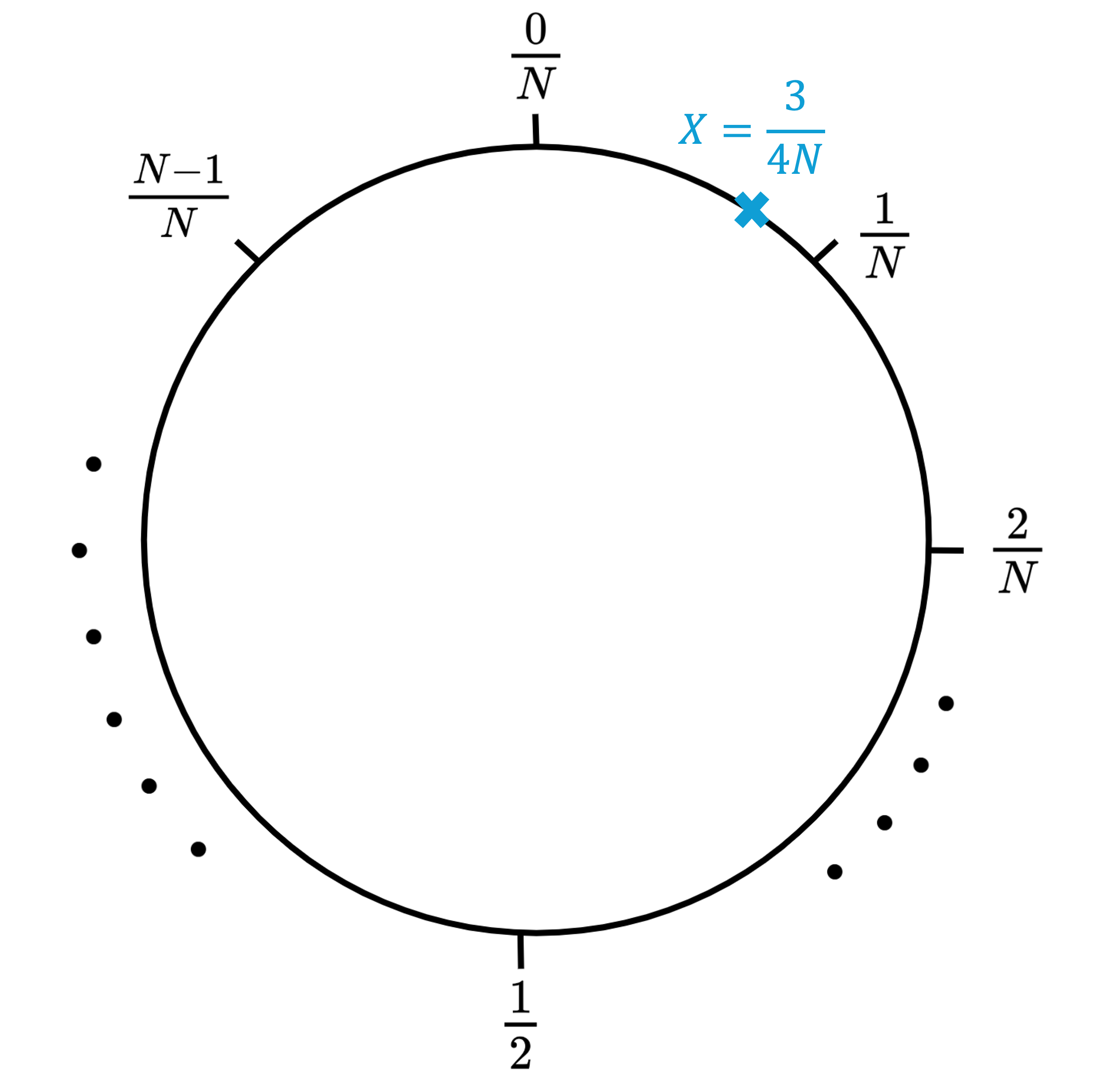}
    \caption{\centering Illustration of the Salop circle model. A customer with $X = \frac{3}{4N}$ will always purchase product $i = 1$ at revenue-maximizing price $\hat{p}$. Moreover, she will purchase the opaque option for any $\delta \geq \frac{\gamma}{4}\left(1-\frac1N\right)$.}
    \label{fig:salop}
\end{figure}

Beyond selling the products individually, the retailer may offer an opaque product (also referred to as an opaque {option}) to customers, as shown in \Cref{fig:opaque_example}. The opaque product gives the retailer the freedom to allocate any of the $N$ products to the customer; in exchange for this uncertainty (in particular, the likelihood that they may receive one of $N$ products, whose characteristics may be far from their preferred characteristics), customers purchase the product at a discounted price. Let $p_o = \hat{p}-\delta$ denote the price of the opaque option, where \mbox{$\delta \in (0,\hat{p}]$} is fixed. We assume customers are risk-neutral, as in \citet{jiang2007price, fay2008probabilistic, jerath2010revenue, fay2015timing, elmachtoub2015retailing, elmachtoub2021customerflex}, and that they value the opaque option as the average of all non-opaque products. Formally, letting $V_o$ denote a customer's random valuation for the opaque product, we have \mbox{$V_o = \frac1N\sum_{i=1}^{N}V_i$}.\footnote{Note that this is equivalent to saying that customers assume they will receive each product uniformly at random.}

In the absence of the opaque option, the customer purchases a product that maximizes her utility, given by $V_i-\hat{p}$ for $i \in [N]$, breaking ties lexicographically. If all products generate negative utility for the customer, she purchases nothing. Since $\bar{v}, \gamma$, and $\hat{p}$ are common to all products, this is equivalent to the customer choosing the product that minimizes the arc-distance $d(X,x_i)$. In the presence of the opaque product, the customer compares the utility generated by each individual product to that of the opaque option, given by $V_o-p_o$. If \mbox{$V_o-p_o \geq \max\{ \max_{i\in[N]}V_i-\hat{p},0\}$}, the customer purchases the opaque option. We illustrate the purchase decision for a customer with preferred product characteristic $X = \frac{3}{4N}$ in \Cref{fig:salop}.

Finally, for expositional simplicity, we assume $N$ is even in the remainder of this section. {This restriction is due to the fact that the revenue-maximizing price under the Salop model differs based on whether or not $N$ is even. However, all of our analysis and insights go through with $N$ odd.}

\subsubsection*{Inventory model.} In addition to the revenue generated from product sales, the retailer incurs costs for managing product inventories. Specifically, in each period, the retailer incurs a holding cost $h > 0$ per unit of on-hand inventory. Additionally, as in \citet{silver1965some,ignall1969optimal,elmachtoub2015retailing,elmachtoub2019value}, we assume that the retailer jointly replenishes her stock of all products whenever the inventory of any product is depleted to 0, incurring a fixed replenishment cost of $K > 0$. We refer to the time between consecutive replenishments as a {\it replenishment cycle}, and let $S \in \mathbb{N}^+$ be the inventory level of each product at the beginning of each replenishment cycle.

\subsubsection*{Performance metrics and trade-offs.} The goal of the retailer is to design simple opaque selling strategies that (i) achieve high long-run average revenue, and (ii) incur low long-run average inventory costs. To formalize this, we introduce some additional notation. We define a policy $\pi$ to be a mapping from the current state of the system (i.e., the inventory levels of all products) to the decision of whether or not to offer the opaque product. For $t \in \mathbb{N}^+$,  let $\omega^\pi(t)$ denote this decision, with $\omega^\pi(t) = 1$ if the product is offered, and $\omega^\pi(t) = 0$ otherwise. We moreover let $q_i^\pi(t)$ be the probability that a customer arriving in period $t$ purchases individual product $i \in [N]$, and denote by $q_o^\pi(t)$ the probability that the opaque product is purchased by the customer (for ease of notation, we omit the dependence of these quantities on~$\delta$ and $\hat{p}$). If the customer purchases the opaque product in period $t$, {we restrict} the retailer {to choosing} a random subset of two products, denoted by $\mathcal{F}(t)$, and allocating the product in $\mathcal{F}(t)$ with higher remaining inventory, breaking ties lexicographically. {We defer a discussion of this restriction, which only strengthens our insights, to our discussion of modeling assumptions at the end of the section.}

Finally, we use $\perperiodrev^\pi(t)$ to denote the expected revenue in period $t$ induced by the decision made by $\pi$. Notice that, absent inventory considerations, it is a priori unclear why the retailer would want to offer an opaque product, given the revenue loss associated with the incentivizing discount $\delta$.  We formalize the intuition that offering the opaque product is costly to the retailer via \Cref{lem:purchase_probabilities,lem:long_run_revenue} below.

{
\begin{proposition}\label{lem:purchase_probabilities}
The revenue-maximizing price in each period is given by $\hat{p} = \bar{v}-\frac{\gamma}{2N}$. Moreover,
\begin{enumerate}[label=(\roman*)]
    \item in periods where the opaque product is not offered, 
    \[q^\pi_i(t) = \frac1N \ \forall \ i \in [N] \qquad \text{ and } \qquad \perperiodrev^\pi(t) = \bar{v}-\frac{\gamma}{2N}\]
    \item in periods where the opaque product is offered:
    \begin{equation*}
      q^\pi_o(t) =
        \begin{cases}
          0 & \text{if } \delta \in \squarebracket{0,\parenthesis{\frac{1}{4} - \frac{1}{2N}} \gamma}\\
          1 - \frac{N}{2} + \frac{2N \delta}{\gamma} & \text{if } \delta \in \parenthesis{\parenthesis{\frac{1}{4} - \frac{1}{2N}} \gamma, \frac{1}{4} \gamma}\\
          1 & \text{if } \delta \geq \frac{1}{4} \gamma
        \end{cases}
        \qquad \text{and} \qquad q^\pi_i(t) = \frac{1-q^\pi_o(t)}{N} \ \forall \ i \in [N],
    \end{equation*}
    with $\perperiodrev^\pi(t) = \overline{v} - \frac{\gamma}{2N} - \delta \cdot q^\pi_o(t)$.
\end{enumerate}
\end{proposition}
}

\begin{figure}[t]
    \centering
    \subfloat[\centering $q_o^\pi(t)$ vs. $\delta$]{{\includegraphics[width=0.45\textwidth]{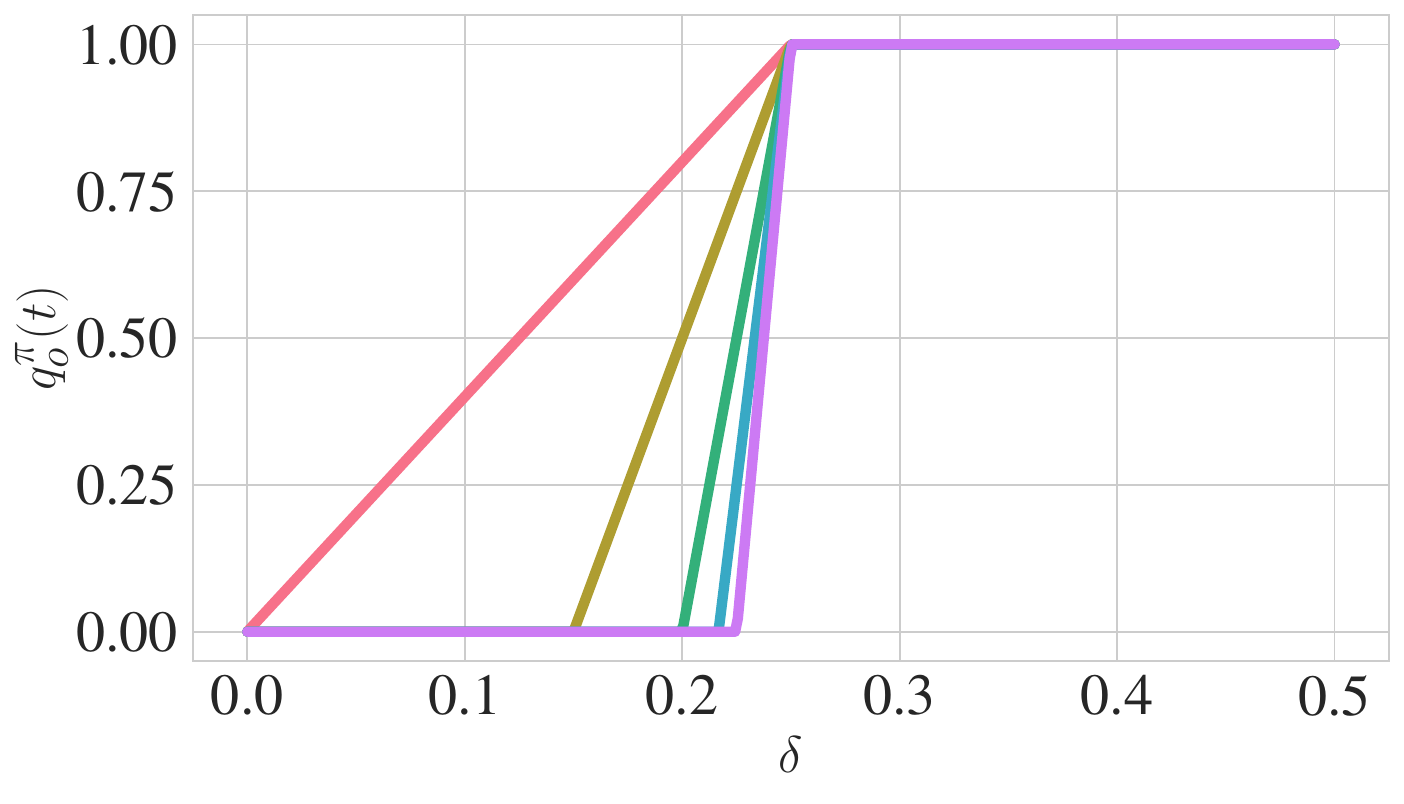} } \label{fig:qo}}%
    \quad 
    \subfloat[\centering $\mathcal{R}^\pi(t)$ vs. $\delta$]{{\includegraphics[width=0.45\textwidth]{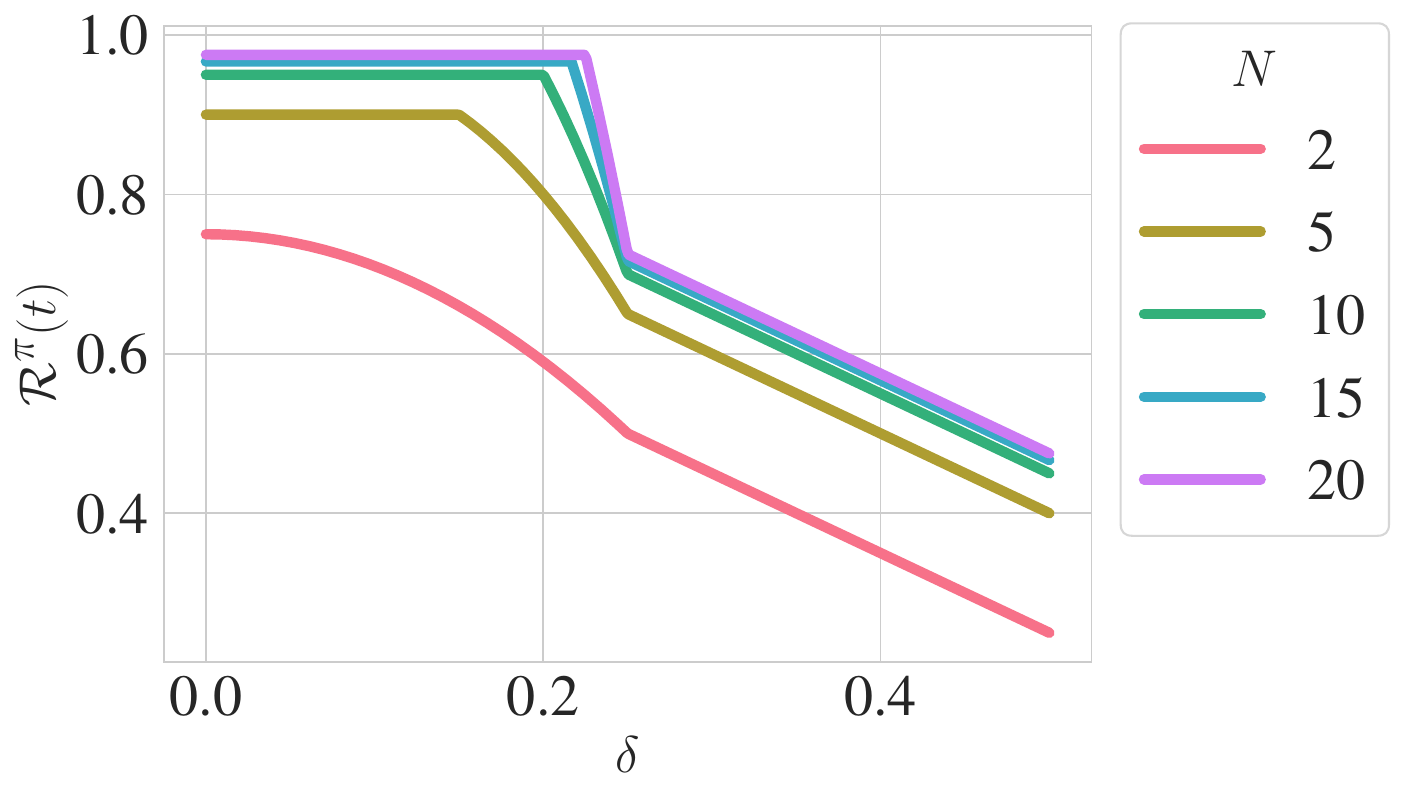} } \label{fig:ro}}%
    \caption{\centering Dependence of $q_o^\pi(t)$ and $\mathcal{R}^\pi(t)$ on $\delta$ and $N$, for $\bar{v} = \gamma = 1$.}
    \label{fig:qo-ro}%
\end{figure}

We illustrate Part (ii) of \Cref{lem:purchase_probabilities} in \Cref{fig:qo-ro}, deferring the proof of the fact to Appendix \ref{app:purchase_probabilities}. In \Cref{fig:qo} we observe that, as $N$ increases, the threshold past which customers begin purchasing the opaque product with positive probability increases;
moreover, for $\delta \in \parenthesis{\parenthesis{\frac{1}{4} - \frac{1}{2N}} \gamma, \frac{1}{4} \gamma}$, the purchase probability $q_o^\pi(t)$ is linearly increasing in $\delta$. These two facts together imply that, as $N$ grows large, $q_o^\pi(t)$ approaches a step function: for any $\delta \leq \gamma/4$, no one chooses the opaque product; past this point however, everyone chooses the opaque product. \Cref{fig:ro} illustrates the implications of these factors for the retailer's revenue. In particular, note that the retailer's maximum possible revenue is achieved when no one purchases the opaque product (i.e., for $\delta \leq \parenthesis{\frac{1}{4} - \frac{1}{2N}} \gamma$). Between this point and $\gamma/4$, her loss is quadratic in $\delta$; finally, past $\gamma/4$ her revenue is linearly decreasing, since everyone purchases the opaque product. This also highlights that there is no incentive for the retailer to ever offer a discount higher than $\gamma/4$. As a result, in the remainder of this section we assume $\delta \in \left(\left(\frac14-\frac{1}{2N}\right)\gamma, \min\left\{\frac14\gamma,\hat{p}\right\}\right]$, with $\bar{v} > \frac{\gamma}{4}$ in order to ensure that this interval is non-empty. Finally, throughout the section we abuse notation and define \mbox{$q_o := 1-\frac{N}{2}+\frac{2N\delta}{\gamma}$}.

Having established the retailer's per-period revenues with and without the opaque option, \Cref{lem:long_run_revenue} provides a closed-form expression for the long-run average revenue of any opaque selling policy $\pi$, denoted by $Rev^\pi$. We denote by $R^\pi$ the random variable representing the length of a replenishment cycle under $\pi$, and by $M^\pi$ the number of times the opaque product is purchased in a replenishment cycle. 
\begin{proposition}\label{lem:long_run_revenue}
    For a fixed policy $\pi$, the retailer's long-run average revenue is given by 
    \begin{align}\label{eq:lr-avg-rev}
    {Rev}^\pi = \hat{p} - \delta \cdot \frac{\EE{M^\pi}}{\mathbb{E}[\R^\pi]}.
    \end{align}
\end{proposition} 
We defer the straightforward proof of \Cref{lem:long_run_revenue} to Appendix \ref{app:long_run_revenue}. Noting that $Rev^\pi$ attains its maximum when $M^\pi = 0$, \Cref{lem:long_run_revenue} underscores that the retailer should never offer the opaque product if her only goal is to maximize revenue. With that said, the loss in revenue from offering the opaque option may be justified if the opaque selling strategy results in inventory cost savings, as observed in \citet{elmachtoub2015retailing,elmachtoub2019value}. To see why this would be the case, notice that exercising the opaque option gives the retailer additional control over product inventories, since she has the freedom to allocate a specific product to the current customer under this option. By allocating products with higher inventory levels, the retailer can delay the time to replenishment (since each individual product's inventory level will deplete more slowly), yielding lower replenishment costs overall. \Cref{eq:lr-cost} below (derived in  \citet{elmachtoub2019value}, Equations (1) and (2)) provides a closed-form expression for the retailer's long-run average inventory costs, denoted by $Inv^\pi$, thereby helping to formalize this intuition:
\begin{align}\label{eq:lr-cost}
    {Inv}^\pi = \frac{K}{\mathbb{E}[\R^\pi]} + \frac{h}{2}\left(2\Ninv \Sinv + 1 - \frac{\mathbb{E}[(\R^\pi)^2]}{\mathbb{E}[\R^\pi]}\right).
\end{align}
For ease of notation, we let $\mathcal{K}^\pi = \frac{K}{\mathbb{E}[\R^\pi]}$ and $\mathcal{H}^\pi =  \frac{h}{2}\left(2\Ninv \Sinv + 1 - \frac{\mathbb{E}[(\R^\pi)^2]}{\mathbb{E}[\R^\pi]}\right)$. 

 \Cref{eq:lr-avg-rev,eq:lr-cost} together recover the intuition that a ``good'' policy trades off between exercising the opaque option frequently enough to keep inventory levels approximately balanced, thus ensuring long replenishment cycles in expectation, but not so frequently as to incur large revenue losses. In what follows, we will evaluate the performance of various opaque selling strategies in a ``large-inventory'' limit in which $N$ is fixed and $S$ grows large. It will moreover be useful to interpret $K \in \Theta(S)$ and $h \in \Theta\parenthesis{\frac1S}$ for our theoretical results, so that the replenishment cost per unit and the total holding cost per period are constants.\footnote{One can alternatively view this as a regime in which the holding cost per period $h$ is small relative to the replenishment cost $K$, as is expected in practice.}

\subsubsection*{Discussion of modeling assumptions.} Before presenting our results, we discuss our most important modeling assumptions for this section. 
\paragraph{On the Salop circle model.} As alluded to above, the Salop circle model is a common modeling framework in settings where the retailer offers horizontally differentiated products, such as differently-colored shirts. Moreover, not only does it generalize the Hotelling model considered in the majority of existing works on opaque selling when $N = 2$ \citep{jiang2007price, fay2008probabilistic, jerath2010revenue,elmachtoub2015retailing}, but, joint with the assumption that customers are risk-neutral, it also is one of the simplest models for which there are {no} revenue gains from opaque selling. As a result, the clear-cut trade-off between revenue and inventory costs makes it a prime candidate to illustrate the relevance of our insights, i.e., that (i) there exist settings in which exerting flexibility is costly, and (ii) {\it judiciously} exerting flexibility in load balancing settings can generate almost all of the benefits of a perfectly balanced load, at small cost to the decision-maker. With that said, our results themselves rely fairly minimally on the Salop circle model. In particular, in \Cref{rem:min-asp} at the end of this section we highlight the minimal assumptions on the customer choice model required for our results to hold. Moreover, in \Cref{sec:opaque_extensions} we demonstrate via extensive numerical experiments the robustness of our insights to the choice model, by considering the popular multinomial logit choice model, as well as {risk-averse and risk-seeking customer behavior}.

\paragraph{On the control levers.} An important assumption in our setting is that the retailer can only influence customer choice by offering the opaque product. Specifically, all products share the same price; moreover, this price and the opaque discount are fixed at the beginning of time. While the fact that all products share the same price is without loss of generality for our analytical results, since products are homogeneous, in our numerical experiments we show that our insights are robust to settings in which products are heterogeneous and the decision-maker implements a {\it discriminatory pricing} (DP) policy, which fixes different prices for different products. 

{We additionally note that a common and well-studied practice to address supply-demand imbalances in retail settings is dynamic pricing, wherein the retailer adaptively varies the prices of products as a function of their inventory levels. Here, we focus on a separate class of policies that address this issue, i.e., policies that set a {\it static} price for all products, and begin offering an opaque option at some point in time. While we do not seek to argue that implementing dynamic pricing is always unreasonable, in many settings (especially when products are homogeneous), frequent price changes may be viewed unfavorably by customers. The addition of a discounted product, however, is less likely to be contentious. Additionally, \citet{elmachtoub2015retailing} numerically showed that policies that fix a static price and offer an opaque option in each period frequently outperform dynamic pricing strategies. Thus motivated, we restrict our benchmarks to opaque selling strategies that fix a single price throughout time.}

We also note that we do not study the {joint} optimization of the fixed price~$\hat{p}$ and the opaque discount~$\delta$. Since we study an asymptotic regime in which~$S$ grows large and~$\delta$ is constant, our main results continue to hold for any (optimized) $\delta$ that is bounded away from 0. However, by focusing on the asymptotics we hide potentially important dependencies on $\delta$ that would be required for this joint optimization. With that said, in light of the fact that we assume that $\hat{p}$ and $\delta$ are set separately, letting $\hat{p}$ be the revenue-maximizing price is a natural choice. Moreover, our numerical experiments provide guidance for choices of $\delta$ that generate high profits for the retailer, given $\hat{p}$.

Finally, recall that we restrict the retailer to allocate one of two products chosen uniformly at random if the customer purchases the opaque product. This restriction is for expositional purposes, in order to maintain consistency with the balls-into-bins strategies presented in \Cref{sec:balls}; we formalize this connection in \Cref{sec:bib-to-opaque}. While sampling two products at random may be desirable in settings where $N$ is large {and querying all $N$ inventory levels is costly from a systems perspective} (as motivated by the power-of-two-choices paradigm in balls-into-bins), in many settings it may be more desirable to allocate the product with the highest remaining inventory across all $N$ products, henceforth referred to as the $N$-sample setting. The algorithm we propose in \Cref{sec:bib-to-opaque} can be directly adapted to such a setting. Moreover, it is easy to see that the two-sample setting constitutes a lower bound on the performance of our policies for the $N$-sample setting. Since our results will show that the two-sample algorithm is asymptotically optimal in the large-inventory regime we consider, extending to $N$-sampling only strengthens the insight that late-stage semi-adaptive opaque selling can yield gains for the retailer.

\paragraph{On the metrics.} In line with our results for the vanilla balls-into-bins model, our main focus is on characterizing the trade-off between revenue and inventory costs for various opaque selling strategies, as opposed to a single objective, such as profit. The reason for this is two-fold. From a technical perspective, obtaining a closed-form expression for the profit of any given policy proves to be a challenge analytically, given the asymptotic regime in which we situate ourselves. Moreover, analyzing revenue and inventory costs separately allows us to obtain finer-grained insights into the value a retailer derives from exercising the opaque option. While we obtain sufficient conditions on regimes for which the semi-dynamic policy outperforms all benchmark policies in theory, our computational experiments moreover show its strong performance over a wide variety of randomly generated instances.

\subsection{Benchmark Policies}\label{ssec:opaque-benchmarks}

Having described the model, we now analyze two policies against which we will benchmark our algorithm: (i) the no-flex policy $\pi^{nf}$, which never offers the opaque option to customers, and (ii) the always-flex policy $\pi^a$, which offers it in each period.\footnote{The always-flex policy is equivalent to the $N$-opaque selling strategy considered in \citet{elmachtoub2019value}.} As before, we use the superscripts $nf$ and $a$ to refer to metrics under the no-flex and always-flex policy, respectively.

Intuitively, the no-flex and always-flex policies lie on opposite extremes of the revenue-inventory costs trade-off curve: while the no-flex policy achieves the maximum possible revenue, it exerts no control over inventory levels; the always-flex policy, on the other hand, suffers substantial revenue losses by offering the opaque option in each period, but should incur low inventory costs due to balanced inventory levels throughout the horizon (and consequently long replenishment cycles). The impact of these two policies on replenishment cycle length is illustrated in \Cref{fig:inventory_path}, which 
plots a sample path of the minimum inventory level under the always-flex and no-flex policies (the ``flex-$\sqrt{S}$'' policy $\pi^f$ is later discussed in \Cref{rmk:elmachtoub2015}). In particular, we observe that $\pi^{nf}$ leads to replenishment cycles that are on average much shorter than those under $\pi^a$, with periods in which the minimum inventory level decreases steeply, as customers' random purchases deplete the product with the lowest inventory. The minimum inventory level under the always-flex policy, on the other hand, looks staircase-like, as the only time the product with the minimum inventory level is allocated to a customer is when inventory levels are perfectly balanced, thereby yielding replenishment cycles that are on average 10 periods longer.

\begin{figure}[t]
    \centering
    \includegraphics[width=0.8\textwidth]{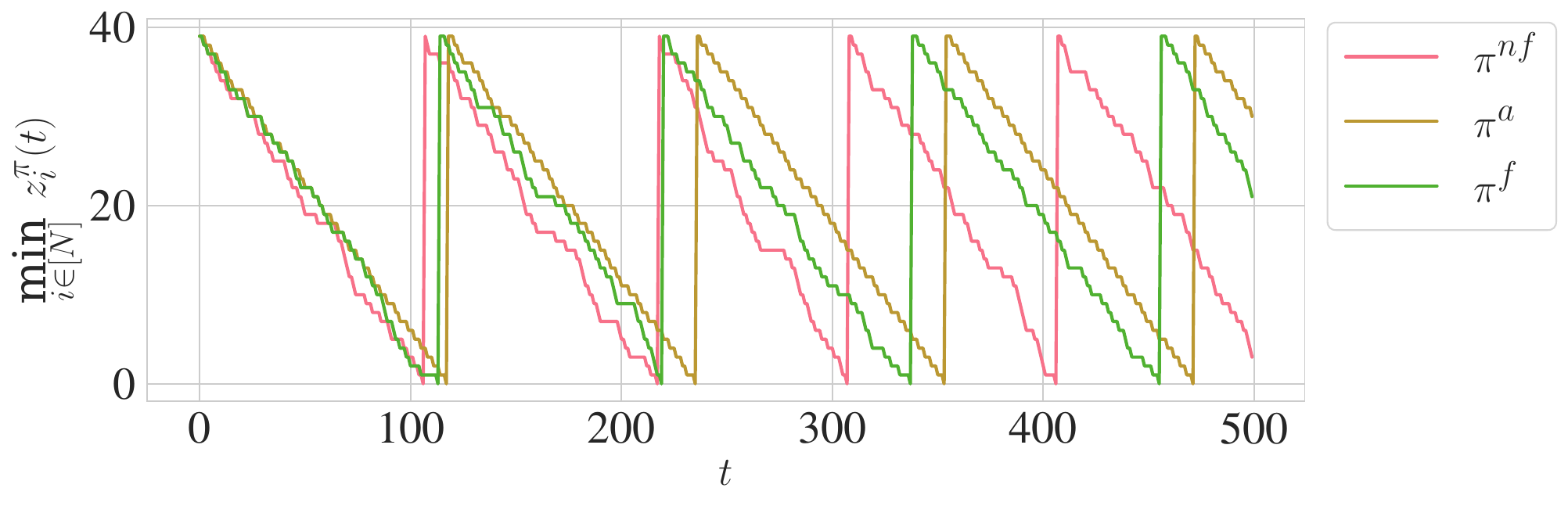}
    \caption{\centering Minimum inventory level, denoted by $\min_{i\in[N]}z_i^\pi(t)$, under $\pi^{nf}, \pia$ and the ``flex-$\sqrt{S}$" policy $\pi^f$ which offers the opaque product uniformly at random in $\Theta\parenthesis{\frac{1}{\sqrt{S}}}$ of the periods (see \Cref{rmk:elmachtoub2015}). Here, $S = 40, N = 3$ and $q_o = 1.$ Periods in which the minimum inventory level increases to 40 signify the start of a new replenishment cycle.}
    \label{fig:inventory_path}
\end{figure}

\Cref{prop:nf_objective,prop:af_objective} formalize this intuition. We defer their proofs to Appendix \ref{app:benchmark}.

\begin{proposition}\label{prop:nf_objective}
Under the no-flex policy $\pi^{nf}$, the following holds:
\begin{enumerate}[label=(\roman*)]
    \item $\mathbb{E}[M^{nf}] = 0$ and $\mathbb{E}[\R^{nf}] = NS - \theta_{nf}$, where $\theta_{nf} \in \Omega(\sqrt{S})$.
    \item ${Rev}^{nf} = \hat{p}$ and ${Inv}^{nf} 
    = \frac{K}{NS} + {\frac{h}{2} \cdot NS} 
    +\xi_{nf}$, where $\xi_{nf} \in \Omega\parenthesis{\frac{1}{\sqrt{S}}}$.
\end{enumerate}
\end{proposition}

\begin{proposition}\label{prop:af_objective}
Under the  always-flex policy $\pi^{a}$ with $\delta \in \Theta(1)$\footnote{Since $\delta > \parenthesis{\frac{1}{4} - \frac{1}{2N}}\cdot \gamma$, the condition that $\delta \in \Theta(1)$ is naturally satisfied when $N > 2.$}, the following holds: 
\begin{enumerate}[label=(\roman*)]
    \item $\mathbb{E}[M^{a}] \in \Theta(S)$ and $\mathbb{E}[\R^{a}] = NS - \theta_{a}$, where $\theta_{a} \in \mathcal{O}(1)$.
    \item ${Rev}^{a} = \hat{p} - \delta \left({1-N/2+2N\delta/\gamma}\right)$ and ${Inv}^{a} = \frac{K}{NS} + {\frac{h}{2} \cdot NS} + \xi_{a}$, where $\xi_{a} \in \mathcal{O}\parenthesis{\frac{1}{S}}$.
\end{enumerate}
\end{proposition}

Since the retailer stocks exactly $NS$ products at the beginning of each replenishment cycle, $NS$ is a loose upper bound on the maximum possible length of a replenishment cycle, which we refer to as the {\it ideal} replenishment cycle length. As a result, the constant $\frac{K}{NS} + {\frac{h}{2} \cdot NS}$ can be thought of as a theoretical lower bound on the retailer's inventory costs. We henceforth refer to this quantity as the {\it ideal} inventory cost. Using this interpretation, \Cref{prop:nf_objective} establishes that the stochasticity in customers' purchase decisions without the opaque option results in a decrease in expected cycle length of $\Omega(\sqrt{S})$ relative to the ideal replenishment cycle, leading to a long-run average inventory cost increase of $\Omega\left(\frac{1}{\sqrt{S}}\right)$ relative to this theoretical lower bound. \Cref{prop:af_objective}, on the other hand, establishes that exerting the opaque option in each period results in a cycle length that is within $\mathcal{O}(1)$ of the ideal replenishment cycle length, with inventory costs within $\mathcal{O}\parenthesis{\frac1S}$ of the ideal inventory cost. However, this comes at a constant per-period revenue loss. Notice that this latter loss far outweighs the inventory cost savings from offering the opaque option in each period. As a result, from a profit perspective it is clear that in this homogeneous setting, the retailer should not offer the opaque product in each period, for any value of $\delta \in \Theta(1)$.

Motivated by these extremes, we seek to design policies that achieve a two-fold objective, as it relates to the revenue-inventory cost tradeoff: (i) incurring $o(1)$ loss relative to the revenue-maximizing solution that never offers the opaque option, and (ii) incurring order-wise identical inventory savings as the always-flex policy that has the most control over inventory levels, i.e., $\mathcal{O}\left(\frac1S\right)$ loss relative to the ideal inventory costs.

\begin{remark}\label{rmk:elmachtoub2015}
    A natural idea to reduce the revenue loss incurred from offering the opaque product would be to set $\delta \in \mathcal{O}\parenthesis{\frac{1}{\sqrt{S}}}$, so that a customer only purchases the opaque product with probability $\mathcal{O}\parenthesis{\frac{1}{\sqrt{S}}}$, as suggested in Proposition 4 of \citet{elmachtoub2015retailing} for the special case of $N = 2$. This idea, however, does not achieve the desired trade-off in our model, for two reasons. First of all, when $N > 2$, \Cref{lem:purchase_probabilities} establishes that $\delta > \parenthesis{\frac{1}{4} - \frac{1}{2N}} \gamma \in \Theta(1)$ is required for the opaque product to be chosen with strictly positive probability. Moreover, independent of the behavioral model, {a ``flex-$\sqrt{S}$" policy $\pi^f$} that exerts flexibility uniformly at random in $\Theta\parenthesis{\frac{1}{\sqrt{S}}}$ of the periods will not guarantee that the expected replenishment cycle length is a constant away from the ideal cycle length $NS$; as a result, it will not be able to achieve the inventory cost savings of the same order as $Inv^a$, our second desideratum. Intuitively, this is due to the fact that, for any constant $a > 0,$ a policy that exerts flexibility randomly in $\frac{a}{\sqrt{S}}$ fraction of periods will not be able to approximately balance the system if its gap ever exceeds $(a+1) \sqrt{S}$; this event occurs with constant probability for any $a$. We illustrate this phenomenon in \Cref{fig:inventory_path}, which shows that while replenishment cycles under $\pi^f$ are longer than those under the no-flex policy, they still fall short of those under $\pia$. We provide a more comprehensive inventory cost and revenue analysis of this policy in our numerical experiments (see \Cref{sec:opaque_extensions}).
\end{remark}

\subsection{Leveraging Balls-Into-Bins for Effective Late-Stage Opaque Selling Strategies}\label{sec:bib-to-opaque}

In this section we leverage the insights from \Cref{sec:balls} to design opaque selling strategies that appropriately trade off between revenue and inventory costs. To do so, we first specify the analogy to the balls-into-bins framework, first identified in \citet{elmachtoub2019value}. Here, bins correspond to product types, and balls correspond to customers. Every time a customer purchases a product (thereby depleting its inventory by one), a ball is allocated to a bin (thereby increasing the bin's load by one);  a customer purchasing an opaque product is therefore analogous to allocating a flexible ball. Moreover, keeping inventory levels balanced in the opaque selling model is analogous to the goal of maintaining a balanced load across bins. In the opaque selling model, we are similarly penalized for flexing too frequently given the dependence of the long-run average revenue on the expected number of flexes, $\mathbb{E}[M^\pi]$ (see \Cref{eq:lr-avg-rev}). 

Despite this fairly straightforward analogy, there exist important distinctions between the two settings. First of all, while the balls-into-bins setting has a fixed horizon of length $T$, the horizon in the opaque selling model is the replenishment cycle, whose length is not only unknown, but more importantly {\it endogenous} to (i) the random fluctuations in customers' purchases, and (ii) the retailer's opaque selling decisions. As a result, naively adapting the dynamic policy designed for the balls-into-bins setting (see \Cref{alg:ball_full}) to this setting first would require the algorithm to maintain a prediction of the time at which the replenishment cycle ends. More fundamentally, however, this endogeneity issue makes it a priori unclear that such ``point of no return''-style policies, which exercise the opaque option extremely infrequently, suffice to induce long replenishment cycles. Finally, we highlight the subtle difference in the way in which flexing is penalized in the opaque selling model, relative to the vanilla balls-into-bins setting. In particular, \Cref{eq:lr-avg-rev} establishes that the retailer's long-run average revenue is decreasing in the expected number of opaque purchases, {\it for a fixed expected replenishment cycle length}. However, as $M^\pi$ increases, so does $R^\pi$. On the one hand, this renders the exact dependence of the revenue loss on the number of opaque purchases less straightforward; on the other, it makes clear the importance of policies that {\it efficiently} exercise the opaque selling option (i.e., in a way that maximizes the bang-per-buck of the cost of an opaque discount per increase in replenishment cycle length). As we saw in \Cref{prop:af_objective}, the always-flex policy fails this litmus test, given that it offers the opaque product to customers in each period, even when inventory levels are approximately balanced.

\subsubsection*{The semi-dynamic policy.} In our main result for this section, we propose a variant of the dynamic policy designed for the balls-into-bins setting that addresses these subtleties. Specifically, we consider an analogous notion of system balancedness for product inventories, similarly referred to as the {\it gap} of the system and defined as:
\begin{align}\label{eq:inv-gap1}
\Gap_I^{\pi}(t) = S - \frac{t}{N} - \min_{i \in [N]} z_i^{\pi}(t), \quad \forall\, t \in [\R^\pi],
\end{align}
where $\stateinv_i^{\pi}(t)$ denotes the remaining inventory of product $i$ in period $t$ under policy $\pi$, and $t$ is re-initialized at the beginning of each replenishment cycle. Intuitively, the term $S-t/N$ captures the remaining inventory for each product if the previous realized demand had been split equally across products. Thus, $\Gap_I^{\pi}(t)$ captures the deviation between the current inventory level for product $i$, $z_i^{\pi}(t)$, and a perfectly balanced state of the system.

As in \Cref{alg:ball_full}, the semi-dynamic policy triggers the opaque option the first time $\Gap_I^{\pi}(t)$ reaches a ``point of no return.'' To formalize this point of no return, we define a fictitious end of horizon, denoted by $T:=N(S-1)+1$, corresponding to the length of the replenishment cycle if the retailer balanced inventory in each period. The quantity $T$ is a tight upper bound on the maximum cycle length; we henceforth refer to it as the {\it achievable-ideal} cycle length.\footnote{To see this, note that it takes $N(S-1)$ periods for all products to reach an inventory level of one. In the next period, no matter which product is purchased, its inventory level will then drop to 0, and a replenishment will occur.} Under the semi-dynamic policy, the retailer verifies the following threshold condition in each period:
\begin{align}\label{eq:threshold-opaque}
\Gap_I^d(t) \geq \frac{c_d(T-t)q_o}{N},
\end{align}
where $c_d > 0$ is a tuning parameter. If this condition is satisfied, the opaque option is offered to customers until the end of the replenishment cycle. Notice that this policy is entirely analogous to \Cref{alg:ball_full}, save for the fact that the opaque option allocates the product with the lowest inventory level, as compared to the balls-into-bins policy, which throws a flexible ball into the highest-loaded bin. For completeness, we include a formal description of the semi-dynamic policy within this context in Appendix \ref{apx:alg-semi}.

\Cref{thm:dynamic_objective} establishes that the semi-dynamic policy achieves a ``best-of-both-worlds'' as it relates to the revenue-inventory cost trade-off faced by the retailer: its long-run average revenue is within $o(1)$ of the revenue-maximizing policy that never offers any opaque discounts, all the while achieving all of the inventory cost savings of the policy that seeks to perfectly balance inventory levels at all times.

\begin{theorem}\label{thm:dynamic_objective}
Under the semi-dynamic policy, the following holds, for any $\cstsemi \leq \frac{1}{10 \binom{N}{2}}$ and $\delta \in \Theta(1)$:
\begin{enumerate}[label=(\roman*)]
    \item $\mathbb{E}[M^{d}] \in \mathcal{O}(\sqrt{S})$ and $\mathbb{E}[\R^{d}] = NS  -\theta_{d}$, where {$\theta_{d} \in \mathcal{O}(1)$}.
    \item ${Rev}^{d} = \hat{p} - \delta\left({1-N/2+2N\delta/\gamma}\right) \cdot \xi^1_{d}$ and ${Inv}^{d} = \frac{K}{NS} + {\frac{h}{2} \cdot NS} + \xi^2_{d}$, where $\xi^1_{d} \in \mathcal{O}\parenthesis{\frac{1}{\sqrt{S}}}$ and \mbox{$\xi^2_{d} \in \mathcal{O}\parenthesis{\frac{1}{S}}$}.
\end{enumerate}
\end{theorem}

\begin{table}[t]
\parbox{.45\linewidth}{
\centering
\begin{tabular}{cccc}
\toprule
&$\pi^{nf}$ &$\pi^a$& $\pi^d$\\
\midrule
$\mathbb{E}[M^\pi]$ & 0 & $\Theta(S)$ & $\mathcal{O}\parenthesis{\sqrt{S}}$
\\
$NS-\mathbb{E}[R^\pi]$ & $\Omega\parenthesis{{\sqrt{S}}}$ & $\mathcal{O}\parenthesis{1}$ & $\mathcal{O}\parenthesis{1}$ \\
\bottomrule
\end{tabular}
\caption{Discount Frequency-Cycle Length Trade-off}\label{tab:res2}
}
\hfill
\parbox{.45\linewidth}{
\centering
\begin{tabular}{cccc}
\toprule
&$\pi^{nf}$ &$\pi^a$& $\pi^d$\\
\midrule
$Rev^*-Rev^\pi$ & 0 & $\Theta(1)$ & $\mathcal{O}\parenthesis{\frac{1}{\sqrt{S}}}$
\\
$Inv^{\pi}-Inv^*$ & $\Omega\parenthesis{\frac{1}{\sqrt{S}}}$ & $\mathcal{O}\parenthesis{\frac{1}{S}}$ & $\mathcal{O}\parenthesis{\frac{1}{S}}$ \\
\bottomrule
\end{tabular}
\caption{Revenue-Inventory Cost Trade-off}\label{tab:res1}
}
\end{table}



We summarize our results in \Cref{tab:res1,tab:res2}, using $Rev^*$ and $Inv^*$ to respectively denote the maximum long-run average revenue $\hat{p}$ and the ideal inventory costs $\frac{K}{NS}+{\frac{h}{2} \cdot NS}$.  We provide a brief proof sketch below, deferring the formal proof of \Cref{thm:dynamic_objective} to Appendix \ref{apx:opaque-semi}.

\begin{proof}{Proof sketch.}
The proof of the theorem relies on a coupling between the inventory and balls-into-bins processes, where each customer's purchase decision corresponds to an allocation of an arriving ball to the corresponding bin. Under this coupling, $\Gap_I^d(t) = \Gap^d(t)$ for all $t$ in a given replenishment cycle; as a result, throughout the proof we focus on the coupled balls-into-bins system.

As before, let \mbox{$\Tstar := \inf\{t: \Gapd(t) \geq \frac{\cstsemi(T-t)q}{N}\}$} be the first time the threshold condition is satisfied. The bound on the expected number of opaque purchases $\mathbb{E}[M^d]$ is a direct corollary of \Cref{thm:ball_semi2}, which established that $\mathbb{E}[T-\Tstar] \in \mathcal{O}(\sqrt{T})$. Using the fact that $T = N(S-1)+1$ by definition, we obtain $\mathbb{E}[M^d] \in \mathcal{O}(\sqrt{S})$. All other results then follow from the bound on $\mathbb{E}[R^d]$, which we describe below.

To bound this quantity, we first relate the distribution of the replenishment cycle length to that of the gap of the system, for all $t$. In particular, leveraging results from \citet{elmachtoub2019value}, we show: 
\begin{equation}\label{eq:relation-body}
    \mathbb{E}[R^d]= NS - \sum_{t=S}^{N(S-1)+1}\mathbb{P}\left(\Gap^d(t)\geq S-t/N\right) + \Omega(1).
\end{equation}
Hence, it suffices to establish that $\sum_{t=S}^{N(S-1)+1}\mathbb{P}\left(\Gap^d(t)\geq S-t/N\right)\in \mathcal{O}(1)$.

Herein lies the main technical challenge as compared to the vanilla balls-into-bins model. While this latter model only required a bound on the {\it expected} gap at the end of the horizon, since the horizon is random in the opaque selling model, we must now bound the {\it tail} of the gap {\it in each period}. We do so by splitting the flexing horizon into two, letting $\Ttilde$ denote the midpoint between $\Tstar$ and $T$, for any realization of $\Tstar$.

\smallskip

\textbf{Case 1: $t \in \{\Tstar,\ldots,\Ttilde\}$.} In this region, $S-t/N$ is large, so we should naturally expect $\mathbb{P}\left(\Gap^d(t)\geq S-t/N\right)$ to be small. The main challenge in obtaining tight bounds on the tail of $\Gapd(t)$ lies in the fact that there exists an accumulated imbalance at $\Tstar$ on which we would need to condition. The following key fact, however, allows us to eschew this obstacle by instead turning our attention to the more tractable always-flex policy.
\begin{lemma}[Informal Lemma]\label{lem:informal-lemma}
Let $\Gapa(t-\Tstar)$ denote the gap induced by the always-flex policy for a system initialized with empty bins at $\Tstar$. Then,
\[\mathbb{P}\left(\Gapd(t) \geq S-t/N \mid \Gapd(\Tstar) \leq a, \Tstar\right) \leq \mathbb{P}\left(\Gapa(t-\Tstar) \geq S-t/N-a \mid \Tstar\right).\]
\end{lemma}

To apply this lemma, we require an upper bound on $\Gapd(\Tstar)$. We use the threshold condition to argue that \mbox{$\Gapd(\Tstar) < \frac{c_dq(T-\Tstar+1)}{N}+1$}. We then argue that, for $t$ such that \mbox{$S-t/N-\left(\frac{c_dq(T-\Tstar+1)}{N}+1\right)>0$},
\[\mathbb{P}\left(\Gapa(t-\Tstar) \geq S-t/N-\left(\frac{c_dq(T-\Tstar+1)}{N}+1\right) \mid \Tstar\right)\in e^{-\Omega\left(S-t/N-\left(\frac{c_dq(T-\Tstar+1)}{N}+1\right)\right)}.\] Summing over all $t$ in this interval, we obtain the first part of our constant upper bound.

\smallskip

\textbf{Case 2: $t \in \{\Ttilde+1,\ldots,T\}$.} Unfortunately, in this region the same arguments fail to apply given that $S-t/N$ is now small (and in particular, $S-t/N-\left(\frac{c_dq(T-\Tstar+1)}{N}+1\right) \leq 0$, preventing us from applying the tail bound on $\Gapa(t-\Tstar)$). We first consider the ``easy'' scenario in which the loads of the maximally and minimally loaded bins at $\Ttilde$ intersected for some $\tau \in \{\Tstar,\ldots,\Ttilde\}$. In this scenario, we leverage our analysis from the proof of \Cref{thm:ball_static} to show that the likelihood that the gap is nonzero at $\Ttilde$ is exponentially decreasing in $\Ttilde-\tau$. If it is indeed non-zero, in the worst case, $\Gapd(\Ttilde) \leq \Ttilde-\tau$, since at most $\Ttilde-\tau$ balls could have been placed in any one bin between $\tau+1$ and $\Ttilde$. We can then again apply \Cref{lem:informal-lemma} to bound the tail of the gap. Summing over all $t$ in this region, we again obtain a constant bound for the easy scenario. 

The remaining scenario to consider is when the loads of maximally and minimally loaded bins at $\Ttilde$ never intersected before $\Ttilde$. Let $E^1$ denote this event. Here is where we leverage the definition of $\Ttilde$ as the midpoint of $\Tstar$ and $T$. In particular, the threshold condition implies that the difference between the loads of the maximally and minimally bins at $\Tstar$ is a linear function of $T-\Tstar$, and therefore also linear in $\Ttilde-T$. We then leverage the analysis from the proof of \Cref{thm:ball_static} to show that the likelihood that $E^1$ occurs is exponentially decreasing in $\Ttilde-T$. Summing this over $t \in \{\Ttilde+1,\ldots,T\}$, we obtain our constant bound.
\end{proof}

\smallskip

Note that an immediate corollary of \Cref{thm:dynamic_objective} is that the semi-dynamic policy is a strict improvement on the always-flex policy with respect to profit, for any $\delta \in \Theta(1)$. The comparison between the no-flex and semi-dynamic policies, however, is more subtle, and will in general be instance-dependent. In \Cref{cor:comparison} below we identify a wide regime of parameters for which the semi-dynamic policy generates higher profits than the no-flex policy. 

\begin{proposition}\label{cor:comparison}
Let $Pr^a, Pr^{d},$ and $Pr^{nf}$ respectively denote the long-run average profits under the always-flex, the semi-dynamic, and the no-flex policies. For $\cstsemi = \frac{1}{10 \binom{N}{2}}$\footnote{The assumption that $\cstsemi = \frac{1}{10 \binom{N}{2}}$ is for simplicity. The result would still hold given any constant lower bound on $\cstsemi$, with the only modification being a change in the constant $C$ in Part (ii).}, the following holds:
\begin{enumerate}[label=(\roman*)]
    \item ${Pr}^{d} - {Pr}^{a} \in \Omega(1)$ for any $\delta \in \Theta(1)$.
    \item Fix $\delta \in \Theta(1)$ and suppose $K = \psi \cdot NS$ and $h = \psi/(NS)$ for some constant $\psi > 0$. Then, there exists a constant $C > 0$ such that ${Pr}^{d} - {Pr}^{nf} \in \Omega\parenthesis{\frac{1}{\sqrt{S}}}$ for all $\psi > C \delta$.
\end{enumerate}
\end{proposition}

We defer the proof of \Cref{cor:comparison} to Appendix \ref{app:comparison}. At a high level, this result formalizes the intuition that, if the replenishment and holding costs are high enough relative to the disutility a customer incurs from the opaque product, then it is worthwhile for the retailer to implement the semi-dynamic policy. The threshold at which this occurs naturally has a dependence on $N$ and $\gamma$, since the discount required to incentivize customers to choose the opaque option increases with these two parameters, as established in \Cref{lem:purchase_probabilities}. {Our computational experiments validate that the semi-dynamic policy outperforms the no-flex policy over a wide range of randomly generated instances.}

\begin{remark}
{Though of less practical interest, one can also adapt the static policy in the balls-into-bins model to the opaque selling model. Using similar arguments as those used in \Cref{thm:dynamic_objective}, it is easy to show that this static policy asymptotically leads to the same expected replenishment cycle length, at the cost of a higher number of opaque discounts given. We omit this analysis for brevity.}
\end{remark}

\begin{remark}\label{rem:min-asp}
    We conclude the section with a brief discussion of the minimal assumptions on the choice model required for our theoretical results. In particular, for \Cref{prop:nf_objective}, \Cref{prop:af_objective}, and \Cref{thm:dynamic_objective} to hold, we only require that (i) a purchase is made in each period, and (ii) each individual product is equally likely to be purchased. As discussed in \Cref{ssec:opaque-setup}, we require (i) to hold as we are interested in settings where there is a clear trade-off between exercising the flexible option and load balancing. With respect to (ii), we {suspect} that our results can be extended in a fairly straightforward manner to settings in which the purchase probability across products is heterogeneous, as discussed in \Cref{rem:non-uniform-bins} for the balls-into-bins model. Moreover, the simplicity of the Salop circle model allows us to obtain a closed-form expression of $q_o$ as a function of $\delta$, and therefore allows for the profit comparison in \Cref{cor:comparison}.
    Finally, we underscore that while these two minimal assumptions allow us to formalize the main managerial insight of this section --- that strategically timed, late-stage opaque selling achieves the optimal revenue and inventory cost trade-off ---  our numerical experiments demonstrate that this insight continues to hold for a wide variety of choice models in which neither condition holds.
\end{remark}

%% file: opaque-numerics.tex
\section{Computational Experiments}\label{sec:numerics}

In this section we complement our theoretical results via extensive computational experiments for both the vanilla balls-into-bins and the opaque selling setups.

\subsection{Balls-into-Bins}\label{sec:bib_simulation}

We begin with a numerical study of the balls-into-bins model. Throughout our experiments, we fix $N = 5$, $q = 0.1, \staticcst = 20$ and $\semicst = 0.5$, except when specified. All results are averaged across 500 replications.

\paragraph{Benchmark comparison.} In \Cref{fig:bib_comparison} we validate our theoretical results by comparing $\EE{\Gapp(T)}$ and $\EE{M^\pi}$ under the no-flex, always-flex, static, and semi-dynamic policies, for $T \in \{10^4, 2\cdot10^4,\ldots,9\cdot10^4\}$. Additionally, we investigate the performance of a practical, fully dynamic variant of the semi-dynamic policy, termed the {\it dynamic} policy $\pi^{dyn}$. Rather than flexing in all periods after threshold condition \eqref{eq:dynamic-threshold} is met, this policy {allows for the possibility that the gap of the system may naturally self-correct even after $\Tstar$;} as a result, it keeps checking the threshold condition in each period after $\Tstar$, only ever flexing in the period immediately after the condition is met.

\begin{figure}[t]
    \centering
    \subfloat[\centering {$\mathbb{E}[\Gap^\pi(T)]$ vs. $T$}]{\includegraphics[width=0.42\textwidth]{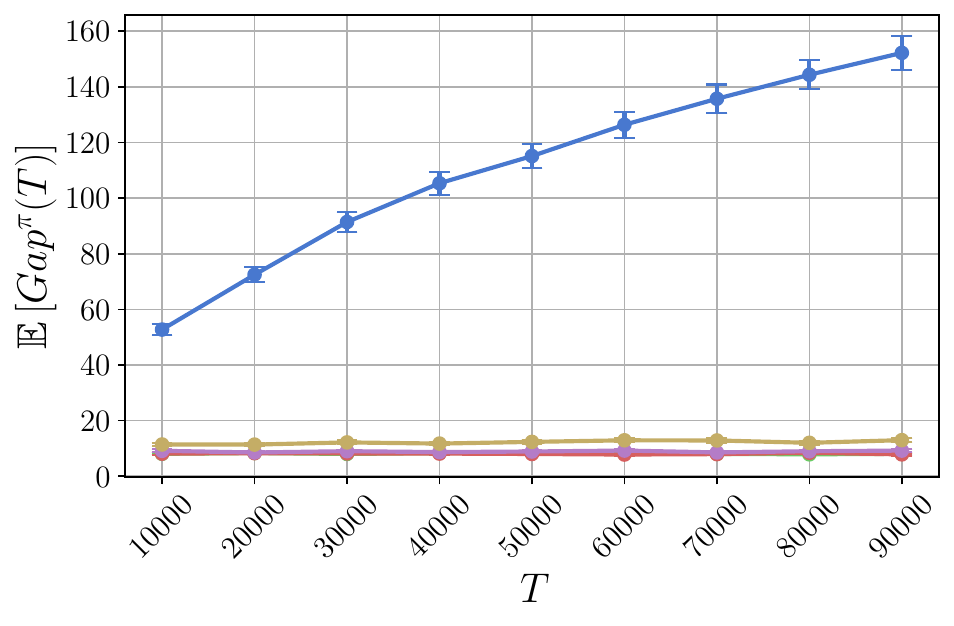} } %
    \quad 
    \subfloat[\centering {$\mathbb{E}[M^\pi]$ vs. $T$}]{\includegraphics[width=0.51\textwidth]{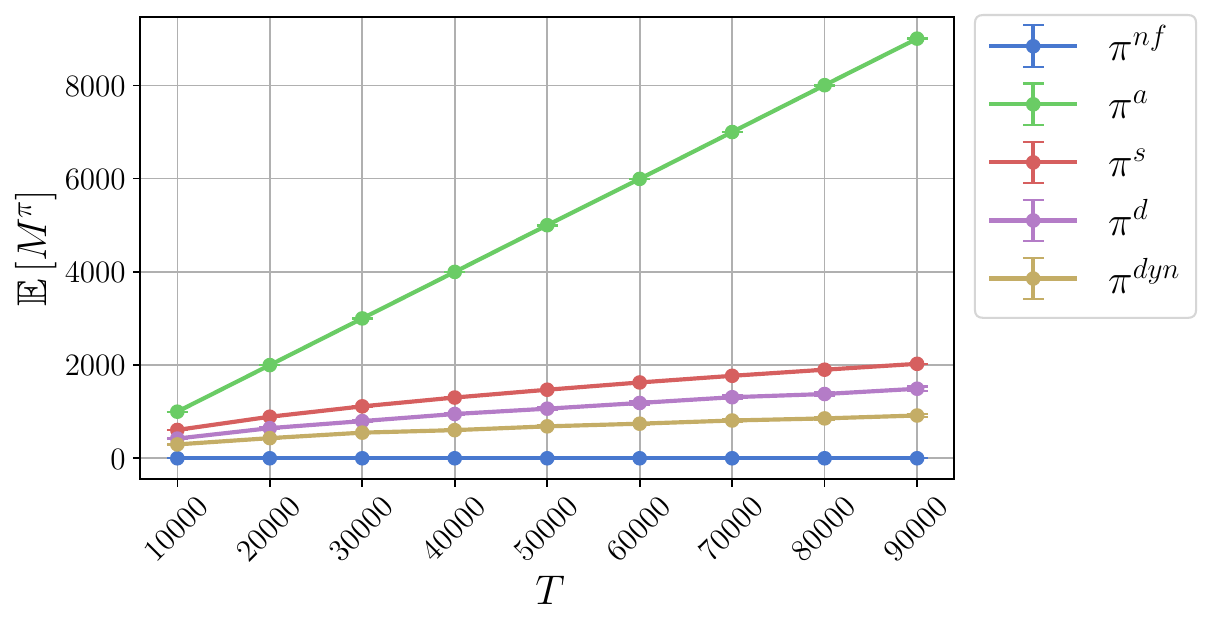} }
    \caption{\centering Comparison of the no-flex, always-flex, static, semi-dynamic, and dynamic policies.}
    \label{fig:bib_comparison}
\end{figure}

Our results corroborate Theorems \ref{thm:ball_static}-\ref{thm:ball_semi}. In particular, in contrast to the no-flex policy, all flexing policies maintain a constant gap as $T$ scales large. We moreover observe that the static, semi-dynamic, and dynamic policies exert flexibility in a sublinear number of periods, with the dynamic policy exerting half as many flexible throws as the static policy, for $T = 9\cdot 10^4$. This comes at the cost of a slightly higher gap.

\paragraph{Impact of $N$ and $q$.} {We now examine the impact of the instantiation of $a_d$ {on} the performance of the semi-dynamic policy. 
Throughout all experiments, we fix $T = 2 \cdot 10^5$.

In \Cref{fig:sensitivity_N} we fix $q = 0.5$ and plot $\mathbb{E}[\Gap^d(T)]$ and $\mathbb{E}[M^d]$ versus $N$, for various scalings of $a_d$ with respect to $N$. While \Cref{thm:ball_semi} establishes that $\mathbb{E}[\Gap^d(T)] \in \mathcal{O}(1)$ for all $a_d \leq \frac{1}{5 {N \choose 2}} \in \Theta(1/N^2)$, our results suggest that setting {$a_d \in \mathcal{O}\parenthesis{\sqrt{N}/\log N}$} suffices for our algorithm to yield a constant end-of-horizon gap (see \Cref{fig:n-gap}).\footnote{In this subsection we abuse notation and use $\Theta(\cdot)$ to refer to the scaling with respect to $N$ or $q$, where relevant.} Moreover, since the right-hand side of the threshold condition, $\frac{a_d(T-t)q}{N}$, is much higher for this larger value of $a_d$, significantly fewer flexes are exerted, as can be seen in \Cref{fig:n-md}. {Additionally, {the fact that setting $a_d\in \Theta(1/\sqrt{N})$ leads to an increasing gap in} \Cref{fig:n-gap} suggests that setting $a_d \in \mathcal{O}\parenthesis{\sqrt{N}/\log N}$ is not only sufficient, but also necessary, to close the gap by the end of the horizon.
}

\begin{figure}[t]
    \centering
    \subfloat[\centering {$\mathbb{E}[\Gapd(T)]$ vs. $N$}\label{fig:n-gap}]{{\includegraphics[width=0.45\textwidth]{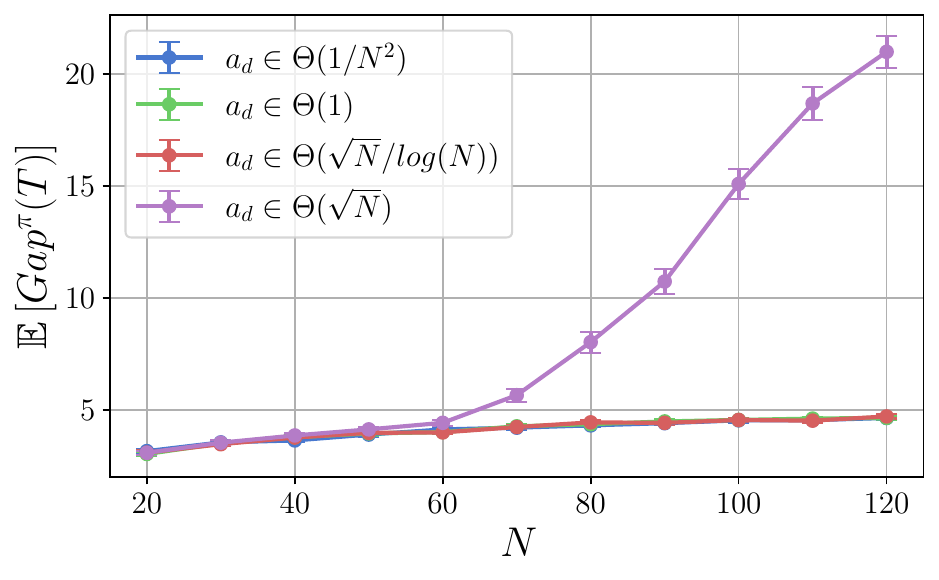} }}%
    \quad 
    \subfloat[\centering {$\mathbb{E}[M^d]$ vs. $N$}\label{fig:n-md}]{{\includegraphics[width=0.45\textwidth]{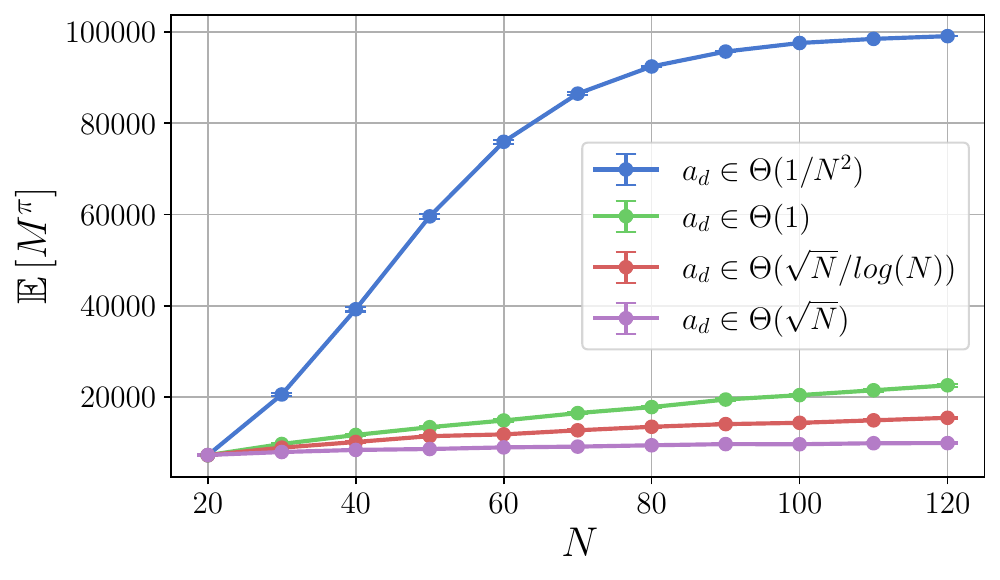} }}%
    \caption{\centering Performance of the semi-dynamic policy for various instantiations of $a_d$ with respect to $N$. Here, $T = 2 \cdot 10^5$, $q=0.5$.}
    \label{fig:sensitivity_N}
\end{figure}

In \Cref{fig:sensitivity_q} we turn our attention to the dependence of $a_d$ on $q$.  For $a_d \leq \frac{1}{5 {N \choose 2}}$, the threshold $\frac{\semicst(T-t)q}{N}$ scales linearly in $q$, which implies that the policy begins flexing earlier for small values of $q$. Such behavior is intuitively desirable, since small values of $q$ also imply that the policy has fewer opportunities to reduce the gap once it begins flexing. \Cref{fig:q-gap} justifies our choice of $a_d$ independent of $q$, since choices of $a_d$ that are decreasing in $q$ lead to substantial increases in $\mathbb{E}[\Gap^d(T)]$, especially for small values of $q$.  
Naturally, setting $a_d \in \Theta(1)$ results in significantly more flexes than with a smaller threshold (see \Cref{fig:q-md}), though that number does not depend on $q$. This independence can easily be explained via the following back-of-the-envelope calculation that uses the threshold condition to establish that ${T-}\Tstar \in \Theta\left(\frac{1}{a_dq}\right)$, and therefore $\mathbb{E}[M^d] \in \Theta(\frac{1}{a_d})$.}

\begin{figure}[t]
    \centering
    \subfloat[\centering {$\mathbb{E}[\Gapd(T)]$ vs. $q$}\label{fig:q-gap}]{{\includegraphics[width=0.45\textwidth]{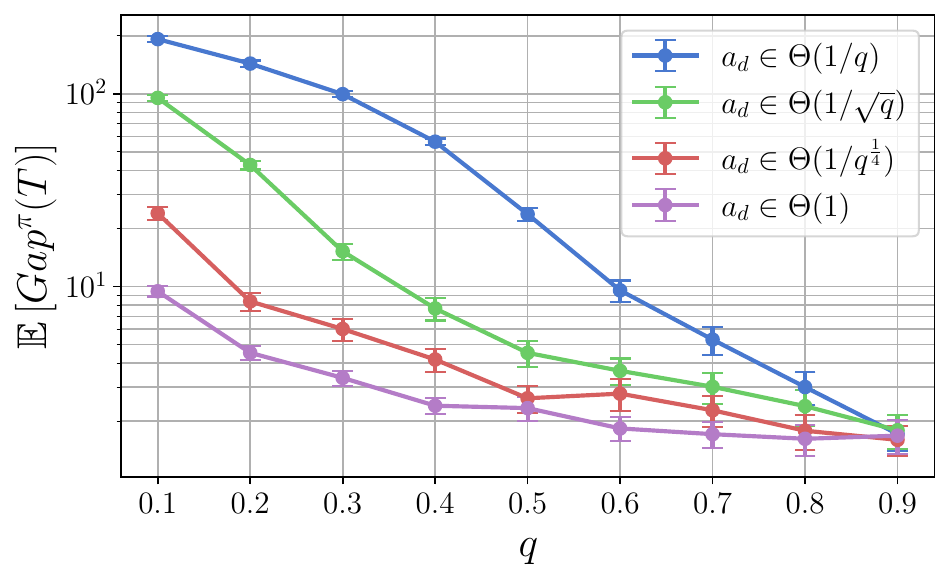} }}%
    \quad 
    \subfloat[\centering {$\mathbb{E}[M^d]$ vs. $q$}\label{fig:q-md}]{{\includegraphics[width=0.45\textwidth]{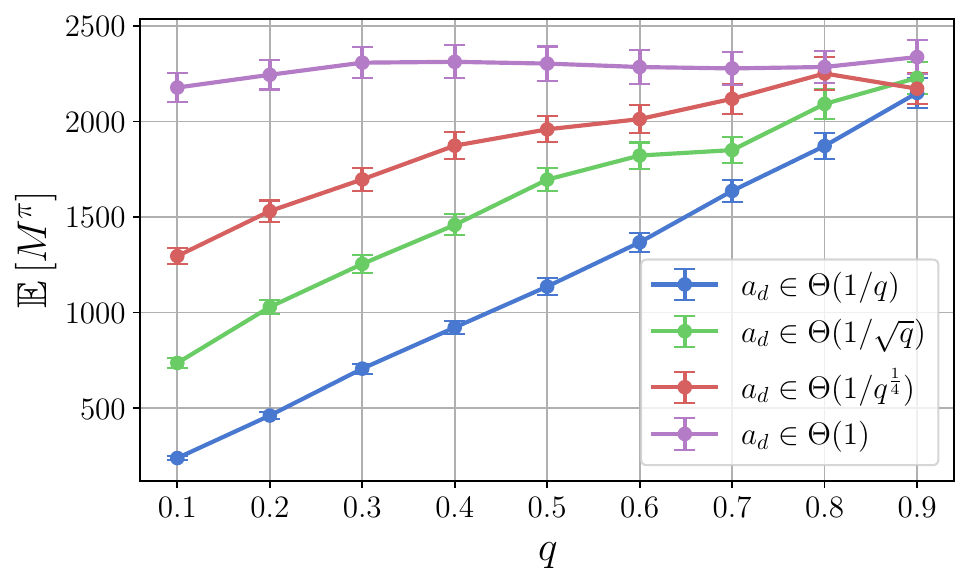} }}%
    \caption{\centering Performance of the semi-dynamic policy for various instantiations of $a_d$ with respect to $q$. Here, $T = 2 \cdot 10^5$, $N = 5$.}
    \label{fig:sensitivity_q}
\end{figure}

\subsection{Opaque Selling}\label{sec:opaque_extensions}

{We now turn to the opaque selling model, and study the robustness of the insights that we derived in \Cref{sec:opaque_salop}. In particular, we study the performance of the semi-dynamic policy when customers (i) may exhibit heterogeneous preferences across products, and the decision-maker implements {\it discriminatory pricing} (DP), i.e., charges different prices for different products, and (ii) may not be risk-neutral. Our numerical results validate the strong performance of the semi-dynamic policy for this much more general setting, but also give rise to the following insights:
\begin{itemize}
\item Though offering the opaque product increases the purchase probability in every period in which it is offered, whether or not it increases the expected revenue depends on customers' idiosyncratic preferences as well as the magnitude of the discount. If it decreases revenue, our results show that the semi-dynamic policy can strongly outperform both the always-flex policy (due to fewer periods with decreased revenue) and traditional selling (due to better-balanced inventory). 
In contrast, when it increases revenue, one would assume that ``more opaque selling'' is better given that it increases both revenue and ensures that inventory is better-balanced. This however is not true as the higher aggregate demand under the always-flex policy may actually increase the inventory costs associated with these additional purchases. {Thus, even when opaque selling increases revenue, the semi-dynamic policy may yet outperform the always-flex policy.}
\item The key requirement for the superior performance of the semi-dynamic policy is that 
the opaque offering be offered {\it strategically late in the replenishment cycle}, rather than simply infrequently: offering the opaque product in each period independently with the same probability as the semi-dynamic policy fails to effectively balance the inventory and thus does not achieve comparable cost savings.
\item The semi-dynamic policy retains its superior performance in the face of risk-averse and risk-seeking customers. As we vary both $\delta$ and customers' risk tolerance, we find that the profit obtained by the semi-dynamic policy consistently exceeds that of the two benchmark policies. Moreover, compared to the always-flex policy, we find that the semi-dynamic policy is more robust to an ill-chosen value of $\delta$.
\end{itemize}
}

\subsubsection*{Numerical setup.} We first specify the behavioral and inventory models used in our experiments. Throughout, we use the notation $\{o\}$ to denote the product corresponding to the opaque option.
\paragraph{Choice model.} {We consider a setting in which the retailer sells $N = 3$ types of products, and assume that each product has a marginal production cost of $\underline{c}$.} There are $L = 3$ types of customers. In each period, a type-$l$ customer arrives with probability $\alpha_l \in (0,1)$. A type-$l$ customer has valuation $V_i^l = \underline{v} + \overline{V}_i^l$ for product $i \in [N]$, where $\underline{v}$ represents customers' intrinsic value for obtaining any product, and $\overline{V}_i^l\sim U(0,1-\underline{v})$ is an idiosyncratic term unique to each type, product pair. For our main set of experiments, we assume customers are risk-neutral, and value the opaque option at $V_o^l = \frac{1}{N}\sum_{i=1}^N V_i^l$. We moreover randomize over ${\mathbf{\alpha}} \in \bracket{\parenthesis{1/3, 1/3, 1/3}, \parenthesis{2/5, 3/10, 3/10}, \parenthesis{1/2, 1/4, 1/4}}.$

To model idiosyncrasies across customers of the same type, we assume customers make their purchase decisions according to the multinomial logit choice model (MNL) \citep{mcfadden1973conditional}. Formally, when the opaque product is not offered, we let $q_{l,i}$ be the probability that a type-$l$ customer purchases product $i$. Then,
\[q_{l,i} = \frac{\exp\left((V_i^l-p_i)/\mu\right)}{1 + \sum_{j\in [N]}\exp\left((V_j^l-p_j)/\mu\right)},\]
where $p_j$ is the price of product $j \in [N]\cup\{o\}$, and $\mu$ is a scale parameter.  

When the opaque option is offered, we denote the purchase probabilities by $q_{l,i}^o \ \forall \ i \in [N]\cup \{o\}$, with:
\[q_{l,i}^o = \frac{\exp\left((V_i^l-p_i)/\mu\right)}{1 + \sum_{j\in [N]\cup \{o\}}\exp\left((V_j^l-p_j)/\mu\right)}.\]

As in our analytical results, we assume the prices of all products are fixed across time; however, we allow the retailer to vary prices across products. Specifically, the retailer sets $\hat{\mathbf{p}} = (\hat{p}_1,\ldots,\hat{p}_N)$ to be a {revenue-maximizing price vector absent the opaque option, i.e., $\hat{\mathbf{p}}\in \arg\max_{\mathbf{p}}\sum_{i\in[N]}(p_i-\underline{c}) q_i,$} where \mbox{$q_i = \sum_{l\in[N]}\alpha_l q_{l,i}$} is the aggregate purchase probability of product $i$ across all customer types.\footnote{In general, solving for $\hat{\mathbf{p}}$ exactly is computationally intractable; we instead approximate $\hat{\mathbf{p}}$ via grid search, enumerating over $\mathbf{p} \in \{0.01,0.02,\ldots,1\}^N$.} {Given $\hat{\mathbf{p}}$, we apply the opaque selling discount $\delta > 0$ to the weighted average price of a product in the absence of the opaque option, letting \mbox{$p_o = \sum_{i\in[N]}\hat{p}_i q_i -\delta$}.} Finally, across all experiments, we let {$\underline{v} = 0.6, \underline{c} = 0, \mu = 0.1$ and $\delta = 0.05$\footnote{{This value of $\delta$ represents a 7.7\% discount relative to $\sum_{i\in[N]}\hat{p}_i q_i$, on average.}}. We defer a sensitivity analysis of our results to various values of $\underline{c}$ to Appendix \ref{apx:numerical-results-balls-into-bins}.} 

\paragraph{Inventory model.} We assume the retailer sets the aggregate initial inventory level across all products according to the classical Economic Order Quantity (EOQ) formula \citep{harris1990many}. Letting $\hat{S}$ denote this total inventory level, we have $\hat{S} := \sqrt{\frac{2DK}{h}}$ (rounded to the nearest integer), where $D$ represents an estimate of aggregate demand across all products. Notice that, under this general model, offering the opaque option may boost product sales, thereby potentially significantly impacting the aggregate demand $D$ if offered frequently. As a result, we consider different values of $D$ depending on the opaque selling policy. For the always-flex policy, we let \mbox{$D = \sum_{l\in[L]}\alpha^l\left(\sum_{i\in[N]\cup\{o\}}q_{l,i}^o\right)$}, since the opaque option is offered in each period. For all other policies, we let $D = \sum_{l\in[L]}\alpha^l\left(\sum_{i\in[N]}q_{l,i}\right)$ be the aggregate expected demand absent the opaque option.\footnote{This estimate of $D$ is exact for the no-flex policy. However, for all other tested policies that offer the opaque option, this will underestimate the expected demand only slightly, given that these policies exercise the opaque option very infrequently. While the retailer will stock suboptimally in these cases, for this value of $\hat{S}$, the profit induced by these policies can then be viewed as a conservative estimate of the profit under the ``optimal'' order-up-to levels.} {To adjust for the heterogeneous purchase probabilities across products, we set the order-up-to level $S_i$ for each product $i\in [N]$ by normalizing $\hat{S}$ by the market share for product $i$. For instance, under the always-flex policy:
\begin{align}\label{eq:si-af}
S_i = \hat{S}\cdot \frac{\sum_{l\in[L]}\alpha^l q_{l,i}^o}{\sum_{l\in[L]}\alpha^l\left(\sum_{j\in[N]}q_{l,j}^o\right)}.
\end{align}
For all other policies, we set $S_i$ by replacing $q_{l,i}^o$ in \Cref{eq:si-af} above with $q_{l,i}$:
\begin{align}\label{eq:si-others}
S_i = \hat{S}\cdot \frac{\sum_{l\in[L]}\alpha^l q_{l,i}}{\sum_{l\in[L]}\alpha^l\left(\sum_{j\in[N]}q_{l,j}\right)}.
\end{align}
}
Finally, we randomly generate \mbox{$K \in \bracket{1,2,\cdots,5}$} and \mbox{$h \in \bracket{0.004,0.008,\cdots, 0.02}$} in our experiments, unless otherwise specified.

\paragraph{Policies.} Throughout our experiments we compare the performance of the semi-dynamic policy to that of the no-flex and always-flex policies. Recall, in the homogeneous setting considered in \Cref{sec:opaque_salop}, the retailer allocates the product with the highest remaining inventory to customers purchasing the opaque product. While this naturally balances inventory levels across time for homogeneous products, the notion of system balancedness changes in the heterogeneous setting we consider here. Specifically, in order to account for the different purchase rates (and stocking levels) of products in this setting, we now reason with respect to the {\it normalized} inventory levels of products, defined as $\bar{z}_i^\pi(t) = \frac{z_i^\pi(t)}{S_i}$ for all $i \in [N]$,  $t \geq 0$. System balancedness, then, naturally corresponds to the normalized inventory levels being equal across all products, achieved by allocating the product with the highest normalized inventory level to a customer purchasing the opaque option. In line with this reasoning, we redefine the gap of the system to take in the normalized inventory levels of all products, i.e., 
\begin{align}\label{eq:inv-gap1_ext}
\overline{\Gap}_I^{d}(t) := \frac{\sum_{i\in[N]} \overline{z}_i^{\pi}(t)}{N} - \min_{i \in [N]} \overline{z}_i^{\pi}(t), \quad \forall\, t \in [\R^\pi].
\end{align}
The heterogeneous analog of the threshold condition in \Cref{alg:semi} is then:
\begin{equation}\label{eq:flexing_condition_ext}
    \overline{\Gap}_I^{d}(t) \geq \frac{a \cdot \squarebracket{\parenthesis{\sum_{i\in[N]}\parenthesis{S_i - 1}} +1 -t}}{\hat{S}},
\end{equation}
where $a \in (0,1]$ is a tuning parameter. In our experiments, we let $a = 0.5$.

\begin{remark}[Equivalence of the semi-dynamic policies in Sections \ref{sec:opaque_salop} and \ref{sec:opaque_extensions}]
For the special case of the Salop circle model, the above semi-dynamic policy and the semi-dynamic policy specified in \Cref{alg:semi} are equivalent. To see this, observe that since $S_i$ is identical across products under the Salop model, allocating the product with the highest normalized inventory level is equivalent to allocating the product with the highest actual inventory level. Moreover, since exactly one product is purchased in each period, under the Salop model we have that $\sum_{i\in[N]}\bar{z}_i^d(t) = \frac{NS-t}{S}$. Letting $a = c_d q_o$ in \Cref{eq:flexing_condition_ext}, we have that the threshold condition is satisfied if and only if:
\begin{align*}
&\frac{S-t/N}{S}-\frac{\min_{i\in[N]}z_i^d(t)}{S} \geq \frac{c_dq_o\parenthesis{N(S-1)+1-t}}{NS} 
\iff S-t/N-\min_{i\in[N]}z_i^d(t) \geq \frac{c_dq_o\parenthesis{T-t}}{N},
\end{align*}
which is precisely the threshold condition used in \Cref{alg:semi}.
\end{remark}

\subsubsection*{Results.} Across all experiments, we run 100 replications over $10^4$ periods for each randomly generated instance, defined by $K, h, {\alpha}$ and $\overline{V}_i^l$, $i \in [N], l \in [L]$. All reported metrics are averaged over these $10^4$ periods and 100 replications.

Before presenting our results, we note that for these instantiations, whenever the opaque product is offered, the aggregate purchase probability across all offerings is on average 6.5\% higher than periods in which the opaque product is not offered. However, since the opaque product is sold at a discount, opaque selling results in an average revenue loss of 3.8\% in periods in which it is offered. As a result, we observe a revenue-inventory cost trade-off over a wide variety of instances, even in the heterogeneous setting.

\paragraph{Profit comparison.} 
We plot the distribution of relative profit improvement of the semi-dynamic policy over the no-flex and always-flex policies in \Cref{fig:summary}. Here, we abuse notation and refer to the better of the no-flex and always-flex policies (profit-wise) as $\max\{\pi^{nf},\pi^a\}$.

\begin{figure}
    \centering
    \subfloat[\centering Relative profit improvement (\%) over $\pi^{nf}$]{{\includegraphics[width=0.45\textwidth]{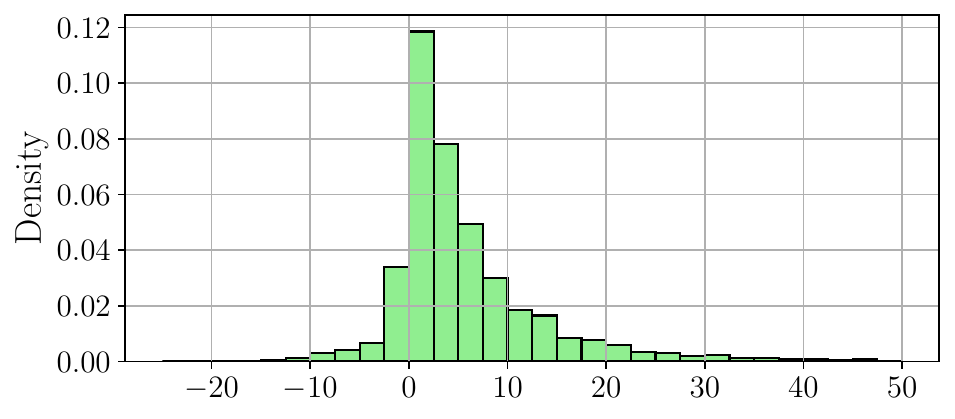}} \label{fig:sum1}}%
    \subfloat[\centering Relative profit improvement (\%) over $\pi^{a}$]{{\includegraphics[width=0.45\textwidth]{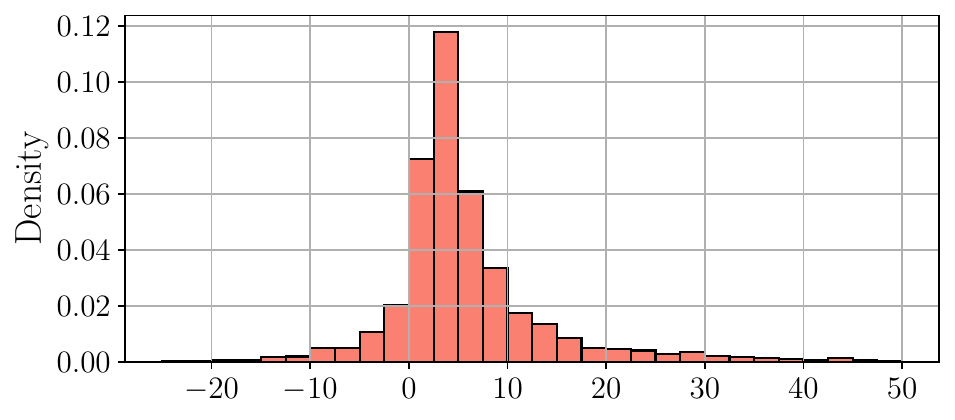} }\label{fig:sum2}}
    \\
    \subfloat[\centering Relative profit improvement (\%) over $\max\{\pi^{nf},\pi^a\}$]{\includegraphics[width=0.91\textwidth]{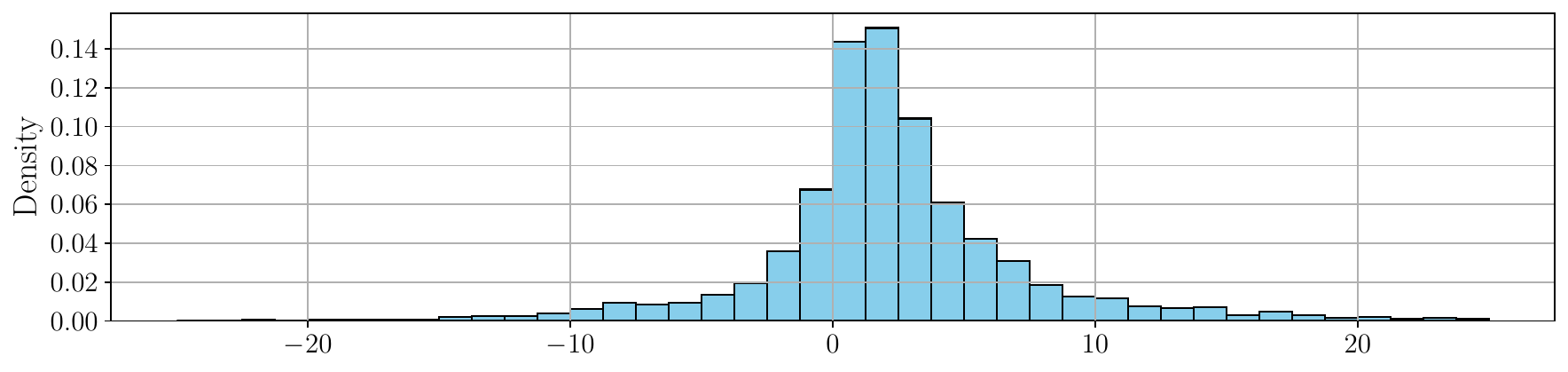} \label{fig:sum3}}
    \caption{\centering Distribution of the relative profit improvement of $\pi^d$ over $\pi^{nf}$, $\pi^a$, and $\max\{\pi^{nf},\pi^a\}$.}
    \label{fig:summary}
\end{figure}

These results strengthen our analytical results by demonstrating that strategically timing opaque selling can yield significant value. In particular, \Cref{fig:sum1} shows that the semi-dynamic policy generates higher profits than the no-flex policy in over 87\% of instances, {with an average relative profit improvement of 5.9\%. In \Cref{fig:sum2}, we observe that it generates higher profits than the always-flex policy in over 88\% of instances, with an average relative profit improvement of 8.4\%. Moreover, the semi-dynamic policy outperforms at least one of the two benchmark policies in 99.1\% of instances, with an average profit gain of 33.6\% relative to the worse of the two. This is due to the fact that the semi-dynamic policy is constructed to efficiently interpolate between the always-flex and no-flex policies: it hedges against the bad instances where (i) the no-flex policy performs poorly due to imbalanced inventory, or (ii) the always-flex policy performs poorly by offering the opaque discount too aggressively. Even when compared to the {\it better} of the two benchmarks, the semi-dynamic policy achieves higher profits in over 76\% of instances, as shown in \Cref{fig:sum3}.

{To better understand the source of these gains, we compare the revenue and inventory cost savings of the semi-dynamic policy relative to the no-flex and always-flex policies in \Cref{tab:comparison}. We observe that the semi-dynamic policy generates 1.4\% less revenue than the no-flex policy on average. This follows from the fact that it offers the opaque product approximately 35\% of the time, with customers purchasing it 17\% of the time (or, approximately half of the time that it is offered). However, the modest revenue drop is more than offset by the semi-dynamic policy's inventory cost savings: its inventory costs are approximately 7.5\% lower than those of the no-flex policy. These savings follow from a lengthening of the average replenishment cycle from 19.2 to 21.1 periods (a 9.7\% increase) thanks to end-of-horizon inventory balancing.

Relative to the always-flex policy, the semi-dynamic policy achieves 2.5\% higher revenue, on average. This follows from the fact that customers purchase the opaque product approximately 47.9\% of the time under the always-flex policy, as compared to 17\% under the semi-dynamic policy. {\it Crucially, we observe that the semi-dynamic policy also incurs lower inventory costs than the always-flex policy.} Specifically, its inventory costs are 2.2\% lower, with savings achieved in approximately 50\% of all instances. This a priori unexpected behavior is due to the fact that, in many instances, the always-flex policy induces customers who would have otherwise chosen the outside option to purchase the opaque product, thereby increasing the aggregate demand across all products. The aggregate order-up-to level $\hat{S}$ increases as a result, yielding higher average inventory levels (approximately 2\% higher than the semi-dynamic policy) and inventory costs overall. These facts together help to explain the significant profit advantage that the semi-dynamic policy has over the always-flex policy.

Finally, we note that while our experiments randomize over the inventory cost parameters $K$ and $h$, in Appendix \ref{apx:numerical-results-balls-into-bins} we investigate the impact of these parameters on the performance of the semi-dynamic policy. Our results echo \Cref{cor:comparison}: while the semi-dynamic policy performs well relative to the no-flex and always-flex policies across a variety of values of $(K,h)$, the gains are the largest as these two grow large, and inventory management becomes a more important part of a retailer's daily operations.
}

\begin{table}
\centering
\begin{tabular}{c|cc|cc}
\toprule
Benchmark & Abs. revenue gain & \% revenue gain & Abs. inventory savings & \% inventory savings \\
\midrule
$\pi^{nf}$ & -0.009 & -1.4 & 0.022 & 7.5\\
$\pi^{a}$ & 0.016 & 2.5 & 0.006 & 2.2 \\
\bottomrule
\end{tabular}
\caption{Comparison of semi-dynamic policy's revenue and inventory costs to no-flex and always-flex policies}
\label{tab:comparison}
\end{table}

\paragraph{Importance of strategically timing the opaque offering.} In order to illustrate the value of {\it strategically} timing the opaque selling offering, we also consider a policy termed the ``flex-$\sqrt{S}$'' policy $\pi^f$, which offers the opaque option with the same probability as the semi-dynamic policy, but does so in an i.i.d., state-independent fashion in each period.\footnote{The flex-$\sqrt{S}$ policy can be viewed as the $N$-opaque selling strategy in \citet{elmachtoub2019value}, where customers purchase the opaque product with probability $\Theta\parenthesis{\frac{1}{\sqrt{S}}}$.} Specifically, for each sample path, we instantiate the flex-$\sqrt{S}$ policy by first simulating the semi-dynamic policy $\pi^d$, and using the probability with which this latter policy offers the opaque product in each period as the flexing probability for $\pi^f$.

We plot the distribution of inventory costs and profit of the semi-dynamic policy relative to the flex-$\sqrt{S}$ policy in \Cref{fig:pif}. ({By construction, these two policies yield the same expected revenue of approximately $0.642$.}) As shown in \Cref{fig:pif-inv}, the semi-dynamic policy yields significantly lower inventory costs, with an average inventory cost saving of 4.7\% relative to $\pi^f$. This result is intuitive, since the decision to offer the opaque product is uncorrelated with inventory under the flex-$\sqrt{S}$ policy, meaning that it may offer the opaque product when inventory levels are balanced, and fail to offer the opaque product in periods when the gap is large. This leads to an average of 6.6\% shorter replenishment cycle lengths compared with $\pid$. \Cref{fig:pif-prof} shows that this has important implications for profit, with the semi-dynamic policy generating higher profit in over 99\% of instances, with an average relative profit improvement of 5.4\%.

\begin{figure}
    \centering
    \subfloat[\centering Inventory cost (\%)]{{\includegraphics[width=0.45\textwidth]{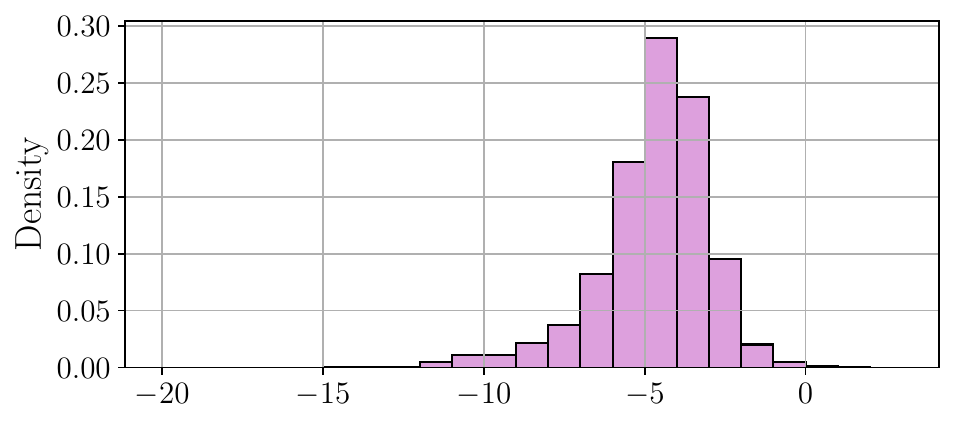} } \label{fig:pif-inv}}%
    \quad 
    \subfloat[\centering Profit improvement (\%)]{{\includegraphics[width=0.45\textwidth]{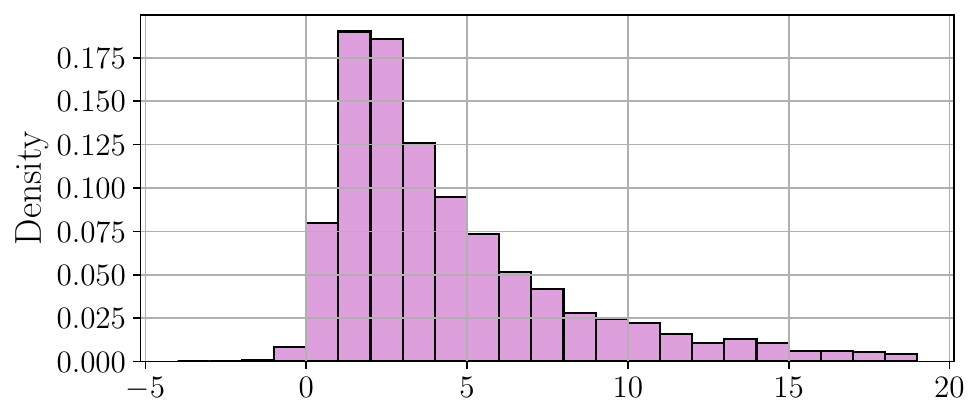} } \label{fig:pif-prof}}%
    \caption{\centering {Histograms illustrating the inventory cost and profit improvement of $\pi^d$ relative to $\pi^f$.}}
    \label{fig:pif}
\end{figure}

\paragraph{Robustness to customer risk preferences.} We next test the robustness of the semi-dynamic policy to risk-averse and risk-seeking behavior. Specifically, we model risk-seeking customers by letting \mbox{$V_o^l = \max_{i \in [N]} V_i^l, \forall \ l \in [L]$}; on the other hand, we model risk-averse customers by letting \mbox{$V_o^l = \min_{i \in [N]} V_i^l, \forall \ l \in [L]$}. 
{We plot the average profits of $\pi^{nf}, \pia$ and $\pi^d$ as a function of $\delta$ under these three behavioral models in \Cref{fig:risk}.}

\Cref{fig:risk} illustrates that, across all three risk preferences, the semi-dynamic policy (represented by the red curve) on average performs at least as well as the no-flex and always-flex policies. This is not only true when comparing the maximum profit of each policy across all values of $\delta$, but also when comparing profits for almost any {\it fixed} value of $\delta \in \{0.001,0.01,0.02,\ldots,0.1\}$. 

Beyond the superior performance of the semi-dynamic policy, our results also illustrate how (i) customers' risk preferences and (ii) the choice of $\delta$ affect the performance of the always-flex and semi-dynamic policies. 
As customers become more risk-seeking, the cost of flexibility decreases, and one would expect the profit of both policies to increase. However, our results display that this only holds for carefully chosen~$\delta$: for small values of $\delta$ it is indeed the case that both policies yield greater profits the more risk-seeking the customers are. In contrast, for large values of $\delta$ both policies incur significant loss under more risk-seeking customers, which is due to the revenue loss from customers choosing the opaque option without further increasing the inventory cost savings. Finally, our results show, for all risk preferences, that the profit loss of the semi-dynamic policy due to an ill-optimized value of $\delta$ is much smaller than that of the always-flex policy. In other words, the semi-dynamic policy is not only more profitable but also more robust than the always-flex policy.

\begin{figure}
    \centering
    {{\includegraphics[width=1\textwidth]{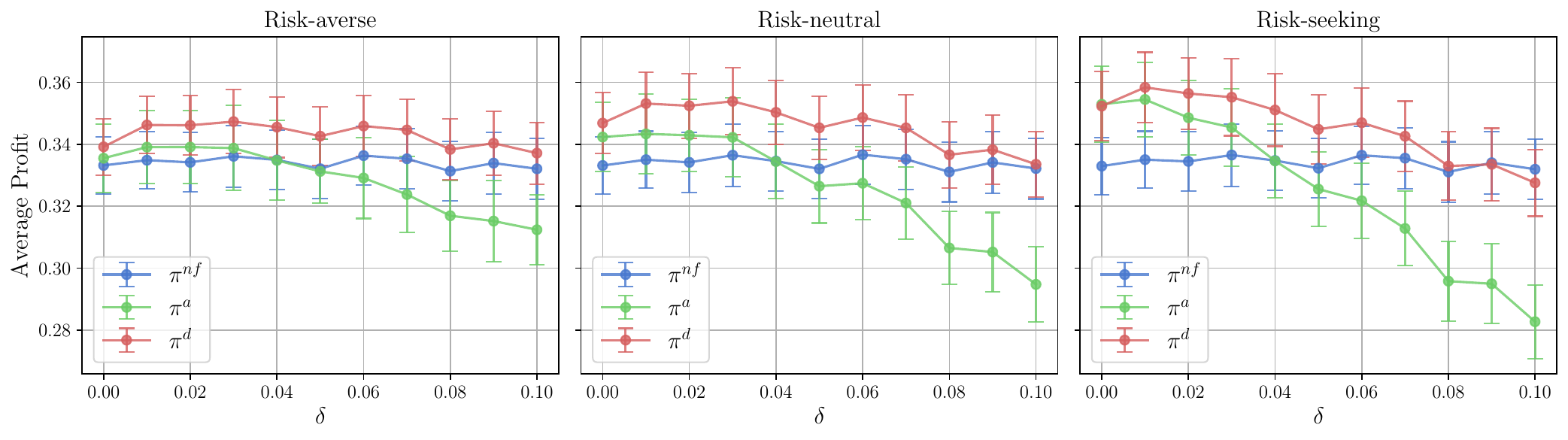} }}%
    \caption{\centering Profit of $\pi^{nf}, \pia$ and $\pi^d$ under risk-averse, risk-neutral, and risk-seeking behavior as a function of $\delta$. Here we report the average profit and the interquartile range (25th–75th percentiles) across all instances. Note that the average profit under $\pi^{nf}$ has no dependence on the behavioral model, since the opaque product is never offered under $\pi^{nf}$; we plot the performance of this policy for benchmarking purposes.}
    \label{fig:risk}
\end{figure}

%% file: conclusion.tex
\section{Conclusion}\label{sec:conclusion}

{In this work we considered {\it end-of-horizon} load balancing problems in settings where load balancing across resources is costly to the decision-maker. Using the canonical balls-into-bins paradigm, we demonstrated the power of late-stage flexibility, designing an algorithm that achieves approximate end-of-horizon balance by exerting flexibility the minimum number of times, as the horizon $T$ grows large. Building on these algorithmic insights, we then illustrated the power of late-stage opaque selling strategies for inventory management. Such policies achieve the optimal trade-off between the revenue loss incurred by opaque discounts and the resulting inventory cost savings that the opaque product generates due to load balancing. We moreover demonstrated that these policies can dominate classical benchmarks that either blindly offer or withhold the opaque product in each period.

Our work opens several avenues for future exploration. {Within the context of opaque selling, possible extensions of our work include the joint optimization of pricing and opaque selling policies, as well as extensions of our algorithms to settings in which there are nonstationarities in (i) arriving demand's preferences for resources, and (ii) the probability with which arriving demand is flexible.}  Additionally, as discussed in the introduction, late-stage flexibility may prove valuable in {many} {other} applications, including workforce scheduling, warehouse operations, and e-commerce fulfillment. Though our model provides {useful} parsimonious insights for how to enable flexibility in these settings, {it would be interesting to extend our model to capture application-specific idiosyncrasies}. For example, in dynamic workforce scheduling, regulations may {place restrictions on shifts workers may be assigned to}; similarly, there may be correlations between delivery windows that customers find acceptable in e-commerce fulfillment. Although these types of constraints are unlikely to fundamentally affect the validity of the design principles that our work suggests, {incorporating them into models remains practically relevant}. {Moreover, the theoretical analysis of policies in further constrained settings is also likely to pose significant new challenges.}


}

%% file: appendix.tex
\newpage

\input{balls-appendix.tex}

{
\input{opaque-appendix.tex}

\input{balls-into-bins-numerics.tex}

}

%% file: balls-appendix.tex
\section{The Balls-Into-Bins Model}

\subsection{Useful Bounds} \label{sec:auxi}
In this section we prove a series of facts that will be useful to us throughout.

The first lemma presents Binomial tail bounds that we use throughout our analysis. We defer its proof to Appendix \ref{apx:chernoff}.
\begin{lemma} \label{lem:chernoff}
Consider two binomial random variables $X_i(t) \sim B(t,p_i), X_j(t) \sim B(t,p_j)$, where $X_i(t)$ and $X_j(t)$ need not be independent. 
Then, for any $\epsilon > 0$:
\begin{enumerate}[label=(\roman*)]
    \item  $\mathbb{P}\left(\left|(X_j(t) - X_i(t)) - \mathbb{E}[X_j(t) - X_i(t)]\right| \geq \epsilon t\right) \leq 4e^{-\epsilon^2t/2}$, 
    \item $\mathbb{P}\left(\left|(X_j(t) - X_i(t)) - \mathbb{E}[X_j(t) - X_i(t)]\right| \geq \epsilon \sqrt{t \log t}\right) \leq 4 t^{-\epsilon^2/2}.$ 
\end{enumerate}
\end{lemma}

\medskip

\Cref{cl:ball_lb} next establishes that, for a policy that never exercises the flex option, there is a constant probability that the gap of the system exceeds $\Theta(\sqrt{T})$ by the end of the horizon. We defer its proof, which is a direct application of the Berry-Esseen Theorem, to Appendix \ref{apx:ball_lb}.

\begin{lemma}\label{cl:ball_lb}
For the no-flex policy $\pin$, for any $a>0,$ there exists a constant $a' > 0$ such that, for large enough $T$, $$\mathbb{P}\left(\Gapnf(T) \geq a \sqrt{T}\right) \geq a'.$$
\end{lemma}

\medskip

\Cref{lem:e1ij} bounds the likelihood that the loads of the maximally and minimally loaded bins never intersect during any set of periods over which a policy flexes consecutively. {We defer the proof of the lemma to Appendix \ref{apx:e1ij}.}
\begin{lemma}\label{lem:e1ij}
Consider any policy $\pi$ such that (i) $\flexaction^\pi(t) = 0$ for $t \leq t_1$, and (ii) $\flexaction^\pi(t) = 1$ for $t_1 < t \leq t_2$, where $t_2-t_1 \geq a_s\sqrt{t_2\log t_2}$. Moreover, let $F_{ij}^1$ be the event that $i$ and $j$ are respectively the maximally and minimally loaded bins in period $t_2$, and that their loads have never intersected between $t_1$ and $t_2$. Formally,  $$F_{ij}^1 = \left\{\history{t_2} \in \historyset{t_2}\mid i = \arg\max_{k\in[N]} \xp_k(t_2), j = \arg\min_{k\in[N]} \xp_k(t_2), \xp_i(t) \neq \xp_j(t), \forall \ t \in \{t_1, \ldots , t_2\}\right\}.$$ Then, there exist a constant $\alpha > 0$ such that: 
{\[\mathbb{P}(F_{ij}^1) \leq 4 {t}_1^{-1} + e^{-\alpha (t_2 - t_1)} \ \forall \ i,j \in [N].\]}
\end{lemma}

\smallskip

\Cref{lem:e1ij_semi} further tightens this bound, conditioned on the loads of bins $i$ and $j$ being within a constant factor of $t_2-t_1$ of each other when the policy begins flexing consecutively. {We defer its proof to Appendix \ref{apx:e1ij_semi}.}
\begin{lemma}\label{lem:e1ij_semi}
{Fix random variables $t_1 < t_2$, and} consider any policy $\pi$ such that (i) $\flexaction^\pi(t) = 0$ for $t \leq t_1$, and (ii) $\flexaction^\pi(t) = 1$ for $t_1 < t \leq t_2$. Let $F_{ij}^1$ be as in \Cref{lem:e1ij}, i.e.,
\begin{align*}F_{ij}^1 = \left\{\history{t_2} \in \historyset{t_2} \mid i = \arg\max_{k\in[N]} \xp_k(t_2), j = \arg\min_{k\in[N]} \xp_k(t_2), \xp_i(t) \neq \xp_j(t), \forall \ t \in \{t_1, \ldots , t_2\}\right\}.\end{align*} 
For any constants $\policycst \leq \frac{1}{5 \binom{N}{2}}$ and $a > 0$, there exist constants $\alpha_0 > 0$ and $t_0 > 0$ such that \[\PP{F_{ij}^1 \, | \, \xp_i(t_1) - \xp_j(t_1) \leq \policycst(t_2-t_1)q + a{, t_1, t_2, t_2 - t_1 \geq t_0}} \leq 5 e^{-\alpha_0(t_2 - t_1)}.\]
\end{lemma}

\medskip

\Cref{lem:e2ij} now bounds the likelihood that, for any such ``consecutively-flexing'' policy described above, the maximum and minimally loaded bins have different loads, given the amount of time that has elapsed since the last time they intersected during the flexing period. We defer its proof of Appendix \ref{apx:e2ij}.
\begin{lemma}\label{lem:e2ij}
Fix random variables $t_1 < t_2$, and consider any policy $\pi$ such that (i) $\flexaction^\pi(t) = 0$ for $t \leq t_1$, and (ii) $\flexaction^\pi(t) = 1$ for $t_1 < t \leq t_2$. Let $F_{ij}^2$ be the event that $i$ and $j$ are respectively the (strictly) maximally and minimally loaded bins in period $t_2$, and that their loads intersected between $t_1$ and $t_2-1$. Formally,
\begin{align*}
    F_{ij}^2 := \{&\history{t_2} \in \historyset{t_2}\mid i = \arg\max_{k\in[N]} \xp_k(t_2), j = \arg\min_{k\in[N]} \xp_k(t_2), \xp_i(t_2) \neq \xp_j(t_2), \notag\\
    &\hspace{2.5cm}\xp_i(t) = \xp_j(t) \text{ for some } t \in \{t_1, \ldots , t_2-1\}\}.
\end{align*}
Moreover, let $\tau := \max \left\{t \, \mid \, \xp_i(t) = \xp_j(t), t \in \{t_1, \ldots , t_2-1\}\right\}$ be the last time the loads of these two bins intersected.
Then, there exists a constant $\alpha_2 > 0$ such that:
\[\mathbb{P}(F_{ij}^2 \ \mid t_1, t_2, \tau) \leq 5 e^{-\alpha_2 (t_2 - \tau)}\quad \forall \ i,j \in [N].\]
\end{lemma}

\medskip

\subsubsection{Proof of \Cref{lem:chernoff}}\label{apx:chernoff}

\begin{proof}{Proof.}
We first prove (i). We have:
\begin{align}\label{eq:chernoff}
         &   \mathbb{P}\left(\left|(X_j(t) - X_i(t)) - \mathbb{E}[X_j(t) - X_i(t)]\right| \geq \epsilon t\right) \notag \\
        &\leq \mathbb{P}\left(\left|(X_j(t) - \EE{X_j(t)}\right| + \left|X_i(t) - \EE{X_i(t)}\right| \geq \epsilon t\right) \notag \\
        &\leq  \mathbb{P}\left(\left|X_j(t) - \mathbb{E}[X_j(t)]\right| \geq \frac{\epsilon}{2} t\right) + \mathbb{P}\left(\left|X_i(t) - \mathbb{E}[X_i(t)]\right| \geq \frac{\epsilon}{2} t\right),
\end{align}
where the second inequality follows from a union bound.
By Hoeffding's inequality~\citep{dubhashi2009concentration}, for all $k \in [N]$,
\begin{align*}
      \mathbb{P}\left(\left|X_k(t) - \mathbb{E}[X_k(t)]\right| \geq \frac{\epsilon}{2} t\right)\leq 2\exp\left(-\frac{\epsilon^2t}{2}\right).
\end{align*}

Plugging this back into \eqref{eq:chernoff}, we obtain:
\begin{align*}
    \mathbb{P}\left(\left|(X_j(t) - X_i(t)) - \mathbb{E}[X_j(t) - X_i(t)]\right| \geq \epsilon t\right) \leq 4e^{-\epsilon^2t/2}.
\end{align*}

For (ii), we similarly have that
\begin{align}\label{eq:lem-part-2}
            &\mathbb{P}\left(\left|(X_j(t) - X_i(t)) - \mathbb{E}[X_j(t) - X_i(t)]\right| \geq \epsilon \sqrt{t \log t}\right) \notag \\
        &\leq \mathbb{P}\left(\left|X_j(t) - \mathbb{E}[X_j(t)]\right| \geq \frac{\epsilon}{2} \sqrt{t \log t}\right) + \mathbb{P}\left(\left|X_i(t) - \mathbb{E}[X_i(t)]\right| \geq \frac{\epsilon}{2} \sqrt{t \log t }\right).
\end{align}
As before, by Hoeffding's inequality, for all $k \in [N]$,
\begin{align*}
    \mathbb{P}\left(\left|X_k(t) - \mathbb{E}[X_k(t)]\right| \geq \frac{\epsilon}{2} \sqrt{t \log t}\right) &\leq 2\exp\left(-\frac{\epsilon^2 t\log t}{2t}\right)
    = 2t^{-\epsilon^2/2}.
\end{align*}
We thus obtain:
\begin{align*}
     \mathbb{P}\left(\left|(X_j(t) - X_i(t)) - \mathbb{E}[X_j(t) - X_i(t)]\right| \geq \epsilon \sqrt{t \log t}\right)\leq 4 t^{-\epsilon^2/2}.
\end{align*}
\end{proof}

\subsubsection{Proof of \Cref{cl:ball_lb}}\label{apx:ball_lb}
\begin{proof}{Proof.}
We have:
\begin{align*}
    \mathbb{P}(\Gapnf(T) \geq a \sqrt{T}) &= \mathbb{P}\left(\max_{i'} x_{i'}(T) - \frac{T}{N} \geq a \sqrt{T}\right)
    \geq \mathbb{P}\left(x_1(T)- \frac{T}{N} \geq a \sqrt{T}\right)
    = \mathbb{P}\left(\frac{x_1(T)- \frac{T}{N}}{\sqrt{T} \sigma} \geq \frac{a}{\sigma}\right),
\end{align*}
where $\sigma = \frac{1}{N}(1-\frac{1}{N}).$ {By the Berry{-}Esseen Theorem (Theorem 3.4.17 in \citet{durrett2019probability}):}
\begin{align*}
\mathbb{P}\left(\frac{x_1(T)- \frac{T}{N}}{\sqrt{T} \sigma} \geq \frac{a}{\sigma}\right) \geq 1-\Phi\left(\frac{a}{\sigma}\right)-\frac{b}{\sqrt{T}}\notag \geq a',
\end{align*}
for some constants $a'$, $b$ and large enough $T$.
\end{proof}

\medskip

\subsubsection{Proof of \Cref{lem:e1ij}}\label{apx:e1ij}
The proofs of \Cref{lem:e1ij,lem:e1ij_semi,lem:e2ij} rely on a class of \emph{fictional} allocation rules that we couple to any flexing policy $\pi$. A fictional allocation rule is parameterized by bins $i$ and $j$, and is denoted by $\mathcal{A}_{ij}$. This allocation rule mimics $\mathcal{A}^{\pi}$ before $t_1$. After $t_1$, it similarly mimics $\mathcal{A}^\pi$ in all periods such that the flex option cannot be exercised by $\pi$ (i.e., $f(t)\omega^\pi(t) = 0$). In periods when $\pi$ can exercise the flex option (i.e., $f(t)\omega^\pi(t) = 1$), $\mathcal{A}_{ij}$ also makes the same decisions as $\pi$, except in two scenarios: (i) $\mathcal{F}(t) = \{i,j\}$ or (ii) $\mathcal{F}(t) = \{k,j\}$ for some $k \neq i$, with $x_{k}^\pi(t) \geq x_i^\pi(t)$. In both of these cases, it allocates the ball to bin $j$. At a high level, $\mathcal{A}_{ij}$ favors bin $j$ over bin $i$ {regardless of the relative loads in bin $i$ and $j$. This allows us to analyze the allocation decision in any period $t > t_1$ independent of $x_j^\pi(t)$}. We formalize this fictional allocation rule below.

\begin{definition} \label{def:fic_allocation}
For $i, j \in [N]$, the fictional allocation rule $\mathcal{A}_{ij}$ is defined as follows. For all $t \leq t_1$, \mbox{$\mathcal{A}_{ij}(t) = \mathcal{A}^\pi(t)$}.  For all $t > t_1$:
$$
\mathcal{A}_{ij}(t):= \begin{cases}
    j \quad &\text{if } \flex{t}\flexaction^\pi(t) = 1 \text{ and } \flexset{t} = \{i,j\}, \\
    j \quad &\text{if } \flex{t}\flexaction^\pi(t) = 1, \flexset{t} = \{k,j\} \text{ for some } k \neq i \text{ and }\arg\min_{j' \in \{i,k\}} x^{\pi}_{j'}(t) = i, \\
    \mathcal{A}^{\pi}(t) &\text{otherwise}.
    \end{cases}
$$
\end{definition}

The following lemma provides probabilistic bounds on the number of flex balls thrown into $i$ and $j$ under $\mathcal{A}_{ij}$. We defer its proof to the end of this section.{
\begin{lemma}\label{lem:flexes}
Let $\mathbf{x}^\pi(t_1) = (x_k^\pi(t_1))_{k \in [N]}$ denote the loads across different bins at the end of period $t_1$, and suppose the fictional allocation rule runs for $\bar{t}$ periods after $t_1$. Additionally, for $k \in [N]$, let $Y_k(\bar{t})$ be the number of flex balls placed into bin $k$ between $t_1$ and $t_1+\bar{t}$ under $\mathcal{A}_{ij}$. Then, there exists a constant $\alpha > 0$ such that, for any $\mathbf{x}^\pi(t_1)$:
\begin{enumerate}[label=(\roman*)]
    \item $\mathbb{E}[Y_j(\bar{t})-Y_i(\bar{t}) \mid \mathbf{x}^\pi(t_1)] \geq \frac{q}{\binom{N}{2}}\bar{t}$,
    \item $\mathbb{P}\left(Y_j(\bar{t})-Y_i(\bar{t}) \leq \frac{q}{2 \binom{N}{2}}\bar{t} \mid \mathbf{x}^\pi(t_1)\right) \leq e^{-\alpha \bar{t}}$.
\end{enumerate}
\end{lemma}}

\medskip

\begin{proof}{Proof of \Cref{lem:e1ij}.}
We prove the claim via the coupling between $\mathcal{A}^\pi$ and $\mathcal{A}_{ij}$. We first introduce some additional notation. For bin $k \in [N]$, let $x_k^{\text{fic}}(t)$ denote the load in bin $k$ and time~$t$ under the {\it fictional} allocation policy $\mathcal{A}_{ij}$. We let $Y_k$ denote the number of flex balls that land in bin~$k$ between $t_1+1$ and $t_2$ under $\mathcal{A}_{ij}$, and define $Y := \sum_{k = 1}^N Y_k$. Note that $Y\sim B\left(t_2-t_1,q\right)$. {We use $\overline{T} = t_2-Y$ to denote the total number of random (non-flex) throws throughout $\{1,2,...,t_2\}$, and let $Z_k$ be the number of balls that landed in bin $k$ during the $\overline{T}$ random trials. Thus, for $k\in [N]$, $x^{\text{fic}}_k(t_2) = Y_k + Z_k.$}

Note that, under $F_{ij}^1$, $\xp_i(t) > \xp_j(t) \ \forall \ t \in \{t_1,...,t_2\}$. By \Cref{def:fic_allocation}, this implies that both $\pi$ and the fictional allocation rule pick bin $j$ whenever $\mathcal{F}(t) = \{i,j\}$. On the other hand, whenever $\mathcal{F}(t) = \{k,j\}$ for some $k \neq i$ such that $\arg\min_{j' \in \{i,k\}}x_{j'}^\pi(t) = i$, it must be the $x_k^\pi(t) \geq x_i^\pi(t) > x_j^\pi(t)$. Therefore, both $\pi$ and the fictional allocation rule would place the ball in bin $j$ in this case as well. As a result, the two rules make identical decisions in $\{t_1, .., t_2\}$. Since they also make identical decisions prior to $t_1$, it must be that $x^{\text{fic}}_{k}(t) = \xp_k(t) \ \forall \ t \in \{t_1,...,t_2\}$. Thus, we have:
\begin{align}\label{eq:fictitious}
    F_{ij}^1 &= \Bigg\{\history{t_2} \in \historyset{t_2}\mid \xp_i(t_2) = \max_{k} \xp_{k}(t_2), \, \xp_{j}(t_2) = \min_{k} \xp_{k}(t_2), \, \xp_i(t) > \xp_j(t) \ \forall \ t \in \{t_1,...,t_2\}, \notag \\
     &\hspace{3.3cm} x_{i}^{\text{fic}}(t) > x_{j}^{\text{fic}}(t) \, \forall \, t \in \{t_1,...,t_2\} \Bigg\}\notag \\
     &\subseteq \left\{\history{t_2} \in \historyset{t_2} \mid x_i^{\text{fic}}(t) > x_j^{\text{fic}}(t) \, \forall \, t \in \{t_1,...,t_2\}\right\} \notag \\
     &\subseteq \left\{Y_i + Z_i > Y_j + Z_j\right\}.
\end{align}

As a result, it suffices to bound the likelihood that, by $t_2$, the fictional allocation rule has allocated more balls to bin $i$ than to bin $j$. Namely:
\begin{align}
    \PP{F_{ij}^1} &\leq \PP{Y_i + Z_i > Y_j + Z_j} 
    \leq \PP{Z_i - Z_j \geq \frac{q}{2{\binom{N}{2}}} (t_2-t_1)} + \PP{Y_j - Y_i \leq \frac{q}{2{\binom{N}{2}}} (t_2-t_1)} \label{eq:two_union}.
\end{align}
Note that the last inequality holds because we need at least one of $$Z_i - Z_j \geq \frac{q}{2{\binom{N}{2}}} (t_2-t_1) \text{ and } Y_j - Y_i \leq \frac{q}{2{\binom{N}{2}}} (t_2-t_1)$$ for $Y_i + Z_i > Y_j + Z_j$ to hold.

{Recall, $t_2-t_1 \geq \staticcst \sqrt{t_2\log t_2}$ by assumption. Thus, for $\overline{T} \in \{t_1,...,t_2\}$ we have:
\begin{align}
   \PP{Z_i - Z_j \geq \frac{q}{2{\binom{N}{2}}} (t_2-t_1) \mid \overline{T} = t}  &\leq \PP{Z_i - Z_j \geq \frac{q}{2{\binom{N}{2}}} \staticcst \sqrt{t_2\log t_2}\mid \overline{T} = t} \notag \\
   &\leq  \PP{Z_i - Z_j \geq \frac{q}{2{\binom{N}{2}}} \staticcst \sqrt{\overline{T}\log \overline{T}} \mid \overline{T} = t} \notag \\
   &\leq \PP{Z_i - Z_j \geq \sqrt{2}\sqrt{\overline{T}\log \overline{T}}\mid \overline{T} = t}, \label{eq:def-of-a-static}
  \end{align}
where \eqref{eq:def-of-a-static} follows from plugging in $\staticcst \geq \frac{2\sqrt{2}{N\choose 2}}{q}$.
  
Recall, $Z_i, Z_j$ respectively denote the number of {\it random} balls that landed in bins $i$ and $j$ during $\overline{T}$ periods. Thus, conditioned on $\overline{T}$, $Z_i$ and $Z_j$ are both binomially distributed, with parameters $(\overline{T},1/N)$. Applying \cref{lem:chernoff} (ii) to \eqref{eq:def-of-a-static}, with $\epsilon = \sqrt{2},$ we then have: 
\begin{align}
    \PP{Z_i - Z_j \geq \frac{q}{2{\binom{N}{2}}} (t_2-t_1) \mid \overline{T} = t} \leq 4 t^{-1} \leq 4 t_1^{-1}, \label{eq:applying-chernoff}
\end{align}
since $\overline{T} \geq t_1$ almost surely.}
{
}

Moreover, applying \cref{lem:flexes} (ii) to $\PP{Y_j - Y_i \leq \frac{q}{2{\binom{N}{2}}} (t_2-t_1)}$, with \mbox{$\bar{t}:= t_2-t_1$} we have:
\begin{align}\label{eq:y}
\mathbb{P}\left(Y_j - Y_i \leq \frac{q}{2 \binom{N}{2}}(t_2-t_1) \right)\leq e^{-\alpha(t_2-t_1)},
\end{align}
for some constant $\alpha > 0.$\footnote{Here we abuse notation when using $Y_j$ and $Y_i$ above in omitting their dependency on $(t_2-t_1).$} Plugging {\eqref{eq:applying-chernoff} and \eqref{eq:y}} into \eqref{eq:two_union}, we obtain the result.
\end{proof}

\medskip

We now prove \Cref{lem:flexes}, the key lemma driving this result.
\begin{proof}{Proof of \cref{lem:flexes}.}
We abuse notation and let $M = \sum_{t = 1}^{\bar{t}} \mathbbm{1}\{\flexset{t} = \{i,j\}\}$ be the number of times the flex set is $\{i,j\}$ during $\bar{t}$ flexible throws. We have:
\begin{align*}
        \mathbb{E}[Y_j(\bar{t})-Y_i(\bar{t})\mid M, \mathbf{x}^\pi(t_1)]
         &= \mathbb{E}[Y_j(\bar{t})\mid M, \mathbf{x}^\pi(t_1)] -\mathbb{E}[Y_i(\bar{t})\mid M, \mathbf{x}^\pi(t_1)]\\
         &= M + \sum_{t=1}^{\bar{t}} \mathbb{P}\left(\{\flexset{t} = \{j,k\} \text{ for some } k \neq i\} \cap \{\mathcal{A}_{ij}(t) = j\}\mid M, \mathbf{x}^\pi(t_1)\right)\\&\qquad-\sum_{t=1}^{\bar{t}}\mathbb{P}\left(\{\flexset{t} = \{i,k\} \text{ for some } k \neq j\} \cap \{\mathcal{A}_{ij}(t) = i\}\mid M, \mathbf{x}^\pi(t_1)\right),
\end{align*}
where the second equality follows from the fact that, by construction, whenever $\flexset{t} = \{i,j\}$ (this happens $M$ times, by definition), the ball is allocated to bin $j$ under allocation rule $\mathcal{A}_{ij}$. 
{By symmetry, $i$ and $j$ are equally likely to be included in $\mathcal{F}(t)$. Combining this with the $\mathcal{A}_{ij}$ construction, which places a ball into bin $j$ whenever the policy would have placed it into $i$ rather than $k$, we have:
\begin{align*}
&\mathbb{E}[Y_j(\bar{t})-Y_i(\bar{t})\mid M, \mathbf{x}^\pi(t_1)] \geq M  \\
\implies &\mathbb{E}[Y_j(\bar{t})-Y_i(\bar{t}) \mid \mathbf{x}^\pi(t_1)] = \mathbb{E}[\mathbb{E}[Y_j(\bar{t})-Y_i(\bar{t})|M, \mathbf{x}^\pi(t_1)]] \geq \mathbb{E}[M] = \frac{q}{\binom{N}{2}}\bar{t},
\end{align*}
}
and we obtain \cref{lem:flexes} (i).

\medskip

We now prove \cref{lem:flexes} (ii). Let {$W_{kt}$} be the random indicator variable representing whether a ball in period $t$ is a flex ball that goes into bin $k$, for $k \in [N], t\in [\bar{t}]$. Then, by definition, $Y_k(\bar{t}) = \sum_{t = 1}^{\bar{t}}W_{kt}$. Define moreover  $Z_\tau = \left(\sum_{t = 1}^{\tau} W_{jt} - W_{it}\right) - \frac{q}{\binom{N}{2}}\tau$ for $\tau = 1, 2, ..., \bar{t}$, with $Z_0 = 0$.
Then,
\begin{align}\label{eq:get-to-submartingale}
    \mathbb{P}\left(Y_j(\bar{t})-Y_i(\bar{t}) \leq \frac{q}{2 \binom{N}{2}}\bar{t} \mid \mathbf{x}^\pi(t_1)\right) &= \mathbb{P}\left(\sum_{t = 1}^{\bar{t}} (W_{jt} - W_{it}) \leq \frac{q}{2 \binom{N}{2}}\bar{t} \mid \mathbf{x}^\pi(t_1)\right) \notag 
    \\ &=\mathbb{P}\left(Z_{\bar{t}} - Z_0 \leq -\frac{q}{2 \binom{N}{2}}\bar{t} \mid \mathbf{x}^\pi(t_1)\right).
\end{align}

In order to bound \eqref{eq:get-to-submartingale}, we argue that, conditioned on $\mathbf{x}^\pi(t_1)$, the sequence $Z_0,Z_1,Z_2,\ldots$ is a sub-martingale. To see this, note that:
\begin{align}\label{eq:submartingale}
    \mathbb{E}[Z_{\tau+1} \mid Z_0, ..., Z_\tau, \mathbf{x}^\pi(t_1)] &=\mathbb{E}\left[\left(\sum_{t = 0}^{\tau+1} W_{jt} - W_{it}\right) - \frac{q}{\binom{N}{2}}(\tau+1)\; \bigg{\vert}\; Z_0, ..., Z_\tau, \mathbf{x}^\pi(t_1)\right] \notag \\
    &= \mathbb{E}\left[W_{j,\tau+1} - W_{i,\tau+1} -\frac{q}{\binom{N}{2}}  + Z_\tau \; \bigg{\vert} \; Z_0, ..., Z_\tau, \mathbf{x}^\pi(t_1)\right] \notag \\
    &= \mathbb{E}\left[W_{j,\tau+1} - W_{i,\tau+1} \; \bigg{\vert} \; Z_0, ..., Z_\tau, \mathbf{x}^\pi(t_1)\right] -\frac{q}{\binom{N}{2}}  + Z_\tau.
\end{align}

Note that a flex ball lands in bin $j$, i.e., $W_{jt}-W_{it} = 1,$ if one of two events occurs: (i) \mbox{$\flexset{t} = \{i,j\}$}, in which case the ball is always thrown into bin $j$, or (ii) $\flexset{t} = \{j,k\}$ for some $k \neq i,j$, and the ball is thrown into bin $j$. The first event occurs with probability $\frac{q}{\binom{N}{2}}$, and the second with probability $\mathbb{P}\left(\{\flexset{t} = \{j,k\} \text{ for some } k \neq i\} \cap \{\mathcal{A}_{ij}(t) = j\}\right).$ As a result:
\begin{align*}
&\PP{W_{j,\tau+1}-W_{i,\tau+1} = 1 \mid Z_0, ..., Z_\tau, \mathbf{x}^\pi(t_1)} \\ &= \frac{q}{\binom{N}{2}} + \mathbb{P}\left(\{\flexset{\tau+1} = \{j,k\} \text{ for some } k \neq i\} \cap \{\mathcal{A}_{ij}(\tau+1) = j\} \mid Z_0, ..., Z_\tau, \mathbf{x}^\pi(t_1)\right).
\end{align*}
Via similar reasoning, it follows that 
\begin{align*}
&\PP{W_{j,\tau+1}-W_{i,\tau+1} = -1 \mid Z_0, ..., Z_\tau, \mathbf{x}^\pi(t_1)}\\
&= \mathbb{P}\left(\{\flexset{\tau+1} = \{i,k\} \text{ for some } k \neq j\} \cap \{\mathcal{A}_{ij}(\tau+1) = i\} \mid Z_0, ..., Z_\tau, \mathbf{x}^\pi(t_1)\right).
\end{align*}

Plugging this into \eqref{eq:submartingale}:
\begin{align}
   & \mathbb{E}[Z_{\tau+1} \mid Z_0, ..., Z_\tau, \mathbf{x}^\pi(t_1)] \notag \\& =\frac{q}{\binom{N}{2}} + \mathbb{P}\left(\{\flexset{\tau + 1} = \{j,k\} \text{ for some } k \neq i\} \cap \{\mathcal{A}_{ij}(\tau + 1) = j\} \mid Z_0, ..., Z_\tau, \mathbf{x}^\pi(t_1)\right) \notag\\
    &\quad - \mathbb{P}\left(\{\flexset{\tau + 1} = \{i,k\} \text{ for some } k \neq j\} \cap \{\mathcal{A}_{ij}(\tau + 1) = i\} \mid  Z_0, ..., Z_\tau, \mathbf{x}^\pi(t_1)\right) - \frac{q}{\binom{N}{2}} + Z_\tau \notag \\
    &\geq Z_\tau,
\end{align}
where the final inequality follows from the same arguments as those used to derive \cref{lem:flexes} (i) above. Therefore, conditioned on $\mathbf{x}^\pi(t_1)$, $(Z_t)_{t\in\{1,\ldots,\bar{t}\}}$ is a submartingale. Applying Azuma's inequality \citep{chung2006concentration} to \eqref{eq:get-to-submartingale}, we obtain:
\begin{align*}
    \mathbb{P}\left(Y_j(\bar{t})-Y_i(\bar{t}) \leq \frac{q}{2 \binom{N}{2}}\bar{t} \mid \mathbf{x}^\pi(t_1)\right) = \mathbb{P}\left(Z_{\bar{t}} - Z_0 \leq -\frac{q}{2 \binom{N}{2}}\bar{t} \mid \mathbf{x}^\pi(t_1)\right) \leq e^{-\frac{2 \left(\frac{q}{2 \binom{N}{2}}\bar{t}\right)^2}{\bar{t}}} \leq e^{-\alpha \bar{t}},
\end{align*}
for some constant $\alpha > 0$.
\end{proof}

\subsubsection{Proof of \Cref{lem:e1ij_semi}}\label{apx:e1ij_semi}
\begin{proof}{Proof.}
Let $t_0 = \frac{2a}{q\left(\frac{1}{2{N\choose 2}}-2a_p\right)}$. In this proof we similarly make use of the fictional allocation rule $\mathcal{A}_{ij}$ (see \cref{def:fic_allocation}).

Fix bin $k$, and let $Y_k$ denote the number of flex balls that go into bin $k$ between $t_1 + 1$ and~$t_2$ as a result of $\mathcal{A}_{ij}$. We moreover let $Y := \sum_{k = 1}^N Y_k$. Let $Z_k$ denote the number of balls that land in bin $k$ during the random (non-flexible) throws between $t_1 + 1$ and~$t_2$. Thus, for~$k\in [N]$, the total number of balls that land in bin $k$ is $x^{\text{fic}}_k(t_2) := \xp_k(t_1) + Y_k + Z_k.$

Moreover, under $F_{ij}^1$, we have $\xp_i(t) > \xp_j(t) \, \forall \,  t \geq t_1 + 1$. By the same arguments as those used in \Cref{lem:e1ij}, $\pi$ and the fictional allocation rule make identical decisions in $\{t_1,\ldots,t_2\}$, with $x_k^\pi(t) = x_k^{\text{fic}}(t)$ for all $t$. Moreover, it suffices to bound the probability that $x_{i}^{\text{fic}}(t_2) > x_j^{\text{fic}}(t_2)$. Namely,
\begin{align}
    &\PP{F_{ij}^1\, |\, \xp_i(t_1) - \xp_j(t_1) \leq \policycst(t_2-t_1)q + a, t_1, t_2, t_2 - t_1 \geq t_0} \notag\\
    &\leq  \PP{x_i^{\text{fic}}(t_1) + Z_i + Y_i > x_j^{\text{fic}}(t_1) + Z_j + Y_j\, |\, x_i^{\text{fic}}(t_1) - x_j^{\text{fic}}(t_1) \leq \policycst(t_2-t_1)q + a, t_1, t_2, t_2 - t_1 \geq t_0}\notag \\
    &=  \mathbb{P}\left(x_i^{\text{fic}}(t_1) - x_j^{\text{fic}}(t_1)> \left(Z_j- Z_i\right) + (Y_j-Y_i)\, | \, x_i^{\text{fic}}(t_1) - x_j^{\text{fic}}(t_1) \leq \policycst(t_2-t_1)q + a, t_1, t_2, t_2 - t_1 \geq t_0\right) \notag \\
    &\leq  \mathbb{P}\left(\policycst(t_2-t_1)q + a> \left(Z_j- Z_i\right) + (Y_j-Y_i) \mid x_i^{\text{fic}}(t_1) - x_j^{\text{fic}}(t_1) \leq \policycst(t_2-t_1)q + a, t_1, t_2, t_2 - t_1 \geq t_0\right) \notag \\
    &\leq  \mathbb{P}\left(Z_i - Z_j \geq \policycst(t_2-t_1)q + a \mid x_i^{\text{fic}}(t_1) - x_j^{\text{fic}}(t_1) \leq \policycst(t_2-t_1)q + a, t_1, t_2, t_2 - t_1 \geq t_0\right)\notag \\
    &\qquad + \mathbb{P}\left(Y_j-Y_i \leq 2\left(\policycst(t_2-t_1)q + a\right) \mid x_i^{\text{fic}}(t_1) - x_j^{\text{fic}}(t_1) \leq \policycst(t_2-t_1)q + a, t_1, t_2, t_2 - t_1 \geq t_0\right)\label{eq:conditional-eij-1},
\end{align}
where the last inequality comes from the fact that at least one of \begin{align*}\left\{Z_i - Z_j \geq \policycst(t_2-t_1)q + a\right\} \quad \text{ and } \quad \left\{Y_j-Y_i \leq 2\left(\policycst(t_2-t_1)q + a\right)\right\}\end{align*}
must hold for $\policycst(t_2-t_1)q + a> \left(Z_j- Z_i\right) + (Y_j-Y_i)$ to be true.

Recall, $Z_i$ and $Z_j$ respectively denote the number of {\it random} balls that land in bins $i$ and $j$ between $t_1+1$ and $t_2$. Thus, $Z_i$ and $Z_j$ are both binomially distributed, with $t_2-t_1$ trials and $p_i = p_j = \frac{1-q}{N}$. We then have:
\begin{align*}
&\mathbb{P}\left(Z_i - Z_j \geq \policycst(t_2-t_1)q + a \mid x_i^{\text{fic}}(t_1) - x_j^{\text{fic}}(t_1) \leq \policycst(t_2-t_1)q + a, t_1, t_2, t_2 - t_1 \geq t_0 \right)\\
&= \mathbb{P}\left(Z_i - Z_j \geq \policycst(t_2-t_1)q + a \mid t_1, t_2, t_2 - t_1 \geq t_0 \right)\\ &=  \mathbb{P}\left(\left(Z_i - Z_j\right) - \mathbb{E}[Z_i-Z_j] \geq \policycst(t_2-t_1)q + a \mid t_1, t_2, t_2 - t_1 \geq t_0\right)\\
&\leq \mathbb{P}\left(\left(Z_i - Z_j\right) - \mathbb{E}[Z_i-Z_j] \geq \policycst(t_2-t_1)q \mid t_1, t_2, t_2 - t_1 \geq t_0\right)\\
&\leq 4 e^{-\alpha_1(t_2-t_1)},
\end{align*}
for some constant $\alpha_1 > 0$, by \cref{lem:chernoff} (i) (where we let $\epsilon = \policycst q$.)

{We now analyze the probabilistic bound on $Y_j - Y_i$.} {Since $\policycst \leq \frac{1}{5 \binom{N}{2}}$ and $t_0 = \frac{2a}{q\left(\frac{1}{2{N\choose 2}}-2a_p\right)}$,} for \mbox{$t_2 - t_1 \geq t_0$}, we have $2\left(\policycst(t_2-t_1)q + a\right) \leq \frac{q}{2\binom{N}{2}}(t_2-t_1)$. Thus,
\begin{align*}
&\mathbb{P}\left(Y_j-Y_i \leq 2\left(\policycst(t_2-t_1)q + a\right)\mid x_i^{\text{fic}}(t_1) - x_j^{\text{fic}}(t_1) \leq \policycst(t_2-t_1)q + a, t_1, t_2, t_2 - t_1 \geq t_0\right)\\
&\leq \mathbb{P}\left(Y_j - Y_i \leq \frac{q}{2 \binom{N}{2}}(t_2-t_1)\mid x_i^{\text{fic}}(t_1) - x_j^{\text{fic}}(t_1) \leq \policycst(t_2-t_1)q + a, t_1, t_2, t_2 - t_1 \geq t_0\right)\\
&\leq e^{-\alpha_2(t_2-t_1)}.
\end{align*}
for some constant $\alpha_2 > 0$, {where the second inequality follows from \cref{lem:flexes} (ii).} 

Plugging these two bounds back into~\eqref{eq:conditional-eij-1}, and letting $\alpha_0 = \min\{\alpha_1,\alpha_2\}$, we obtain the result.
\end{proof}

\medskip

\subsubsection{Proof of \Cref{lem:e2ij}}\label{apx:e2ij}
\begin{proof}{Proof.}
We again use the $\mathcal{A}_{ij}$ coupling (see \Cref{def:fic_allocation}). Conditional on $t_1, t_2$ and $\tau,$ we fix bin $k$ and {re-define}~$Y_k$ to be the number of flex balls that go into bin $k$ between $\tau+1$ and $t_2$ as a result of $\mathcal{A}_{ij}$. {As before, we let} $Y = \sum_{k = 1}^N Y_k$. Note that $Y\sim B\left(t_2-\tau,q\right)$. Moreover, let $Z_k$ denote the number of balls that landed in bin $k$ from the random (non-flex) throws between $\tau+1$ and $t_2$. Thus, for $k\in [N]$, the total number of balls that land in bin $k$ from $\tau+1$ to $t_2$ is $x'_k(t_2) := Y_k + Z_k.$

Conditional on $F_{ij}^2$, $\xp_i(t) > \xp_j(t) \ \forall \ t \in \{\tau + 1,\ldots,t_2\}$. Therefore, using the same arguments as those used in the proof of \Cref{lem:e1ij}, under event $F_{ij}^2$, $\mathcal{A}_{ij}$ and $\mathcal{A}^\pi$ make identical decisions for all periods in $\{\tau+1, .., t_2\}$. Moreover, it suffices to bound the probability of the event that the fictional allocation rule placed more balls in bin $i$ than in bin $j$ between $\tau+1$ and $t_2$. We have:
\begin{align}\label{eq:union2}
    \PP{F_{ij}^2\mid t_1, t_2,\tau} &\leq \PP{Y_i + Z_i > Y_j + Z_j\mid t_1, t_2,\tau} \notag \\
    &\leq \PP{Z_i - Z_j > \frac{q}{2 {\binom{N}{2}}} (t_2-\tau)\mid t_1, t_2,\tau} + \PP{Y_j - Y_i < \frac{q}{2{\binom{N}{2}}} (t_2-\tau)\mid t_1, t_2,\tau}.
\end{align}

Recall, $Z_i, Z_j$ respectively denote the number of {\it random} balls that landed in bins~$i$ and~$j$ from~$\tau+1$ to~$t_2$. Thus, $Z_i$ and $Z_j$ are both binomially distributed, with $t_2-\tau$ trials and $p_i = p_j = (1-q)/N$. Applying \Cref{lem:chernoff} (i) to $Z_i$ and $Z_j$, with $\epsilon = \frac{q}{2{\binom{N}{2}}}$, we obtain:
\begin{align}
   \PP{Z_i - Z_j > \frac{q}{2{\binom{N}{2}}} (t_2-\tau) \, \bigg{|} \, t_1, t_2,\tau}
   &\leq  4\exp\left(-\frac{q^2}{8{\binom{N}{2}}^2} (t_2 - \tau)\right). \label{ineq5}
\end{align}

By \cref{lem:flexes} (ii), we also have:
\begin{align*}\PP{Y_j - Y_i < \frac{q}{2{\binom{N}{2}}} (t_2-\tau)\mid t_1, t_2,\tau} \leq e^{-\alpha (t_2 -\tau)}\end{align*}
for some constant $\alpha > 0$. Plugging these two bounds into \eqref{eq:union2} and taking $\alpha_2 = \min\left\{\frac{q^2}{8{\binom{N}{2}}^2},\alpha\right\}$, we obtain the lemma.
\end{proof}

\medskip

\subsection{Lower Bounds}

\subsubsection{Proof of \texorpdfstring{\cref{thm:ball_lb}}{Lg}} \label{app:ball_lb}

\begin{proof}{Proof.}
We construct a ``best-case'' policy $\pi$ as follows. Consider the no-flex policy which never exercises the flex option, and, for any history $\sigma(T) \in \Sigma(T)$ induced by this policy, let $i = \argmax_{i'}x_{i'}^{nf}(T)$, breaking ties lexicographically. Suppose now that $\pi$ has hindsight information on the realization of the $T$ random trials, and, for any $t$ such that $\preferred{t} = i$, $\pi$ has the freedom to move the ball from $i$ to any other bin $j \neq i$, in such a way that always decreases $\Gapnf(T)$ by 1.

We lower bound the expected gap of $\pi$ as follows:
\begin{align*}
 \mathbb{E}\left[\Gapp(T)\right] &\geq \mathbb{E}\left[\Gapp(T)\mid \Gapnf(T) \geq 2 \sqrt{T}\right] \cdot \mathbb{P}\left(\Gapnf(T) \geq 2 \sqrt{T}\right) \\
 &\geq \mathbb{E}\left[\Gapp(T)\mid \Gapnf(T) \geq 2 \sqrt{T}\right] \cdot a',
\end{align*}
for large enough $T$, where the second inequality follows from \Cref{cl:ball_lb}. Suppose $M^\pi \leq \sqrt{T}$. Then, $\pi$ decreases the gap under the no-flex policy by at most $\sqrt{T}$, yielding:
\begin{align*}
    \mathbb{E}\left[\Gapp(T)\right] 
    &\geq \mathbb{E}\left[\Gapnf(T)-\sqrt{T}\mid \Gapnf(T) \geq 2 \sqrt{T}\right] \cdot a'\geq \sqrt{T}a',
\end{align*}
for large enough $T$. 

Therefore, in sample paths for which $\Gapnf(T) \geq 2\sqrt{T}$, we must have $M^\pi > \sqrt{T}$ in order for $\mathbb{E}[\Gapp(T)] \in \mathcal{O}(1)$.  Thus, we necessarily have: $$\mathbb{E}[\mofp] > \mathbb{E}\left[\mofp\mid \Gapnf(T) \geq 2 \sqrt{T}\right] \cdot \mathbb{P}\left(\Gapnf(T) \geq 2 \sqrt{T}\right) \geq \sqrt{T} \cdot a',$$ for large enough $T$.
\end{proof}

\medskip

\subsubsection{Lower bound for deterministic policies} \label{app:static_tight}

\begin{proposition}\label{prop:static_tight}
For any policy $\pi$, at most one of the following two facts holds:
\begin{enumerate}
    \item $M^\pi \leq a\sqrt{T}$ almost surely, for some $a > 0$, or
    \item $\mathbb{E}\left[\Gap^\pi(T)\right] \in o\left(\sqrt{T}\right).$
\end{enumerate}
\end{proposition}

\begin{proof}{Proof.}
The proof of the result is similar to that of \Cref{thm:ball_lb}. In particular, we prove the claim by contradiction. Suppose there exists a policy $\pi$ such that $M^\pi \leq a \sqrt{T}$ for some constant $a > 0$ and \mbox{$\mathbb{E}[\Gap^\pi(T)] \in o(\sqrt{T}).$}

When no flexing is involved, by \Cref{cl:ball_lb}, $\mathbb{P}\left(\Gapnf(T) \geq (a+1) \sqrt{T}\right) \geq a',$ for some constant $a' > 0$ and large enough $T.$ Now, fix a particular history $\history{T} \in \historyset{T}$ such that $\Gapnf(T) \geq (a+1) \sqrt{T},$ 
and consider \mbox{$i \in \argmax_{i'}x_{i'}(T)$}. By definition, $\Gapnf(T) \geq (a+1) \sqrt{T}$ implies \mbox{$\sum_{t = 1}^T \mathbbm{1}\{\preferred{t} = i\} \geq T/N+ (a+1)\sqrt{T}.$} 
In the best case, $\pi$ knows the exact realizations of the $T$ random balls, i.e., knows $\history{T}$, and
{, for all $t$ such that $\preferred{t} = i$ and $\flexaction(t)=1$, can place the ball that would have gone into bin $i$ into any bin $j$ that would decrease $\Gapp(T)$ by 1}. Then, if $M^\pi \leq a \sqrt{T}$, $\pi$ can relocate at most $a \sqrt{T}$ balls from bin $i$ to the other bins and thus  
\begin{align*}
    \mathbb{E}\left[\Gap^{\pi}(T)\right] &\geq \mathbb{E}\left[\Gap^{\pi}(T)\mid \Gapnf(T) \geq (a+1) \sqrt{T}\right] \cdot \mathbb{P}\left(\Gapnf(T) \geq (a+1) \sqrt{T}\right)\\
    &\geq \mathbb{E}\left[\Gapnf(T)-a \sqrt{T}\mid \Gapnf(T) \geq (a+1) \sqrt{T}\right] \cdot a'\\
    &\geq \sqrt{T} \cdot a'
\end{align*}
for large enough $T$, which contradicts the fact that $\mathbb{E}[\Gapp(T)] \in o\left(\sqrt{T}\right)$.
\end{proof}

\subsection{Analysis of the Static Policy}\label{app:ball_proof}

\subsubsection{Proof of \texorpdfstring{\cref{thm:ball_static}}{Lg}} \label{app:ball_proof_static}

We leverage the auxiliary lemmas in Appendix \ref{sec:auxi} to prove \cref{thm:ball_static} next.

\begin{proof}{Proof of \cref{thm:ball_static}.}
Let $\That = T - \staticcst \sqrt{T \log T}$. By construction, the static policy starts flexing at $\That+1$. Let $E$ denote the event that the gap is non-zero at $T$. 
Then,
\begin{align}\label{eq:main}
    \mathbb{E}[\Gaps(T)] = \mathbb{E}[\Gaps(T)\mid E] \mathbb{P}(E) + \mathbb{E}[\Gaps(T)\mid E^c] \mathbb{P}(E^c)
= \mathbb{E}[\Gaps(T)\mid E] \mathbb{P}(E),
\end{align}
where the second equality follows from the fact that $\Gaps(T) = 0$ given $E^c$, by definition.

Let $E^1$ denote the event that the maximally and minimally loaded bins at time $T$ never had the same loads between $\That$ and $T$, and $E^2$ the event that they do. Letting $\tau$ denote the last period in which the loads of the maximally and minimally loaded bins at time $T$ were equal, we have:
\begin{align}\label{eq:main1}
    \mathbb{E}[\Gaps(T)] &= \EE{\Gaps(T) \mid E^1} \PP{E^1} + \EE{\Gaps(T) \mid E^2}\PP{E^2} \notag \\
   & \leq T\PP{E^1} + \EE{T-\tau \mid E^2}\PP{E^2},
\end{align}
where the inequality uses the fact that $\Gaps(T) \leq T$ trivially. Moreover, letting $i$ and $j$ respectively denote the maximally and minimally loaded bins at time $T$, the bound on $\Gaps(T)$ under $E^2$ follows from:
\begin{align*}
\Gaps(T) = x_i(T)-T/N \leq x_i(T)-x_j(T) &= x_i(T)-x_i(\tau)+x_j(\tau)-x_j(T)+x_i(\tau)-x_j(\tau)\\
&\leq T-\tau,
\end{align*}
since in the worst case the maximally loaded bin at time $T$ received all balls between $\tau+1$ and~$T$.

We further partition event $E$ as follows. {For $i, j \in [N]$, we use $E_{ij}^1$ to denote the event that (i) $i$ and $j$ are respectively the maximally and minimally loaded bins at time $T$, and (ii) the loads of $i$ and $j$ are never the same between $\That$ and $T$. We moreover use $E_{ij}^2$ to denote the event that (i) $i$ and $j$ are respectively the maximally and minimally loaded bins at time $T$ and do not have the same loads, but (ii) their loads were the same at some point between $\That$ and $T-1$. Formally:}
\begin{align}
E_{ij}^1 := \{& \history{T} \in \historyset{T} \mid i = \arg\max_{k\in[N]} \xstatic_k(T), j = \arg\min_{k\in[N]} \xstatic_k(T), \xstatic_i(t) \neq \xstatic_j(t), \forall t \in \{\That, \ldots , T\}\} \label{eq:eij1} \\
E_{ij}^2 := \{& \history{T} \in \historyset{T} \mid i = \arg\max_{k\in[N]} \xstatic_k(T), j = \arg\min_{k\in[N]} \xstatic_k(T), \xstatic_i(T) \neq \xstatic_j(T), \notag\\
&\hspace{2.5cm}\xstatic_i(t) = \xstatic_j(t) \text{ for some } t \in \{\That, \ldots , T-1\}\} \label{eq:eij2}
\end{align}

Union bounding \eqref{eq:main1} over $i$ and $j$, we have: {
\begin{align}
    \EE{\Gaps(T)} &\leq T \sum_{i\neq j} \PP{E_{ij}^1} + \sum_{i\neq j} \EE{T-\tau \mid E_{ij}^2} \PP{E_{ij}^2} \notag\\
    &\leq T \sum_{i\neq j} \PP{E_{ij}^1} + \sum_{t = 1}^{\lceil T-\That \rceil} \sum_{i\neq j} \EE{T-\tau \mid E_{ij}^2, T - \tau = t} \PP{E_{ij}^2, T - \tau = t} \notag\\
    &= T \sum_{i\neq j} \PP{E_{ij}^1} + \sum_{t = 1}^{\lceil T-\That \rceil} \sum_{i\neq j} t \ \PP{E_{ij}^2, T-\tau = t} \notag\\
    &\leq T\sum_{i\neq j} \PP{E_{ij}^1} + \sum_{t = 1}^{\lceil T-\That \rceil} \sum_{i\neq j} t\ \PP{E_{ij}^2 \mid T-\tau = t}.\label{eq:main_simple}
\end{align}}
{Applying \cref{lem:e1ij} to $E_{ij}^1$, with $t_1 = \That$ and $t_2 = T$, there exists a constant $\alpha > 0$ such that:
\begin{align}\label{eq:e1ij_main}
\mathbb{P}(E_{ij}^1) &\leq 4 \That^{-1} + e^{-\alpha(T-\That)}.
\end{align}}
Moreover, applying \cref{lem:e2ij} to $E_{ij}^2$, with $t_1 = \That$ and $t_2 = T,$ there exists a constant $\alpha_2 > 0$ such that: 
\begin{align}\label{eq:e2ij_main}
\mathbb{P}(E_{ij}^2\mid T-\tau = t) = \mathbb{P}(E_{ij}^2\mid \That, T, T-\tau = t) \leq 5 e^{-\alpha_2 t} \ \forall \ i,j.
\end{align}

{Summing \eqref{eq:e1ij_main} and \eqref{eq:e2ij_main} over all $i,j \in [N]$ and plugging this back into \eqref{eq:main_simple}, we then obtain that, for large enough $T$,
\begin{align}\label{eq:thm1-last-bound}
    \EE{\Gaps(T)} \leq & \sum_{i\neq j} T \cdot \parenthesis{4 \That^{-1} + e^{-\alpha(T-\That)}} + \sum_{i\neq j} \sum_{t = 1}^{\lceil T-\That \rceil} 5te^{-\alpha_2 t} \notag \\
    &\leq N^2 \left(\frac{4T}{T-a_s\sqrt{T\log T}} + Te^{-\alpha\cdot a_s\sqrt{T\log T}} + \sum_{t = 1}^{\lceil T-\That \rceil} 5te^{-\alpha_2 t}\right).
\end{align}}
Note that, for all $t \geq 1$, $t \leq \frac{2e^{-1}}{\alpha_2}e^{\alpha_2t/2}$. Therefore:
\begin{align}\label{eq:simplify_lemma5}
\sum_{t = 1}^{\lceil T-\That \rceil} 5te^{-\alpha_2t} &\leq \sum_{t=1}^T\frac{10e^{-1}}{\alpha_2}e^{-\alpha_2t/2} \leq \alpha_4,
\end{align}
for some positive constant $\alpha_4 > 0$, since $e^{-ax}$ is summable for all $a > 0$. Noting moreover that {$\frac{T}{T-a_s\sqrt{T\log T}}$ and} $Te^{-\alpha\cdot a_s\sqrt{T\log T}}$ are upper bounded by constants, and plugging these two facts back into \eqref{eq:thm1-last-bound}, we obtain the result.
\end{proof}

\medskip

\subsection{Analysis of the Semi-Dynamic Policy}

\subsubsection{Proof of \texorpdfstring{\cref{thm:ball_semi2}}{Lg}} \label{app:ball_proof_semi_2}

\begin{proof}{Proof.}
Fix a constant $a \in (0,1]$, and define $\Ta := \inf\left\{t:\Gapnf(t) \geq \frac{a (T-t)q}{N}\right\}$. We will show that, for any such $a$, $\mathbb{E}[T-\Ta]\in \mathcal{O}\left(\sqrt{T}\right)$. This will then immediately imply $\mathbb{E}[M^d] \in \mathcal{O}(\sqrt{T})$, since \mbox{$\Gap^d(t) = \Gap^{nf}(t)$} for all $t \leq \Tstar$. 

{We have:
\begin{align}\label{eq:conditioning}
\mathbb{E}[T-\Ta]
&\leq \sum_{k = 1}^{\left\lceil\sqrt{T}\right\rceil} \mathbb{E}\left[T-\Ta\mid T-\Ta \in [(k-1) \sqrt{T},k \sqrt{T})\right] \notag \cdot \mathbb{P}\left(T-\Ta \in [(k-1) \sqrt{T},k \sqrt{T})\right) \notag \\
&\leq \sum_{k = 1}^{\left\lceil\sqrt{T}\right\rceil} k \sqrt{T} \cdot \mathbb{P}\left(T-\Ta \in [(k-1) \sqrt{T},k \sqrt{T})\right) \notag \\
&=\sqrt{T} \underbrace{\sum_{k=1}^{\lceil \sqrt{T} \rceil} \mathbb{P}\left(T-\Ta \geq (k-1)\sqrt{T}\right)}_{(I)}.
\end{align}

Therefore, it suffices to show that $(I)$ is upper bounded by a constant. The following lemma uses the Berry-Esseen theorem to bound the likelihood that $T-T_a \geq (k-1)\sqrt{T}$, for all $k \in \{1,\ldots,\lceil{\sqrt{T}\rceil}\}$. We defer its proof to Appendix \ref{apx:eta-k}.
\begin{lemma}\label{lem:eta-k}
For all $k \in \{1,\ldots,\lceil{\sqrt{T}}\rceil\}$, there exists a constant $b_6 > 0$ such that:
\begin{align}\label{eq:eta-k}
\mathbb{P}\left(T-T_a\geq (k-1)\sqrt{T}\right) \leq N\left(1-\Phi\left(\frac{a(k-1)q}{\sqrt{N-1}}\right)+ \frac{b_6}{\sqrt{T}}\right),
\end{align}
\end{lemma}
Summing \eqref{eq:eta-k} over all $k$, we obtain: 
\begin{align*}
(I) = \sum_{k = 1}^{\left\lceil\sqrt{T}\right\rceil} \mathbb{P}\left(T-\Ta \geq (k-1) \sqrt{T}\right)
&\leq  \frac{Nb_6}{\sqrt{T}}\left(\sqrt{T}+1\right) +  N\underbrace{ \sum_{k = 1}^{\left\lceil\sqrt{T}\right\rceil} \left(1-\Phi\left(\frac{a(k-1)q}{\sqrt{N-1}}\right)\right)}_{(II)}
\end{align*}
where it remains to be shown that $(II)$ can also be upper bounded by a constant. To see this, note that:
\begin{align*}
(II) \leq \sum_{k = 0}^{\infty} \left(1-\Phi\left(\frac{aq}{\sqrt{N-1}} \cdot k \right)\right)
&\leq \frac{\sqrt{N-1}}{aq} \int_{x = -aq/\sqrt{N-1}}^{\infty} \left(1-\Phi\left(x\right) \right)dx\\
&\leq \frac{\sqrt{N-1}}{aq} \int_{x = -1}^{\infty} \left(1-\Phi\left(x\right) \right)dx,
\end{align*}
since $aq/\sqrt{N-1}<1$.

Using the fact that $\int_{x = -1}^{\infty} \left(1-\Phi\left(x\right) \right)dx \leq 1 + \int_{x = 0}^{\infty} \left(1-\Phi\left(x\right) \right)dx = 1+ \sqrt{\frac{1}{2 \pi}}$, we obtain \mbox{$(II) \leq \frac{\sqrt{N-1}}{aq} \parenthesis{1+ \sqrt{\frac{1}{2 \pi}}}$}. 
Thus, 
\begin{align}\label{eq:tight_analysis}
    \mathbb{E}[T-\Ta] &\leq\parenthesis{N b_6\left(1+\frac{1}{\sqrt{T}}\right) + N \frac{\sqrt{N-1}}{aq} \parenthesis{1+ \sqrt{\frac{1}{2 \pi}}}} \cdot \sqrt{T} \in \mathcal{O}\parenthesis{\sqrt{T}}. \notag 
\end{align}
}
\end{proof}

\medskip

\subsubsection{Proof of \texorpdfstring{\cref{thm:ball_semi}}{Lg}} \label{app:ball_proof_semi_1}

\begin{proof}{Proof.}
Let \mbox{$E = \{ \history{T} \in \historyset{T}\mid \Gapd(T) \neq 0\}$} be the event that the gap is strictly positive at the end of the horizon. As in the proof of \cref{thm:ball_static}, it suffices to show $\mathbb{E}[\Gapd(T)\mid E] \mathbb{P}(E) \in \mathcal{O}(1)$. To do so, we analyze events $E_{ij}^1$ and $E_{ij}^2$, for all $i, j \in [N]$. As before, $E_{ij}^1$ corresponds to the event that $i$ and $j$ are respectively the maximally and minimally loaded bins in period $T$, and that their loads never intersected between $\Tstar$ and $T$. $E_{ij}^2$ corresponds to the event that this same $i$ and $j$ intersected for some $\tau \geq \Tstar$, but had different loads in period $T$. Formally:  
\begin{align}
E_{ij}^1 := \{&\history{T} \in \historyset{T}\mid i = \arg\max_{k\in[N]} \xsemi_k(T), j = \arg\min_{k\in[N]} \xsemi_k(T), \xsemi_i(t) \neq \xsemi_j(t), \forall \ t \in \{\Tstar, \ldots , T\}\} \\ 
E_{ij}^2 := \{&\history{T} \in \historyset{T}\mid i = \arg\max_{k\in[N]} \xsemi_k(T), j = \arg\min_{k\in[N]} \xsemi_k(T), \notag\\
&\hspace{2.5cm}\xsemi_i(t) = \xsemi_j(t) \text{ for some } t \in \{\Tstar, \ldots , T-1\},\xsemi_i(T) \neq \xsemi_j(T)\} \label{eq:e2ij_def}
\end{align}
Via the same arguments as those used in the proof of \Cref{thm:ball_static}, we have: 
\begin{align}\label{eq:thm2_main}
    \mathbb{E}[\Gapd(T)\mid E] \mathbb{P}(E)
&\leq \sum_{i \neq j} \parenthesis{\mathbb{E}[\Gapd(T)\mid E_{ij}^1] \mathbb{P}(E_{ij}^1) + \mathbb{E}[\Gapd(T)\mid E_{ij}^2] \mathbb{P}(E_{ij}^2)},
\end{align}

We next analyze \mbox{$\mathbb{E}[\Gapd(T)\mid E_{ij}^2, \Tstar] \mathbb{P}(E_{ij}^2 \mid \Tstar)$} and \mbox{$\mathbb{E}[\Gapd(T)\mid E_{ij}^1, \Tstar] \mathbb{P}(E_{ij}^1 \mid \Tstar)$} for a fixed first flexing time $\Tstar$. Under event $E_{ij}^2$, $\tau := \max\{t \in \{\Tstar,\ldots, T-1\} \mid x_i^d(t) = x_j^d(t)\}$ be the last time that $\xsemi_i(t) = \xsemi_j(t)$. By \Cref{lem:e2ij}, we have:
\begin{align}\label{eq:use-this-next}
\mathbb{P}(E_{ij}^2\mid \Tstar, T - \tau = t) = \mathbb{P}(E_{ij}^2\mid \Tstar, T, T - \tau = t) \leq 5 e^{-\alpha_2 t},
\end{align}
for some constant $\alpha_2 > 0$. Moreover, conditioned on $T-\tau = t$ and $E_{ij}^2$, we have:
\begin{align*}
\Gapd(T) = x_i(T)-\frac{T}{N} \leq x_i(T)-x_j(T) = \left(x_i(T)-x_i(\tau)\right)-\left(x_j(T)-x_j(\tau)\right) \leq T-\tau,
\end{align*}
where the second equality uses the fact that $x_i(\tau) = x_j(\tau)$ under $E_{ij}^2$, and the inequality follows from the fact that, in the worst case, all balls go into $i$ between $\tau+1$ and $T$.

Putting these two facts together, we obtain: {
\begin{align*}
    \EE{\Gapd(T) \mid E_{ij}^2, \Tstar} \PP{E_{ij}^2 \mid \Tstar}
    &= \sum_{t = 1}^{T-\Tstar} \EE{\Gapd(T) | E_{ij}^2, \Tstar, T-\tau = t} \PP{E_{ij}^2 \mid \Tstar, T - \tau = t}\\
    &\leq \sum_{t = 1}^{T-\Tstar} 5t  e^{-\alpha_2 t}  \leq b_3,
\end{align*}
for some constant $b_3 > 0,$ where the last inequality follows from fact that $xe^{-ax}$ is integrable, for any $a > 0$.} By the tower rule, then, we have that $\mathbb{E}[\Gapd(T)\mid E_{ij}^2]\mathbb{P}(E_{ij}^2) \leq b_3$.

We now bound $\mathbb{E}[\Gapd(T)\mid E_{ij}^1] \mathbb{P}(E_{ij}^1)$. Fix the first flexing time $\Tstar$. Since at most \mbox{$T-\Tstar$} balls could have been added to any bin between $\Tstar+1$ and $T$, $\Gapd(T)\leq \Gapd(\Tstar)+T-\Tstar$. Moreover, by definition of $T^\star$, $\Gapd(T^\star-1) < \frac{\semicst(T-T^\star+1)q}{N}.$ Since \mbox{$\Gapd(T^\star) \leq \Gapd(T^\star-1)  + 1$}, we have the following bound on $\Gapd(T)$:
\begin{align}\label{eq:gap-t-bound}
\Gapd(T) < \frac{a_d(T-\Tstar+1)q}{N}+1+T-\Tstar = (1+T-\Tstar)\left(1+\frac{a_dq}{N}\right).
\end{align}

The following lemma allows us to partition our analysis into two cases, given $\Tstar$: (i) $T-\Tstar < t_0$, and (ii) $T-\Tstar \geq t_0$, for some constant $t_0 > 0$. We defer its proof to Appendix \ref{apx:t0-partition-lem}.

\begin{lemma}\label{lem:t0-partition-lem}
There exist constants $t_0, \alpha_0 > 0$ such that $\mathbb{P}(E_{ij}^1 \mid \Tstar) \leq 5e^{-\alpha_0(T-\Tstar)}$ for all $\Tstar \leq T-t_0$.
\end{lemma}

Fix the constant $t_0 > 0$ from \Cref{lem:t0-partition-lem}, and consider the following two cases.

\textbf{Case 1:} $T-\Tstar < t_0$. In this case, by \eqref{eq:gap-t-bound}, we have:
\begin{align*}
\mathbb{E}\left[\Gapd(T)\mid E_{ij}^1, \Tstar \right] \mathbb{P}\left(E_{ij}^1 \mid \Tstar \right) < (1+t_0)\left(1+\frac{a_dq}{N}\right).
\end{align*}

\textbf{Case 2:} $T-\Tstar \geq t_0$. In this case, putting together \eqref{eq:gap-t-bound} and \Cref{lem:t0-partition-lem}, we have:
\begin{align}\label{eq:to-plug-in}
\mathbb{E}[\Gapd(T) \mid E_{ij}^1, \Tstar] \mathbb{P}(E_{ij}^1 \mid \Tstar)
< (1+T-\Tstar)\left(1+\frac{a_dq}{N}\right)\cdot 5e^{-\alpha_0(T-\Tstar)}.
\end{align}

By the tower rule, we then have:
\begin{align*}
\mathbb{E}\left[\Gapd(T)\mid E_{ij}^1\right]\mathbb{P}(E_{ij}^1) &< t_0\cdot (1+t_0)\left(1+\frac{a_dq}{N}\right) + \sum_{t=t_0}^{T-1} (1+t)\left(1+\frac{a_dq}{N}\right)\cdot 5e^{-\alpha_0t}\leq b_4,
\end{align*}
for some $b_4 > 0$.

We conclude the proof by summing these terms over all $i,j \in [N]$, obtaining a final bound of $\mathbb{E}[\Gapd(T)] \leq N^2\left(b_3 + b_4\right) \in \mathcal{O}(1)$.
\end{proof}

\medskip

\subsubsection{Proof of \Cref{lem:eta-k}}\label{apx:eta-k}

\begin{proof}{Proof.}
{Note that, by construction, the semi-dynamic policy makes the same decisions as the no-flex policy up until $\Ta$. Therefore, 
$\Gapd(\Ta) = \Gapnf(\Ta) \geq \frac{a(T-\Ta)q}{N}.$
Thus, for $k \in \{1,\ldots,\lceil\sqrt{T}\rceil\}$:
\begin{align*}
    T-\Ta \geq (k-1) \sqrt{T} \implies \Gapnf(\Ta) \geq \frac{a(k-1)\sqrt{T} q}{N}.
\end{align*}
}

Hence, we have:
\begin{align}\label{eq:diff-to-gap}
\mathbb{P}\left(T-\Ta \geq (k-1) \sqrt{T}\right)
&\leq \mathbb{P}\left(\Gapnf(\Ta) \geq \frac{a(k-1) q}{N} \sqrt{T}\right)  
\leq \mathbb{P}\left(\Gapnf(\Ta) \geq \frac{a(k-1) q}{N} \sqrt{\Ta}\right),
\end{align}
where the second inequality follows from $\Ta \leq T$. Recall, $\Gapnf(\Ta) = \max_{i}\xnf_i(\Ta)-\Ta/N$. Applying a union bound to \eqref{eq:diff-to-gap}, we obtain:
\begin{align}\label{eq:ta-tail}
    \mathbb{P}\left(T-\Ta \geq (k-1) \sqrt{T}\right) \leq N \cdot \mathbb{P}\left(\xnf_1(\Ta) - \Ta/N \geq \frac{a(k-1) q}{N} \sqrt{\Ta}\right).
\end{align}
We define $X \sim \text{Ber}(1/N) - 1/N$ to be the random variable with $\EE{X} = 0,$ $\sigma^2 := \text{Var}(X) = \frac{1}{N}\parenthesis{1-\frac{1}{N}}$ and {$\rho := \EE{|X|^3} = \parenthesis{1 - \frac{1}{N}} \parenthesis{\frac{1}{N}}^3 + \frac{1}{N} \parenthesis{1 - \frac{1}{N}}^3 = \frac{1}{N} - \frac{3}{N^2} + \frac{4}{N^3} - \frac{2}{N^4}$.}

Applying the Berry{-}Esseen Theorem (Theorem 3.4.17 in \citet{durrett2019probability}) we have:
\begin{align}\label{eq:berry-esseen}
\mathbb{P}\left(\xnf_1(\Ta) - \Ta/N \geq \frac{a(k-1) q}{N} \sqrt{\Ta} \ \mid \ \Ta \right)&= \mathbb{P}\left(\frac{\xnf_1(\Ta) - \Ta/N}{\sqrt{\Ta\cdot\frac1N(1-\frac1N)}} \geq \frac{a(k-1) q}{\sqrt{N-1}} \ \mid \ \Ta \right) \notag \\
&\leq  1-\Phi\left(\frac{a(k-1) q}{\sqrt{N-1}}\right)+\frac{b}{\sqrt{\Ta}},
\end{align}
where $b = \frac{3 \rho}{\sigma^3} = \frac{3\sqrt{N} \cdot \parenthesis{1 - \frac{3}{N} + \frac{4}{N^2} - \frac{2}{N^3}}}{\parenthesis{1-\frac{1}{N}}^{1.5}}$. 

For ease of notation, we define $\eta_k = 1-\Phi\left(\frac{a(k-1) q}{\sqrt{N-1}}\right)$. {Noting that $\Ta \geq \Gapnf(\Ta) \geq \frac{a(T-\Ta)q}{N}$, we have that $\Ta \geq \frac{a q}{N+aq}T$. Plugging this back into~\eqref{eq:berry-esseen}, and defining $b_6 = b \sqrt{\frac{a q + N}{a q}}$, we have:
\begin{align*}
&\mathbb{P}\left(\xnf_1(\Ta) - \Ta/N \geq \frac{a(k-1) q}{N} \sqrt{\Ta} {\mid \Ta}\right) \leq \eta_k + \frac{b_6}{\sqrt{T}}.
\end{align*}
}
Using the law of total probability over $\Ta$ and applying this to \eqref{eq:ta-tail}, we obtain the result.
\end{proof}

\medskip

\subsubsection{Proof of \Cref{lem:t0-partition-lem}}\label{apx:t0-partition-lem}
\begin{proof}{Proof.}
We argue that $x_i^d(\Tstar)-x_j^d(\Tstar) \leq a_d(T-\Tstar)q+a$, for some constant $a > 0$. This will then allow us to apply \Cref{lem:e1ij_semi} to $E_{ij}^1$, with $t_1 = \Tstar, t_2 = T$, and $a_p = a_d \leq \frac{1}{5{N\choose 2}}$. This establishes that there exist constants  $\alpha_0 > 0$, $t_0 > 0$ such that: {
\begin{align}\label{eq:to-plug-in-2}
\PP{E_{ij}^1 \mid \Tstar} &=
\PP{E_{ij}^1 \mid \xsemi_{i}(T^\star) - \xsemi_{j}(T^\star) \leq \semicst(T-T^\star)q + \semicst q+N, \Tstar, T, T - \Tstar \geq t_0} \leq 5 e^{-\alpha_0(T-T^\star)}.
\end{align}}

To prove the linear bound on $x_i^d(\Tstar)-x_j^d(\Tstar)$, note that: 
{
\begin{align}
    \xsemi_{i}(T^\star) - \xsemi_{j}(T^\star) &\leq\max_{k} \xsemi_{k}(T^\star) - \min_{k} \xsemi_{k}(T^\star) \\
    &= N \cdot \max_{k} \xsemi_{k}(T^\star) - \parenthesis{(N-1) \cdot \max_{k} \xsemi_{k}(T^\star) + \min_{k} \xsemi_{k}(T^\star)} \notag\\
    &\leq N\cdot \parenthesis{\max_{k} \xsemi_{k}(T^\star) - T^\star/N} = N \Gapd(\Tstar),\label{eq:min_max}
\end{align}
where the second inequality comes from $(N-1) \cdot \max_{k} \xsemi_{k}(T^\star) + \min_{k} \xsemi_{k}(T^\star) \geq T^\star.$ }

Recall, $\Gapd(\Tstar) < \frac{a_d(T-\Tstar+1)q}{N}+1$. Plugging this into the above, we obtain:
\begin{align*}
x_i^d(\Tstar)-x_j^d(\Tstar) < N\left(\frac{a_d(T-\Tstar+1)q}{N}+1\right) = a_dq(T-\Tstar) + \left(a_dq+N\right),
\end{align*}
which completes the proof of the claim.
\end{proof}

%% file: opaque-appendix.tex
\medskip

\section{Application: Opaque Selling} \label{app:opaque}

\subsection{Analysis of the Salop Circle Model}\label{app:salop_analysis}

\subsubsection{Proof of \texorpdfstring{\cref{lem:purchase_probabilities}}{Lg}} \label{app:purchase_probabilities}

\begin{proof}{Proof.}
The platform's expected revenue in period $t$ is given by $\perperiodrev^\pi(t) = \left(\sum_{i=1}^{N}\hat{p}\cdot q_i^\pi(t)\right) + p_oq_o^\pi(t).$

Consider any period in which the opaque product is not offered. We abuse notation and let $\mathcal{R}(p)$ be the expected revenue obtained in this period, as a function of the price $p$. By definition, for any price ${p}$, $\mathcal{R}(p)$ is given by: 
\begin{align}\label{eq:rev-no-opaque}
\rev(p) = p\sum_{i=1}^{N}q_i^{\pi}(t) = p\cdot \mathbb{P}\left(\max_{i}\left\{\bar{v}-\gamma d(X,x_i)-p\right\} \geq 0\right) &= p\cdot \mathbb{P}\left(\min_id(X,x_i) \leq \frac{\bar{v}-p}{\gamma}\right) \notag \\
&= p\cdot 2N \cdot \mathbb{P}\left(X\leq \frac{\bar{v}-p}{\gamma}, X \in \left[0,\frac{1}{2N}\right]\right),
\end{align}
where the final equality follows from the symmetry of the Salop circle model (see \Cref{fig:salop}).

We partition our analysis into three cases.

\textbf{Case 1: $p \geq \bar{v}$.} By \eqref{eq:rev-no-opaque}, $\rev(p) = 0$. 

\smallskip

\textbf{Case 2: $p \leq \bar{v}-\tfrac{\gamma}{2N}$.} In this case, $\tfrac{\bar{v}-p}{\gamma} \geq \tfrac{1}{2N}$. Using this fact in \eqref{eq:rev-no-opaque}, we obtain $\rev(p) = p \cdot 2N$, whose maximum is then attained at $\bar{v}-\tfrac{\gamma}{2N}$.

\smallskip

\textbf{Case 3: $p \in \left(\bar{v}-\tfrac{\gamma}{2N},\bar{v}\right)$.} In this case, by \eqref{eq:rev-no-opaque}, $
\rev(p) = p\cdot 2N\cdot \frac{\bar{v}-p}{\gamma}.$
Differentiating the above and setting the derivative equal to 0, we obtain that the supremum is attained at $\max\left\{\tfrac{\bar{v}}{2},\bar{v}-\tfrac{\gamma}{2N}\right\} = \bar{v}-\tfrac{\gamma}{2N}$, since $\gamma \in [0, \overline{v} \cdot N]$ by assumption.

\smallskip

Combining the three cases, we obtain that $\hat{p} = \bar{v}-\tfrac{\gamma}{2N}$ is the revenue-maximizing price when the opaque product is not offered, with $\mathcal{R}(\hat{p}) = \bar{v}-\frac{\gamma}{2N}$. Under this price, a customer's utility for each product $i \in [N]$ is given by:
\begin{align}\label{eq:util-under-opt}
V_i - \hat{p}= \bar{v}-\gamma d(X,x_i)- \left(\bar{v}-\tfrac{\gamma}{2N}\right) = \gamma\cdot\left(\frac{1}{2N}-d(X,x_i)\right),
\end{align}
implying that each customer chooses product $i^* \in \arg\min_i d(X,x_i)$. Since $X$ is uniformly distributed over $[0,1]$, this implies that $q_i^\pi(t)  = \frac1N$ for all $i$ under this pricing scheme. 

\medskip 

Consider now any period in which the opaque product is sold, with $\hat{p} = \bar{v}-\tfrac{\gamma}{2N}$ and $p_o = \hat{p}-\delta$. \Cref{lem:vo_computation} derives a closed-form expression for $V_o$. We defer its proof to Appendix \ref{app:vo_computation}. 
\begin{proposition}\label{lem:vo_computation}
    $V_o = \overline{v} - \gamma/4$.
\end{proposition}

By \Cref{lem:vo_computation}, the customer's utility for the opaque product is:
\begin{align}\label{eq:util-under-opaque}
V_o - p_o = \overline{v} - \gamma/4 - (\hat{p} - \delta) = \frac{\gamma}{2N} - \frac\gamma4 +\delta,
\end{align}
where the final equality follows from the definition of $\hat{p} = \bar{v}-\tfrac{\gamma}{2N}$.

Since the utility obtained from the best non-opaque product is non-negative under $\hat{p}$ by construction, the customer purchases the opaque option if and only if:
\begin{align*}
V_o - p_o \geq \max_i \left\{V_i - \hat{p}\right\} = \frac{\gamma}{2N}-\gamma\cdot\min_{i}d(X,x_i),
\end{align*}
by \eqref{eq:util-under-opt}. Putting this together with \eqref{eq:util-under-opaque}, we have that the opaque option is chosen if and only if:
\begin{align*}
-\frac\gamma4 + \delta \geq -\gamma\cdot\min_i d(X,x_i) \iff \min_i d(X,x_i) \geq \frac14-\frac{\delta}{\gamma}.
\end{align*}

Therefore, by symmetry (see \Cref{fig:salop}), we obtain:
\begin{align}\label{eq:qo}
q_o^\pi(t) = 2N\cdot\mathbb{P}\left(X \geq \frac14-\frac\delta\gamma,\ X \in \left[0,\frac{1}{2N}\right]\right).
\end{align}

We again partition our analysis into three cases.

\smallskip

\textbf{Case 1: $\delta \geq \tfrac{\gamma}{4}$.} By \eqref{eq:qo}, $q_o^\pi(t) = 1$.

\smallskip 

\textbf{Case 2: $\delta \leq \gamma\left(\tfrac14-\tfrac{1}{2N}\right)$.} In this case, $\frac14-\frac\delta\gamma \geq \frac{1}{2N}$, therefore $q_o^\pi(t) = 0$.

\smallskip

\textbf{Case 3: $\delta \in \left(\gamma\left(\tfrac14-\tfrac{1}{2N}\right), \tfrac{\gamma}{4}\right)$.} In this case, $q_o^\pi(t) = 2N \cdot \left(\frac{1}{2N}-\left(\frac14-\frac\delta\gamma\right)\right) = 1-\frac{N}{2}+\frac{2N\delta}{\gamma}$.

In all three cases, $\sum_i q_i^\pi(t) = 1-q_o^\pi(t) \implies q_i^\pi(t) = \frac{1-q_o^\pi(t)}{N}$, again by symmetry. The expected revenue is thus $\hat{p} - \delta \cdot q^\pi_o(t) = \overline{v} - \gamma / (2N) - \delta \cdot q^\pi_o(t)$.
\end{proof}

\medskip

\subsubsection{Proof of \texorpdfstring{\cref{lem:vo_computation}}{Lg}} \label{app:vo_computation}

\begin{proof}{Proof.}
By definition, $V_o$ is given by:
\begin{align}\label{eq:vo}
V_o = \frac{\sum_{i=1}^{N}V_i}{N} = \bar{v}-\gamma\cdot\frac1N\sum_{i=1}^{N}d(X,x_i).
\end{align}

By symmetry, it is without loss of generality to assume $X = 0$. Fix $i \in [N]$. Since $x_i = i/N$, the mismatch $d(0,x_i)$ satisfies
\begin{align}\label{eq:d0}
d(0, x_i) \;=\; \min\Bigl(\bigl|0 - x_i\bigr|,\; 1 - \bigl|0 - x_i\bigr|\Bigr)
\;=\;
\min\Bigl(\frac{i}{N},\, 1 - \frac{i}{N}\Bigr).
\end{align}

Let $N = 2m$ for some $m \in \mathbb{Z}^+$. Then, \eqref{eq:d0} reduces to:
\begin{align*}
d(0, x_i) &\;=\;
\min\Bigl(\frac{i}{2m},\, \frac{2m-i}{2m}\Bigr)\\
\implies \sum_{i=0}^{2m-1} d(0,x_i) &= \sum_{i=0}^{m-1}\frac{i}{2m} + \sum_{i=m}^{2m-1}\frac{2m-i}{2m}\\
&=\frac{1}{2m}\left(\frac12m(m-1)+\frac12m(m+1)\right)\\
&= \frac{m}{2} \\
\implies V_o &= \bar{v}-\gamma/4,
\end{align*}
where the final implication follows from \eqref{eq:vo}. \end{proof}

\medskip

\subsubsection{Proof of \texorpdfstring{\cref{lem:long_run_revenue}}{Lg}} \label{app:long_run_revenue}

\begin{proof}{Proof.}
Denote by $\R^\pi_l$ the length of the $l$-th replenishment cycle and by $M^\pi_l$ the number of times the opaque product is purchased during the $l$-th replenishment cycle. Then, the long-run average revenue is given by:
$${Rev}^\pi = \lim_{L \to \infty} \frac{\sum_{l = 1}^L \parenthesis{\R^\pi_l \cdot \hat{p}  - M^\pi_l \cdot \delta}}{\sum_{l = 1}^L \R^\pi_l} = \hat{p} - \delta \cdot \frac{\lim_{L \to \infty} \frac{\sum_{l = 1}^L M^\pi_l}{L}}{\lim_{L \to \infty} \frac{\sum_{l = 1}^L \R^\pi_l}{L}} =\hat{p} - \delta \frac{\EE{M^\pi}}{\mathbb{E}[\R^\pi]},$$
where the second equality follows from the existence of finite limits and the last equality follows from the Renewal Reward Theorem~\citep{grimmett2020probability}.
\end{proof}

\medskip

\subsection{Analysis of the Benchmark Policies}\label{app:benchmark}

\subsubsection{Proof of \texorpdfstring{\cref{prop:nf_objective}}{Lg}} \label{app:nf_objective}

\begin{proof}{Proof.}
    Since the no-flex policy never offers the opaque product, we trivially have $\mathbb{E}[M^{nf}] = 0$. The bound on the expected length of the replenishment cycle is a direct corollary of Lemma 3 in \citet{elmachtoub2019value}. In particular, plugging in $\alpha = \frac{4}{3} \cdot \frac{\log^{(2)} N}{2 \log N}  < 1$ into Eq. (8) in \citet{elmachtoub2019value}, we have: 
    \begin{align}\label{eq:8elmachtoub}
    \mathbb{E}[\Rn] \leq NS - (NS - m_\alpha)(1-o(1)),
    \end{align}
    where {
    \begin{align*}
        m_\alpha =&\;NS - N \cdot \Bigg(\sqrt{2S\log(N)\parenthesis{1-\frac{1}{\alpha}\frac{\log^{(2)} N}{2 \log N}} - \parenthesis{\log(N) \parenthesis{1-\frac{1}{\alpha}\frac{\log^{(2)} N}{2 \log N}}}^2} \\
        &\qquad\qquad\qquad - \log(N)\parenthesis{1-\frac{1}{\alpha}\frac{\log^{(2)} N}{2 \log N}}\Bigg)\\
        =&\;NS - N \cdot \parenthesis{\sqrt{S \log(N)/2 - \parenthesis{\log(N)/4}^2} - \log(N)/4},
    \end{align*}
    where the second equality uses the definition of $\alpha.$} Plugging this into \eqref{eq:8elmachtoub}:
    \begin{align}\label{eq:tight_nf_bound}
\mathbb{E}[R^{nf}] \leq NS - N \cdot \parenthesis{\sqrt{S \log(N)/2 - \parenthesis{\log(N)/4}^2} - \log(N)/4}\left(1-o(1)\right).
    \end{align}
    We can thus conclude that $\mathbb{E}[\Rn] = NS - \thetan$, where $\thetan \in \Omega(\sqrt{S}).$
    This establishes \cref{prop:nf_objective} (i).

    \medskip

    For \cref{prop:nf_objective} (ii), note that \cref{lem:long_run_revenue} implies ${Rev}^{nf} = \hat{p} - \delta \cdot 0 = \hat{p}.$ 
    
    We now bound the long-run average inventory costs under the no-flex policy. We have:
    \begin{align*}
    \mathcal{K}^{nf} = \frac{K}{\mathbb{E}[R^{nf}]} =  \frac{K}{NS - \theta_{nf}} \geq \frac{K}{NS}. 
    \end{align*}
    Moreover, by Eq. (18) in \citet{elmachtoub2019value}, 
    \begin{align*}
    \mathcal{H}^{nf} \geq h\cdot \frac{(NS)^2-\frac{NS}{2} \thetan}{2(NS - \thetan)}=  \frac{h}{2} \frac{(NS)^2-NS\theta_{nf}+\frac{NS}{2} \thetan}{NS-\thetan}&= \frac{h}{2} \frac{(NS)\left(NS-\thetan\right) + \frac{NS}{2} \thetan}{NS-\thetan}\\
    &\geq \frac{h}{2}\left(NS+\frac12\theta_{nf}\right).
    \end{align*}
    Putting these two together, we obtain: 
    \begin{align}\label{eq:nf_tight}
        {Inv}^{nf} &\geq  \frac{K}{NS} + \frac{h}{2} \cdot NS + \frac{h}{2} \cdot \frac{\thetan}{2} = \frac{K}{NS} + \frac{h}{2} \cdot NS + \xi_{nf},
    \end{align}
     {where } $\xi_{nf} \in \Omega\parenthesis{\frac{1}{\sqrt{S}}}$, since $h \in \Theta\parenthesis{\frac{1}{S}}$ and $\thetan \in \Omega(\sqrt{S})$.
\end{proof}

\medskip

\subsubsection{Proof of \texorpdfstring{\cref{prop:af_objective}}{Lg}} \label{app:af_objective}

\begin{proof}{Proof.}
    By Lemma 4 in \citet{elmachtoub2019value}, $\mathbb{E}[\Ra] = NS - \thetaa$, where $\thetaa \in \mathcal{O}\left(1\right)$. Since the always-flex policy offers the opaque product in every period, we trivially have $\mathbb{E}[M^{a}] = q_o \cdot \mathbb{E}[\R^{a}] = \Theta(S)$. This establishes \cref{prop:af_objective} (i).

    For \cref{prop:af_objective} (ii), note that \cref{lem:long_run_revenue} implies ${Rev}^{a} = \hat{p} - \delta \cdot q_o.$ 

    We now bound the long-run average inventory costs under the always-flex policy. We have:
    \[\mathcal{K}^a = \frac{K}{\mathbb{E}[R^a]} = \frac{K}{NS - \thetaa},\] again using the fact that $\mathbb{E}[R^a] = NS-\thetaa$. To bound $\mathcal{H}^a$, note that for any policy $\pi$:
    \begin{align}
        \mathcal{H}^\pi &= \frac{(2NS + 1)\mathbb{E}[R^\pi] - \mathbb{E}[(R^\pi)^2]}{2 \mathbb{E}[R^\pi]}h
        \leq \frac{(2NS+1)\mathbb{E}[R^\pi] - \mathbb{E}[R^\pi]^2}{2 \mathbb{E}[R^\pi]}h
        = \frac{2NS +1 - \mathbb{E}[R^\pi]}{2}h, \label{eq:inv_derivation}
    \end{align}
    where the second inequality follows from Jensen's inequality. Thus, for the always-flex policy, we have:
    \[\mathcal{H}^a \leq \frac{2NS +1 - \mathbb{E}[\Ra]}{2}h = \frac{h}{2}(NS+1+\theta_{a}).\] Thus,
    \begin{align}\label{eq:af_in_detail}
        {Inv}^{a} =\frac{K}{NS - \theta_{a}} + \frac{h}{2}(NS+1+\theta_{a})
        &= \frac{K}{NS} \cdot \frac{1}{1-\frac{\theta_{a}}{NS}} + \frac{h}{2}(NS+1+\theta_{a}) \notag \\
         &\leq \frac{K}{NS} \cdot \left(1+\theta_a'\right) + \frac{h}{2}(NS+1+\theta_{a}),
    \end{align}
    where $\theta_a' \in \mathcal{O}\parenthesis{\frac1S}$, since $\frac{1}{1-\frac{\theta_{a}}{NS}} \leq 1+2\frac{\theta_{a}}{NS}$, for large enough $S$, and $\theta_a \in \mathcal{O}(1)$. Using the fact that $K\in\Theta(S)$ and $h\in\Theta\parenthesis{\frac1S}$, we obtain:
    \begin{align}
        {Inv}^{a} \leq &\;\frac{K}{NS} + \frac{h}{2} \cdot NS +\xi_{a}, \text{ where }\xi_{a} \in \mathcal{O}\parenthesis{\frac{1}{S}}.
    \end{align}
\end{proof}

\subsection{Analysis of the Semi-Dynamic Policy}\label{apx:opaque-semi} 

\subsubsection{Formal description}\label{apx:alg-semi}

We present the semi-dynamic policy for the opaque selling model in \Cref{alg:semi} below.
\begin{algorithm}[H]
\caption{Semi-Dynamic Policy $\pid$}\label{alg:semi}
\begin{algorithmic}[1]
\Require $t = 1, \zsemi_i(0)=S\;\forall \ i \in [N]$, $\flexaction^d(t) = 0\  \forall \ t \geq 0$, {constant $\cstsemi > 0$} 
\While{$ \min_i \zsemi_i(t-1) >0$}
  \If{$\flexaction^d(t) = 1$ and the $t$-th customer purchases the opaque product}
  \State Sample $\mathcal{F}(t)$.
  \State Allocate product $j = \argmax_{i \in \mathcal{F}(t)} \zsemi_i(t-1)$.
  \Else 
  \State Fulfill customer's request for preferred product $j$.
  \EndIf
    \For{$i \in [N]$}
    \State Update inventory levels:
    \[\zsemi_i(t)= \begin{cases}
    \zsemi_i(t-1)-1 &\quad \text{if } i = j \\
    \zsemi_i(t-1) &\quad \text{otherwise}.
    \end{cases}\]
    \EndFor
  \If{$\Gap_I^{d}(t) \geq \frac{\cstsemi(T-t)q_o}{N}$}
        \State Offer the opaque product from $t+1$ onwards, i.e., set $\flexaction^d(t') = 1\;\forall \ t' > t$
  \EndIf
  \State $t = t+1$
\EndWhile
\end{algorithmic}
\end{algorithm}

\subsubsection{Proof of \texorpdfstring{\cref{thm:dynamic_objective}}{Lg}}

For ease of notation, we define $q := q_o$. Note that $\delta \in \Theta(1)$ and $\delta > \parenthesis{\frac{1}{4} - \frac{1}{2N}} \gamma$ imply that $q > 0$ is a constant independent of $S$. 

\begin{proof}{Proof.}
The following fact, derived by \citet{elmachtoub2019value}, will allow us to (i) reduce the opaque selling model to the balls-into-bins model, and (ii) show that ensuring a long expected replenishment cycle requires a bound on the tail of the gap of the system, for all $t \geq S$. Specifically, \Cref{prop:relation} below uses the analogy of a customer being allocated product $i$ as a ball landing into bin $i$.

\begin{lemma}[Lemma 1 in \citet{elmachtoub2019value}] \label{prop:relation}
Consider any policy $\pi$ designed for the opaque selling problem. Let $x^\pi_i(t)$ be the load of the coupled balls-into-bins model under $\pi$, which couples each allocation decision of customer to product, to an allocation decision of ball to bin. Then, the distributions of the replenishment cycle length in the opaque selling model and the maximum load across bins in the balls-into-bins model are related as follows: 
$$\mathbb{P}(\R^\pi \leq t) = \mathbb{P}(\max_i x^\pi_i(t) \geq \Sinv) \quad \forall \, t \in \mathbb{N}^+.$$
\end{lemma}

Noting that, under this coupling, $\min_{i\in[N]}z_i^\pi(t) = S-\max_{i\in[N]}x_i^\pi(t)$, we have, for all $t \in \mathbb{N}^+$:
{
\begin{equation}\label{eq:gap_equivalent}
    \Gap_I^{\pi}(t) = S - \frac{t}{N} - \min_{i \in [N]} z_i^{\pi}(t) = \max_{i \in [\Nballs]} x_i^\pi(t) - \frac{t}{\Nballs} = \Gap^{\pi}(t).
\end{equation}
Therefore, without loss of generality in the remainder of the proof we focus on the coupled balls-into-bins process. In particular, since $\Gapp(t) := \max_j \xp_j(t) - \frac{t}{N}$, \cref{prop:relation} implies
}
\begin{equation}\label{eq:relation}
    \mathbb{P}(R^\pi \leq t) = \mathbb{P}\left(\Gapp(t) \geq S - \frac{t}{N}\right) \quad \forall \ t \in \mathbb{N}^+.
\end{equation}

\medskip

Having established the reduction to the balls-into-bins model, we proceed with the proof of parts (i) and (ii) of the theorem. The main driver of the bounds on $Rev^d$ and $Inv^d$ is the bound on the expected replenishment cycle length induced by the semi-dynamic policy. We state it as a lemma below, deferring its proof to Appendix \ref{app:opaque-semi}. 

\begin{lemma}\label{thm:opaque-semi}
$\mathbb{E}[\Rd] = NS - \thetad$, where \mbox{$\thetad \in \mathcal{O}(1)$, for $\cstsemi \leq \frac{1}{10 \binom{N}{2}}$ and $\delta \in \Theta(1)$.}
\end{lemma}

With this lemma in hand, we prove \Cref{thm:dynamic_objective}.

    We first bound $\mathbb{E}[M^d]$. We have:
    \begin{align}\label{eq:md-bound}
        \EE{M^d} =\EE{\sum_{t = 1}^{\Rd-\Tstar} \mathbbm{1}\{{V_o - p_o \geq \max_i \bracket{V_i - \hat{p}}}\}} &\leq\EE{\sum_{t = 1}^{T-\Tstar} \mathbbm{1}\{V_o - p_o \geq \max_i\bracket{V_i - \hat{p}}\}} \notag \\&= q_o \EE{T - \Tstar},
    \end{align}
    where the inequality uses the fact that $R^d \leq T$, and the last equality is an application of Wald's identity. Since the semi-dynamic policy mimics the no-flex policy until {$\Tstar$}, we have: $$\Tstar = \inf\left\{t:\Gapnf(t) \geq \frac{\cstsemi(T-t)q}{N}\right\}.$$ 
    Then, by \cref{thm:ball_semi2}, $\mathbb{E}[T-T^\star] \in \mathcal{O}\left(\sqrt{T}\right) = \mathcal{O}\left(\sqrt{S}\right)$. Thus, $\EE{M^d} \in \mathcal{O}\left(\sqrt{S}\right)$.

    We now bound the long-run average revenue of the semi-dynamic policy. By \cref{lem:long_run_revenue}:
    \begin{align}\label{eq:rev_tight}
        {Rev}^{d} &=  \hat{p}-\delta \cdot \frac{\mathbb{E}[M^d]}{\mathbb{E}[{R}^d]}=  \hat{p} - \delta \cdot \frac{q_o \cdot \mathbb{E}[T - \Tstar]}{NS - \thetad} = \hat{p} - \delta \cdot q_o \cdot \xi^1_{d},
    \end{align}
    where $\xi^1_{d} \in \mathcal{O}\parenthesis{\frac{1}{\sqrt{S}}}$. Here, the second equality applies \eqref{eq:md-bound} and \Cref{thm:opaque-semi} to $\mathbb{E}[M^d]$ and $\mathbb{E}[R^d]$, respectively. The third equality uses the just-derived facts that $\mathbb{E}[T-\Tstar] \in \mathcal{O}(\sqrt{S})$ and $\theta_d \in \mathcal{O}(1)$, by \Cref{thm:opaque-semi}.

    We now bound the inventory costs incurred by the semi-dynamic policy. By \Cref{thm:opaque-semi}, \[\mathcal{K}^d = \frac{K}{\mathbb{E}[\Rd]} = \frac{K}{NS - \theta_{d}}.\] Applying the bound from \eqref{eq:inv_derivation} on the long-run average holding cost, we have: 
    \begin{equation}\label{eq:inv_tight}
       \mathcal{H}^d \leq \frac{2NS +1 - \mathbb{E}[\Rd]}{2}h = \frac{h}{2}(NS+{1}+\theta_{d}),
    \end{equation}
    where the equality again follows from \Cref{thm:opaque-semi}.
    Putting these two together, we obtain: 
    \begin{align}\label{eq:final-inv-for-cor}
    {Inv}^{d} =\frac{K}{NS - \theta_{d}} + \frac{h}{2}(NS+1+\theta_{d}) = \frac{K}{NS} + \frac{h}{2} \cdot NS + \xi^2_{d}, \text{ where }\xi^2_{d} \in \mathcal{O}\parenthesis{\frac{1}{S}},
    \end{align}
    where the final equality uses the same bounding arguments as those used in the proof of \Cref{prop:af_objective}.
\end{proof}

\medskip

\subsubsection{Proof of \texorpdfstring{\cref{thm:opaque-semi}}{Lg}}\label{app:opaque-semi}

The proof of \Cref{thm:opaque-semi} relies on the following useful facts. We defer their proofs to Appendix \ref{apx:fic_allocation}, \ref{apx:inventory_bound_lem}, and \ref{apx:af_gap}, respectively.

Suppose the gap of the system is bounded by some $a > 0$ at the time the semi-dynamic policy begins flexing. Consider moreover a second system initialized at $t_1$ under a perfectly balanced state, but penalized by $a$ additional balls in each bin initially. \Cref{cl:fic_allocation} establishes the intuitive fact that the likelihood that at least $S$ balls have landed in a bin before time $t \geq t_1$ in the first system is no more than the likelihood that there are least $S$ balls in any bin by time $t$ (counting the initial $a$ balls) in this second system.
\begin{lemma}\label{cl:fic_allocation}
{Fix $t_1 > 0$ (where $t_1$ is potentially random), and consider any policy $\pi$ such that (i) $\flexaction^\pi(t) = 0$ for all $t \leq t_1$ and (ii) $\flexaction^\pi(t) = 1$ for all $t > t_1$. Let $F'$ be any subset of the history before $t_1$ such that $$F' \subseteq F_a = \{\history{t_1} \in \historyset{t_1}\mid \Gapp(t_1) \leq a\}$$ for some $a > 0.$} Then, \[\mathbb{P}(\Gapp(t) \geq S- t/N\mid {t_1}, F') \leq \mathbb{P}(\Gapp(t) \geq S- t/N - a\mid {t_1}, \Gapp(t_1) = 0)\  \forall \ t \geq t_1.\]
\end{lemma}

Fix any time $t_2$ after the start of the flexing horizon. \Cref{lem:inventory_bound} below reduces the task of bounding the likelihood that at least $S$ balls have been placed in a single bin by some time $t \in \{t_2+1,\ldots,N(S-1)+1\}$ to that of bounding the likelihood that the loads of the maximally and minimally loaded bins at $t_2$ never intersected between $t_1$ and $t_2$.
\begin{lemma}\label{lem:inventory_bound}
{Fix $t_1 > 0$ (where $t_1$ is potentially random), and consider any policy $\pi$ such that (i) $\flexaction^\pi(t) = 0$ for all $t \leq t_1$ and (ii) $\flexaction^\pi(t) = 1$ for all $t > t_1$. Define, for random variable $t_2$ such that \mbox{$t_1 < t_2 \leq N(S-1)+1$},} and for $i, j \in [N]$:
\begin{align*}
F_{ij}^1 &:= \left\{\history{t_2} \in \historyset{t_2}\mid i \in \arg\max_{k\in[N]} \xp_k(t_2), j \in \arg\min_{k\in[N]} \xp_k(t_2), \xp_i(t) \neq \xp_j(t), \forall \ t \in \{t_1,...,t_2\}\right\}.
\end{align*}
Then, there exists a constant $a_3 > 0$ such that: $$\sum_{t = t_2+1}^{N(S-1)+1} \mathbb{P}(\Gapp(t) \geq S- t/N {\mid t_1, t_2}) \leq a_3 + \left(N(S-1)+1-t_2\right)\sum_{i,j}\mathbb{P}(F_{ij}^1 {\mid t_1, t_2}).$$
\end{lemma}

\medskip

Finally, throughout the proof we abuse notation and use $\Gapa(t)$ to denote the gap of a system initialized at 0 and for which the always-flex policy as been run for $t$ periods. The following lemma leverages results from \citet{elmachtoub2019value} to establish a bound on the tail of $\Gapa(t)$.
\begin{lemma} \label{prop:af_gap}
There exist constants $\cstaf,\beta > 0$ such that, for any $\eta \geq 0$,  \[\mathbb{P}(\Gapa(t) \geq \eta) \leq \beta e^{-\cstaf \eta} \quad \forall \ t \geq 0.\]
\end{lemma}

\medskip

With these lemmas in hand, we prove \Cref{thm:opaque-semi}.

\begin{proof}{Proof of \Cref{thm:opaque-semi}.}
We first express the length of the expected replenishment cycle as a function of the tail of $\Gapp(t)$, for any policy $\pi$. We defer the proof of \Cref{prop:final_relation} below to the end of the section.
\begin{proposition} \label{prop:final_relation}
For any policy $\pi$, $\mathbb{E}[R^\pi] = N(S-1)+2 - \sum_{t = S}^{N(S-1)+1} \mathbb{P}(\Gapp(t) \geq S- t/N).$
\end{proposition}
Thus, it suffices to show the following:
\[\sum_{t = S}^{N(S-1)+1} \mathbb{P}(\Gapd(t) \geq S- t/N) \in \mathcal{O}(1).\]

We first argue that $\Gapd(t) < S-t/N$ for all $t < \Tstar$. To see this, recall that $c_d q \in (0,1]$ and \mbox{$T = N(S-1)+1$}. As as result, $\frac{\cstsemi(T-t)q}{N} < S - t/N$ for all $t \leq T$. Thus:
\begin{align*}
    \Gapd(t) < \frac{\cstsemi(T-t)q}{N} \quad \forall \, t < \Tstar \implies \Gapd(t) < S - t/N \quad \forall \, t < \Tstar,
\end{align*}
where the inequality on the left-hand side of the implication follows from the fact that $\Tstar$ is the first time the threshold condition is triggered, and the inequality on the right-hand side of the implication uses the upper bound on $\frac{\cstsemi(T-t)q}{N}$ derived above.

Fix a realization of $\Tstar$, and let $\Ttilde := \left\lceil\frac{\Tstar+T}{2}\right\rceil$ be the midpoint between $\Tstar$ and $T$. We then have: 
\begin{align}
       \sum_{t=S}^{N(S-1)+1} \PP{\Gapd(t) \geq S-t/N \mid \, T^\star} &= \sum_{t=T^\star}^{N(S-1)+1} \PP{\Gapd(t) \geq S-t/N \, \mid \, T^\star} \label{eq:tstar_guarantee}\\
       &= \underbrace{\sum_{t=T^\star}^{\Ttilde} \PP{\Gapd(t) \geq S-t/N\, \mid \, T^\star}}_{(I)} \notag \\
       &\qquad + \underbrace{\sum_{t=\Ttilde + 1}^{N(S-1)+1} \PP{\Gapd(t) \geq S-t/N\, \mid \, T^\star}}_{(II)} \notag.
\end{align}

We show that for any value of $\Tstar$ each of these two terms is upper bounded by a constant; applying the law of total probability over $\Tstar$, the lemma is then shown.

\medskip

\noindent\textbf{Bounding $(I)$.} We reduce the task of bounding the likelihood that at least $S$ balls have been placed in a single bin by time $t \in \{\Tstar,\ldots,\Ttilde\}$ to the analysis of an always-flex policy run for $t-\Tstar$ periods on an initially empty system. 

To do so, we first upper bound $\Gapd(\Tstar)$. Noting that (i) {$\Gapd(T^\star-1) < \frac{\cstsemi (T-(\Tstar-1))q}{N}$ by definition of $\Tstar$, and (ii) $\Gapd(T^\star) \leq \Gapd(T^\star-1) + 1$, we have: \begin{align}\label{eq:gap_semi_ub}
    \Gapd(\Tstar) < \frac{\cstsemi q(T-\Tstar)}{N} + \frac{\cstsemi q}{N}+1.
\end{align}}

As a result, for all $t \in \{\Tstar,\ldots,\Ttilde\}$,
\begin{align}
   \mathbb{P}\left(\Gapd(t) \geq S-t/N \mid \Tstar\right) &=  \mathbb{P}\left(\Gapd(t) \geq S- t/N \, \mid \, \Gapd(\Tstar) < \frac{\cstsemi q(T-\Tstar)}{N} + \frac{\cstsemi q}{N}+1, T^\star\right) \label{eq:opaque-gap-ub-prob}
\end{align}

Let $E_{\Tstar}$ be the event that $\Gapd(\Tstar) < \frac{\cstsemi q(T-\Tstar)}{N} + \frac{\cstsemi q}{N}+1$. {Applying \Cref{cl:fic_allocation} to \eqref{eq:opaque-gap-ub-prob}, with \mbox{$F' = E_{\Tstar}, t_2 = \Tstar$} and $a =\frac{\cstsemi q(T-\Tstar)}{N} + \frac{\cstsemi q}{N}+1$, we have, for all $t \geq \Tstar$:}
\begin{align*}
    &\mathbb{P}\left(\Gapd(t) \geq S- t/N \, \mid \, \Gapd(\Tstar) < \frac{\cstsemi q(T-\Tstar)}{N} + \frac{\cstsemi q}{N}+1, T^\star\right)\\
    &\leq \mathbb{P}\left(\Gapd(t) \geq S- t/N-\left(\frac{\cstsemi q(T-\Tstar)}{N} + \frac{\cstsemi q}{N}+1\right)\, \mid \,\Gapd(\Tstar) = 0, T^\star\right) \\
    &=\mathbb{P}\left(\Gapa(t-\Tstar) \geq S- t/N-\left(\frac{\cstsemi q(T-\Tstar)}{N} + \frac{\cstsemi q}{N}+1\right) \mid T^\star\right),
\end{align*}
where the final equality follows from the observation that, conditional on $\Gapd(\Tstar) = 0$, the semi-dynamic policy takes the same action as the always-flex policy initialized on an empty system at $\Tstar$.
Note that
\begin{align}\label{eq:nonneg}
S- t/N-\left(\frac{\cstsemi q(T-\Tstar)}{N} + \frac{\cstsemi q}{N}+1\right) > 0 \quad \forall \ t \leq \Ttilde-2.
\end{align}
Therefore, by \Cref{prop:af_gap}, for all $t \in \{\Tstar,\ldots,\Ttilde-2\}$,
\begin{align*}
&\mathbb{P}\left(\Gapa(t-\Tstar) \geq S- t/N-\left(\frac{\cstsemi q(T-\Tstar)}{N} + \frac{\cstsemi q}{N}+1\right) \mid T^\star\right)\\ &\leq \beta\exp\left[-c_a \left(S- t/N-\left(\frac{\cstsemi q(T-\Tstar)}{N} + \frac{\cstsemi q}{N}+1\right)\right)\right].
\end{align*}

Summing over all $t \in \{\Tstar,\ldots,\Ttilde\}$, we have:
{
\begin{align}
   (I) 
    &\leq 2+ \sum_{t = \Tstar}^{\Ttilde-2} \beta \exp\left({-\cstaf\left[S- t/N-\left(\frac{\cstsemi q(T-\Tstar)}{N} + \frac{\cstsemi q}{N}+1\right)\right]}\right) \notag\\
    &\leq 2+\left(\int_{t = \Tstar - 1}^{\Ttilde-2} \beta \exp\left({-\cstaf\left[S- t/N-\left(\frac{\cstsemi q(T-\Tstar)}{N} + \frac{\cstsemi q}{N}+1\right)\right]}\right) dt\right)  \label{eq:exponential_small}\\
    &=2+\beta e^{-\cstaf\parenthesis{S - \left(\frac{\cstsemi q(T-\Tstar)}{N} + \frac{\cstsemi q}{N}+1\right)}} \cdot N/\cstaf \cdot \parenthesis{e^{\frac{\cstaf(\Ttilde-2)}{N}} - e^{\frac{\cstaf(\Tstar-1)}{N}}}  \notag\\
    &=2+\beta \cdot N/\cstaf \cdot \parenthesis{e^{-\cstaf\parenthesis{S - \frac{\Ttilde-2}{N} - \left(\frac{\cstsemi q(T-\Tstar)}{N} + \frac{\cstsemi q}{N}+1\right)}} - e^{-\cstaf\parenthesis{S - \frac{\Tstar-1}{N} - \left(\frac{\cstsemi q(T-\Tstar)}{N} + \frac{\cstsemi q}{N}+1\right)}}} .\notag
\end{align}}

By \eqref{eq:nonneg}, we obtain the existence of a constant $a_1 > 0$ such that $(I) \leq a_1$.

\medskip

\noindent\textbf{Bounding $(II)$.} Fix a constant $t_0 > 0$ determined in \Cref{lem:e1ij_semi}. Consider first the event that $\Ttilde-\Tstar < t_0$. In this case, we trivially have:
\begin{align*}
(II) \leq (T-\Ttilde) \leq t_0+1, 
\end{align*}
a constant by assumption.

Consider now the event that $\Ttilde-\Tstar \geq t_0$. For $i,j\in [N]$, let $E_{ij}^1$ denote the event such that (i) $i$ is a bin with the highest load in period $\Ttilde$ and $j$ is a bin with the smallest load in period $\Ttilde$, and (ii) the loads of bins $i$ and $j$ were never the same between $\Tstar$ and $\Ttilde$. Formally:
\begin{align*}
E_{ij}^1 & :=\left\{\history{\Ttilde} \in \historyset{\Ttilde}\mid i \in \arg\max_{k\in[N]} \xsemi_k(\Ttilde), j \in \arg\min_{k\in[N]} \xsemi_k(\Ttilde), \xsemi_i(t) \neq \xsemi_j(t), \forall \ t \in \{\Tstar,...,\Ttilde\}\right\}.
\end{align*}

{Applying \cref{lem:inventory_bound} to $(II)$, with $t_1 = \Tstar, t_2 = \Ttilde$ and $F_{ij}^1 = E_{ij}^1 \; \forall \ i,j,$ we obtain:} {
\begin{align}\label{eq:inventory_main2}
    (II) = \sum_{t=\Ttilde + 1}^{N(S-1)+1} \PP{\Gapd(t) \geq S-t/N\, \mid \, T^\star, \Ttilde} \leq a_3 + (T-\Ttilde)\sum_{i,j}\mathbb{P}(E_{ij}^1\mid \Tstar, \Ttilde)
\end{align}
for some constant $a_3 > 0$, where we also used the fact that $T=N(S-1)+1$ by definition. The equality above comes from the fact that conditioned on $T^\star$, $\Ttilde$ is deterministic and provides no additional information on the gap of the system.}
Therefore, it suffices to bound the probability that the maximum and minimum loads never intersect between $\Tstar$ and $\Ttilde$. 

In order to do so, we will show that $x_i^d(\Tstar)-x_j^d(\Tstar)$ is upper bounded by a linear function of $\Ttilde-\Tstar$ and apply \Cref{lem:e1ij_semi} to $E_{ij}^1$. \Cref{lem:opaque-diff-ub} leverages the threshold condition to obtain an explicit bound. We defer its proof to the end of the section.
\begin{lemma}\label{lem:opaque-diff-ub}
$x_i^d(\Tstar)-x_j^d(\Tstar) \leq \frac{1}{5{N\choose 2}}q(\Ttilde-\Tstar)+c_dq+N$.
\end{lemma}

Then, applying \Cref{lem:e1ij_semi} to $E_{ij}^1$, with $t_1 = \Tstar$, $t_2 = \Ttilde$, $\policycst = \frac{1}{5 \binom{N}{2}}$ and $a = \cstsemi q+N$, there exists a  constant $\alpha_0 > 0$ such that: {
\begin{align}\label{eq:bound_t0}
    \mathbb{P}\parenthesis{E_{ij}^1 \mid \Ttilde, T^\star, \Ttilde-T^\star \geq t_0} &= \mathbb{P}\parenthesis{E_{ij}^1 \mid x_i^d(\Tstar)-x_j^d(\Tstar) \leq \frac{1}{5 \binom{N}{2}}q(\Ttilde-T^\star) + \cstsemi q+N, \Ttilde, T^\star, \Ttilde-T^\star \geq t_0} \notag \\
    &\leq 5 e^{-\alpha_0(\Ttilde-T^\star)}.
\end{align}
}

Since $\Ttilde-T^\star \geq t_0$, we plug \eqref{eq:bound_t0} back into \eqref{eq:inventory_main2} to obtain:
\begin{align*}
(II) \leq a_3 + (T-\Ttilde)\sum_{i,j}\mathbb{P}(E_{ij}^1\mid \Tstar, \Ttilde, \Ttilde-T^\star \geq t_0)&\leq a_3 + N^2 (T-\Ttilde)\cdot 5 e^{-\alpha_0(\Ttilde-T^\star)}\\
&\leq a_3 + N^2\left(\Ttilde-\Tstar+1\right)\cdot 5 e^{-\alpha_0(\Ttilde-T^\star)}\\
&\leq a_2,
\end{align*}
for some constant $a_2 > 0$, thereby completing the proof.
\end{proof}

\medskip

\begin{proof}{Proof of \Cref{prop:final_relation}.}
Recall, the maximum length of a replenishment cycle is $N(S-1)+1$. Therefore,
\begin{align*}
    \mathbb{E}[R^\pi] = \sum_{t = 0}^{N(S-1)+1} \mathbb{P}(R^\pi \geq t)
    &= N(S-1)+2-\sum_{t = S}^{N(S-1)+1}\mathbb{P}(R^\pi < t)\\
    & \geq  N(S-1)+2 - \sum_{t = S}^{N(S-1)+1} \mathbb{P}(\Gapp(t) \geq S- t/N),
\end{align*}
where the last inequality follows from \eqref{eq:relation}. Notice also that the second equality uses the fact that \mbox{$\mathbb{P}(R^\pi < t) = 0\ \forall \ t < S,$} since the platform cannot run out of inventory of any type of product before $S$ products are sold.
\end{proof}

\medskip

\begin{proof}{Proof of \Cref{lem:opaque-diff-ub}.}
Note that {
\begin{align*}
    x_i^d(\Tstar)-x_j^d(\Tstar) \leq \max_{k} \xsemi_{k}(T^\star) - \min_{k} \xsemi_{k}(T^\star) =&\;N \cdot \max_{k} \xsemi_{k}(T^\star) - \parenthesis{(N-1) \cdot \max_{k} \xsemi_{k}(T^\star) + \min_{k} \xsemi_{k}(T^\star)} \\
    \leq&\;N\cdot \parenthesis{\max_{k} \xsemi_{k}(T^\star) - T^\star/N} = N \Gapd(\Tstar),
\end{align*}where the inequality uses the fact that $(N-1) \cdot \max_{k} \xsemi_{k}(T^\star) + \min_{k} \xsemi_{k}(T^\star) \geq T^\star.$ } Applying the upper bound on $\Gapd(\Tstar)$ derived in \eqref{eq:gap_semi_ub}, we obtain:
\begin{align}
    x_i^d(\Tstar)-x_j^d(\Tstar)
    &\leq N \left(\frac{\cstsemi q(T-\Tstar)}{N} + \frac{\cstsemi q}{N}+1\right) \notag \\
    &= \cstsemi q(T-T^\star) + \cstsemi q+N \notag \\
    &\leq2\cstsemi q(\Ttilde-T^\star) + \cstsemi q+N  \notag \\
    &= \frac{1}{5 \binom{N}{2}}q(\Ttilde-T^\star) + \cstsemi q+N \label{eq:min_max_ub},
\end{align}
where the second equality uses the definition of $\Ttilde$ as the midpoint between $\Tstar$ and $T$, and the final equality plugs in the fact that $\cstsemi \leq \frac{1}{10{N\choose 2}}$.
\end{proof}

\medskip

\subsubsection{Proof of \Cref{cl:fic_allocation}}\label{apx:fic_allocation}

\begin{proof}{Proof.}
Consider any sample path $\sigma'(t_1) \in F'$. Consider moreover a ``fictitious'' set of bins in which the fictitious load of bin $k$ at $t_1$ is denoted by $\zfic_k(t_1)$, with $\zfic_k(t_1) = t_1/N+a$. Intuitively, the fictitious bins represent a balanced system that is initially penalized by having $a$ additional balls in each bin. 
We assume that both systems --- also referred to as bin configurations --- see the same realizations of randomness (i.e., the same sequence of arrivals and allocations of balls to bins) for all $t \geq t_1$.

We will argue that:
\begin{align}\label{eq:fic_bin}
    \zreal_k(t) \leq \zfic_k(t) \quad \forall \ t \geq t_1, k \in [N],
\end{align}
which will then imply that $\max_k \zreal_k(t) - t/N \leq \max_k \zfic_k(t)- t/N$ for all $t \geq t_1.$ 
We use this to bound the tail of $\Gapp(t)$, for all $t \geq t_1$:
\begin{align*}
    &\mathbb{P}(\Gapp(t) \geq S- t/N \mid t_1, F')\\ &= \sum_{\historypath{t_1} \in F'}\mathbb{P}(\max_k \zreal_k(t) - t/N \geq S- t/N, \sigma'(t_1)\mid t_1, F')\\
    &= \sum_{\historypath{t_1} \in F'}\mathbb{P}(\max_k \zreal_k(t) - t/N \geq S- t/N\mid t_1, F', \sigma'(t_1))\mathbb{P}(\sigma'(t_1)\mid t_1, F')\\
    &= \sum_{\historypath{t_1} \in F'}\mathbb{P}(\max_k \zreal_k(t) - t/N \geq S- t/N\mid t_1, \sigma'(t_1))\mathbb{P}(\sigma'(t_1)\mid t_1, F')\\
    &\leq \sum_{\historypath{t_1} \in F'} \mathbb{P}(\max_k \zfic_k(t)- t/N \geq S- t/N\mid t_1, \historypath{t_1})\mathbb{P}(\historypath{t_1} \mid t_1, F')\\
    &= \mathbb{P}(\max_k \zfic_k(t)- t/N \geq S- t/N \mid t_1, F')\\
    &= \mathbb{P}(\max_k \zfic_k(t)- t/N \geq S- t/N \mid t_1),
\end{align*}
where the final equality uses the fact that $F'$, the history of the ``real" system, provides no additional information on loads of the fictitious system. 

Notice that the fictitious system had ``overloaded'' the system with $a$ additional balls. In order to reduce back to the original system, we consider an ``unloaded" system in which the load of bin $k$ at $t_1$ is denoted by $\zrestart_k(t_1)$, with $\zrestart_k(t_1) = t_1/N$ and $\Gaprestart(t) := \max_k \zrestart_k(t) - t/N$. Note that, by construction, $\zfic_k(t) = \zrestart_k(t) + a$ for all $k \in [N]$, $t \geq t_1$. Using this fact above, we obtain:
\begin{align*}
  \mathbb{P}(\Gapp(t) \geq S- t/N \mid t_1, F')  &\leq \mathbb{P}(\max_k \zrestart_k(t) + a - t/N \geq S- t/N \mid t_1)\\
    &= \mathbb{P}(\Gaprestart(t)+a \geq S- t/N\mid t_1)\\
    &= \mathbb{P}(\Gaprestart(t) \geq S- t/N-a\mid t_1) \\
    &= \mathbb{P}(\Gapp(t) \geq S- t/N-a\mid t_1, \Gapp(t_1) = 0),
\end{align*}
where the final equality uses the fact that, conditional on $\Gapp(t_1) = 0$, the same decisions are made in the original and unloaded systems.

What is left to conclude the proof is \eqref{eq:fic_bin}, which we now show by induction. 

\noindent\textbf{Base case: $t = t_1$.}
Since $\sigma'(t_1)$ is such that $\Gapp(t) \leq a$, we have:
 $$\zreal_k(t_1) \leq \max_{k'} \zreal_{k'}(t_1) \leq t_1/N + a = \zfic_k(t_1), \forall \ k \in [N].$$

\noindent\textbf{Inductive step.} Fix $t \in \{t_1+1,\ldots,T\}$, and suppose $\zreal_k(t) \leq \zfic_k(t), \forall \ k \in [N].$ We argue that \mbox{$\zreal_k(t+1) \leq \zfic_k(t+1), \forall \ k  \in [N],$} by discussing the following cases. Recall, $f(t)$ is used to denote whether or not a ball is flexible, $\mathcal{F}(t)$ denotes the flex set of two randomly chosen bins, and $P(t)$ denotes the non-flexible preferred bin.
\begin{enumerate}[label=\arabic{enumi}.]
    \item Suppose $\flextype{t} = 0$ and $\preferred{t}= k_1$, for some $k_1 \in [N]$. In this case, $$\zreal_{k_1}(t+1) = \zreal_{k_1}(t)+1, \text{ and } \zfic_{k_1}(t+1) = \zfic_{k_1}(t)+1,$$ while the loads of other bins do not change. Thus, $\zreal_k(t+1) \leq \zfic_k(t+1), \forall \ k \in [N].$
    \item Suppose $\flextype{t} = 1$, with $\flexset{t} = \{k_1, k_2\}$, for some $k_1, k_2 \in [N]$.
    \begin{enumerate}[\arabic{enumi}.\arabic{enumii}.]
          \item If $\zfic_{k_1}(t)<\zfic_{k_2}(t)$ then the flex ball is allocated to bin $k_1$ in the fictional configuration. If the flex ball also goes into bin $k_1$ in the real configuration, then the load of bin $k_1$ increases by 1 in both bin configurations. The loads of all other bins do not change, and so the induction holds. If the flex ball goes into bin $k_2$ in the real configuration, then it must be the case that $\zreal_{k_2}(t) \leq \zreal_{k_1}(t)$, which implies:
          \[
          \zreal_{k_2}(t) \leq \zreal_{k_1}(t)\leq \zfic_{k_1}(t)<\zfic_{k_2}(t),\]
          where the second inequality holds by the inductive hypothesis, and the third inequality holds by assumption, in this case. Since the flex ball goes into bin $k_2$ in the real configuration, we then have:
         \[\zreal_{k_2}(t+1) = \zreal_{k_2}(t)+1 \leq \zfic_{k_2}(t) = \zfic_{k_2}(t+1),\]
         since the flex ball was allocated to bin $k_1$ in the fictitious configuration.
         Since the loads of all other binds remain unchanged, the claim holds for all $k \in [N]$. 
          \item If $\zfic_{k_1}(t)>\zfic_{k_2}(t)$, a symmetric argument as the one used for Case 2.1 can be used, and the claim holds in this case as well.
          \item If $\zfic_{k_1}(t)=\zfic_{k_2}(t)$ and $k_1 < k_2$, then a flex ball goes into bin $k_1$ in the fictional configuration, by the lexicographic allocation rule. In this case, if $\zreal_{k_1}(t) \leq \zreal_{k_2}(t)$, the flex ball is also allocated to bin $k_1$ in the real allocation, and the induction still holds. If $\zreal_{k_1}(t) > \zreal_{k_2}(t)$, on the other hand, then the flex ball goes into bin $k_2$. In that scenario, we must have: $$\zreal_{k_2}(t) \leq \zreal_{k_1}(t)-1\leq \zfic_{k_1}(t)-1=\zfic_{k_2}(t)-1,$$
          where the second inequality follows from the inductive hypothesis, and the final equality is  by assumption. This leads to: $$\zreal_{k_2}(t+1) = \zreal_{k_2}(t)+1 \leq \zfic_{k_2}(t) = \zfic_{k_2}(t+1),$$
          since $k_2$ was not chosen in the fictitious configuration. Since all other bin loads remain unchanged, the claim holds. 
          We omit the proof for the case where $k_1 > k_2$, as it is symmetric.
    \end{enumerate}
\end{enumerate}
\end{proof}

\subsubsection{Proof of \Cref{lem:inventory_bound}}\label{apx:inventory_bound_lem}

\begin{proof}{Proof.}
For $i,j \in [N]$, let $F_{ij}^2$ be the event that $i$ and $j$ are respectively the (strictly) maximally and minimally loaded bins at $t_2$, and that their loads intersected for some \mbox{$\tau := \max\left\{t \in \{t_1,...,t_2-1\} \mid \xp_i(t) = \xp_j(t)\right\}$}. Formally:
\begin{align*}
    F_{ij}^2 &:= \Big\{\history{t_2} \in \historyset{t_2}\mid i \in \arg\max_{k\in[N]} \xp_k(t_2), j \in \arg\min_{k\in[N]} \xp_k(t_2),\\
&\hspace{3.5cm}\xp_i(t) = \xp_j(t) \text{ for some } t \in \{t_1,...,t_2-1\},\xp_i(t_2) \neq \xp_j(t_2)\Big\}.
\end{align*}
With slight abuse of notation, we denote by $F_{ij}^2 \cap \tau$ the event that $F_{ij}^2$ occurs and $\tau$ is the last time \mbox{$\xp_i(t) = \xp_j(t)$} before $t_2$. For $t > t_2$, we have:
\begin{equation}\label{eq:static_decomp}
\begin{split}
     &\quad\;\mathbb{P}(\Gapp(t) \geq S- t/N \mid t_1, t_2)\\
    &= \mathbb{P}(\Gapp(t) \geq S- t/N\mid t_1, t_2, \Gapp(t_2) = 0) \cdot \mathbb{P}(\Gapp(t_2) = 0 \mid t_1, t_2) \\
    &\quad + \mathbb{P}(\Gapp(t) \geq S- t/N\mid t_1, t_2, \Gapp(t_2) > 0) \cdot \mathbb{P}(\Gapp(t_2) > 0 \mid t_1, t_2)\\
    &\leq \mathbb{P}(\Gapp(t) \geq S- t/N\mid t_1, t_2, \Gapp(t_2) = 0) \\
    &\quad + \sum_{i,j} \left(\mathbb{P}(\Gapp(t) \geq S- t/N|t_1, t_2, F_{ij}^1) \mathbb{P}(F_{ij}^1 \mid t_1, t_2)+\mathbb{P}(\Gapp(t) \geq S- t/N|t_1, t_2, F_{ij}^2) \mathbb{P}(F_{ij}^2 \mid t_1, t_2) \right) \\
    &\leq \underbrace{\mathbb{P}(\Gapp(t) \geq S- t/N\mid t_1, t_2, \Gapp(t_2) = 0)}_{(I)} \\
    &\quad + \sum_{i,j} \left(\mathbb{P}(F_{ij}^1 \mid t_1, t_2) + \sum_{\tau = t_1}^{t_2-1}\underbrace{ \mathbb{P}(\Gapp(t) \geq S- t/N|t_1, t_2, F_{ij}^2 \cap \tau) \mathbb{P}(F_{ij}^2 \cap \tau \mid t_1, t_2)}_{(II)} \right).
\end{split}
\end{equation}
{Here, the first inequality above comes from bounding $\mathbb{P}(\Gapp(t_2) = 0 \mid t_1, t_2) \leq 1$ and applies a union bound over $F_{ij}^1, F_{ij}^2$, for $i, j \in [N]$. The second inequality similarly bounds \mbox{$\mathbb{P}(\Gapp(t) \geq S- t/N \mid t_1, t_2, F_{ij}^1) \leq 1$}, and applies the law of total probability on the value of $\tau$.} Hence, it remains to bound $(I)$ and $(II)$.

{Consider first $(I)$, which represents the probability that at least $S$ balls have been placed in a single bin by time $t$, given that the system was perfectly balanced at $t_2$. Conditional on $\Gapp(t_2) = 0$, $\pi$ takes the same actions as the always-flex policy would if it was initialized with all-empty bins at time $t_2$.} {
Thus, for all $t > t_2$: 
\begin{align}
    (I) = \mathbb{P}(\Gapp(t) \geq S- t/N\mid t_1, t_2, \Gapp(t_2) = 0) &=\mathbb{P}(\Gapa(t - t_2) \geq S - t/N \mid t_2) \notag \\
    &\leq \beta e^{-c_a\left(S-t/N\right)}, \label{eq:af_relation}
\end{align}
for some constants $c_a, \beta > 0$, where the final inequality follows from \Cref{prop:af_gap}.}
Summing the above over all $t > t_2$, we have:
\begin{align}
\sum_{t = t_2+1}^{N(S-1)+1} \mathbb{P}\left(\Gapp(t) \geq S-t/N \mid t_1, t_2, \Gapp(t_2) = 0\right) &\leq \sum_{t =t_2+1}^{N(S-1)+1}\beta e^{-c_a\left(S-t/N\right)} \notag \\ &\leq \int_{t_2}^{N(S-1)+1}\beta e^{-\cstaf (S-t/N)} dt \notag \\
    &= \beta e^{-c_a S} \cdot \frac{N}{c_a}\left(e^{c_a(S-1+1/N)}-e^{c_a t_2/N}\right) \notag \\ 
    &= \frac{\beta N}{c_a}\left(e^{-c_a(1-1/N)}-e^{-c_a\left(S-t_2/N\right)}\right) \notag \\
    &\leq a_4, \label{eq:i-for-lemma-12}
\end{align}
for some constant $a_4 > 0$, since $S-t_2/N > 0$ for all $t_2 < N(S-1)+1$. 

\smallskip 

We now bound $(II)$, for all $i, j \in [N]$, $t > t_2$. Intuitively, this represents the likelihood that at least $S$ balls are placed into a single bin by time $t > t_2$, if the maximally and minimally loaded bins intersected for some $\tau \in \{t_1,\ldots,t_2-1\}$. By \cref{lem:e2ij}, there exists $\alpha_2 > 0$ such that {
\begin{align}\label{eq:lem12-fij2-1}
\mathbb{P}\left(F_{ij}^2\cap \tau \mid t_1, t_2 \right)\leq \mathbb{P}(F_{ij}^2 \mid t_1, t_2, \tau) \leq  5 e^{-\alpha_2 (t_2 - \tau)}.
\end{align}
}
Moreover, conditional on $F_{ij}^2\cap\tau$, $\Gapp(t_2)\leq t_2-\tau$, since the gap can increase by at most one in each period.
Therefore:
\begin{align*}
\mathbb{P}(\Gapp(t) \geq S- t/N\mid t_1, t_2, F_{ij}^2 \cap \tau) &= \mathbb{P}(\Gapp(t) \geq S- t/N\mid t_1, t_2, F_{ij}^2 \cap \tau, \Gapp(t_2) \leq t_2 - \tau) \\
&\leq \mathbb{P}(\Gapp(t) \geq S- t/N - (t_2-\tau)\mid t_2, \Gapp(t_2) = 0),
\end{align*}
by \Cref{cl:fic_allocation}. Again, conditioned on $\Gapp(t_2) = 0$, $\pi$ makes the same decisions as the always-flex policy initialized at $t_2$ with all-empty bins. Therefore,
\begin{align}\label{eq:lem12-fij2-2}
\mathbb{P}(\Gapp(t) \geq S- t/N\mid t_1, t_2, F_{ij}^2 \cap \tau) &\leq \mathbb{P}(\Gap^a(t-t_2) \geq S- t/N - (t_2-\tau) \mid t_2) \notag \\
&\leq \beta e^{-c_a\left(S-t/N-(t_2-\tau)\right)},
\end{align}
for all $t < N(S-(t_2-\tau))$ (thereby ensuring $S-t/N-(t_2-\tau) > 0$), by \Cref{prop:af_gap}. Putting together \eqref{eq:lem12-fij2-1} and \eqref{eq:lem12-fij2-2}, completes the bound of $(II)$.

Summing $(II)$ over all \mbox{$\tau \in \{t_1,\ldots,t_2-1\}$}, $t > t_2$ and $i,j \in [N]$, we obtain:
\begin{align}
&\sum_{t = t_2+1}^{N(S-1)+1}\sum_{i,j}\sum_{\tau = t_1}^{t_2-1}\mathbb{P}(\Gapp(t) \geq S- t/N \mid t_1, t_2, F_{ij}^2 \cap \tau)\mathbb{P}\left(F_{ij}^2\cap \tau \mid t_1, t_2\right) \notag \\&\leq N^2 \sum_{\tau=t_1}^{t_2-1}5e^{-\alpha_2 (t_2 - \tau)}\left[\left(\sum_{t=t_2+1}^{N(S-(t_2-\tau))-1}\beta e^{-c_a\left(S-t/N-(t_2-\tau)\right)}\right) + \left(N(t_2-\tau-1)+2\right)\right] 
, \label{eq:af_again}
\end{align}
where the final inequality loosely upper bounds $\mathbb{P}\left(\Gapp(t) \geq S-t/N \mid t_1, t_2, F_{ij}^2 \cap \tau \right) \leq 1$ for all \mbox{$t \geq N(S-(t_2-\tau))$}. Using the same bounding arguments as those used in \eqref{eq:exponential_small}, there exist constants $a_5, a_6 > 0$ such that:
\begin{align}\label{eq:af_again2}
\eqref{eq:af_again} \leq N^2 \sum_{\tau=t_1}^{t_2-1}5 e^{-\alpha_2 (t_2 - \tau)}\left[a_5 + \left(N(t_2-\tau-1)+2\right)\right] \leq a_6.
\end{align}

Summing \eqref{eq:i-for-lemma-12} and \eqref{eq:af_again2}, letting $a_3:= a_4 + a_6$ and plugging these bounds back into \eqref{eq:static_decomp}, we obtain the statement of the lemma.
\end{proof}

\subsubsection{Proof of \Cref{prop:af_gap}}\label{apx:af_gap}

\begin{proof}{Proof.}
For $t \in \mathbb{N}$, let $y_i(t)$ denote the difference between the load of the $i$-th most loaded bin and the average load in period $t$ under the always-flex policy, with $y(t) := (y_i(t), i \in [N])$. Moreover, consider the potential function $$\Gamma(t) = \sum_{i = 1}^N \exp(c_1 \epsilon y_i(t)) + \sum_{i = 1}^N \exp(-c_1 \epsilon y_i(t)),$$ where $c_1 > 0$, $\epsilon > 0$ are constants that depend on $N$ but not on $S$. By Lemma 9 in \citet{elmachtoub2019value}, there exists a constant $c_2 > 0$ such that:
\begin{align}\label{eq:gamma-bound}
\mathbb{E}[\Gamma(t)] \leq \frac{c_2}{\epsilon^7}N\ \forall \ t \geq 0.
\end{align}

Since $\Gapa(t) = y_1(t)$ by definition, we can loosely lower bound $\Gamma(t)$ as \mbox{$\Gamma(t) \geq e^{c_1 \epsilon \Gapa(t)}$}. Then, for any $\eta > 0$,
\begin{align*}
    \mathbb{P}(\Gapa(t) \geq \eta) = \mathbb{P}(e^{c_1 \epsilon \Gapa(t)} \geq e^{c_1 \epsilon \eta})\leq \mathbb{P}(\Gamma(t) \geq e^{c_1 \epsilon \eta}) &= \mathbb{P}\left(\Gamma(t) \geq \mathbb{E}[\Gamma(t)]\cdot \frac{e^{c_1 \epsilon \eta}}{\mathbb{E}[\Gamma(t)]}\right)\\
    &\leq \mathbb{P}\left(\Gamma(t) \geq \mathbb{E}[\Gamma(t)] \cdot \frac{e^{c_1 \epsilon \eta}}{\frac{c_2}{\epsilon^7}N}\right)\\
    &\leq \frac{c_2}{\epsilon^7} N e^{-c_1 \epsilon \eta},
\end{align*}
where the second inequality follows from \eqref{eq:gamma-bound}, and the last step follows from Markov's inequality. Taking $\cstaf = c_1 \epsilon$ and $\beta = \frac{c_2}{\epsilon^7} N$ completes the proof.
\end{proof}

\medskip

\subsubsection{Proof of \texorpdfstring{\cref{cor:comparison}}{Lg}}\label{app:comparison} 
\begin{proof}{Proof.}
For \cref{cor:comparison} (i), we have:
\begin{align*}
{Pr}^{d} - {Pr}^{a} &= (Rev^d-Inv^d)-(Rev^a-Inv^a)
\\
&=\left((Rev^*-\delta q_o \xi_d^1)-(Inv^*+\xi_d^2)\right)-\left((Rev^*-\delta q_o)-(Inv^*+\xi_a)\right)
\\
&= \delta \cdot q_o \parenthesis{1-\xi_d^1} + \parenthesis{\xi_a - \xi_d^2}\geq  \delta \cdot q_o \parenthesis{1-\xi_d^1} -\xi_d^2,
\end{align*}
where the first equality follows from \Cref{thm:dynamic_objective} and \Cref{prop:af_objective}.
Since $\xi_d^1, \xi_d^2 \in o(1)$ and $\delta q_o \in \Theta(1)$, we have that $\delta \cdot q_o \parenthesis{1-\xi_d^1} - \xi_d^2 \in \Omega(1).$ Thus, ${Pr}^{d} - {Pr}^{a} \in \Omega(1).$

For \cref{cor:comparison} (ii), we first compare ${Inv}^{d}$ and ${Inv}^{nf}$. Respectively applying \eqref{eq:nf_tight} and \eqref{eq:inv_tight} to each of these terms, we have:
\begin{align}
{Inv}^{nf} - {Inv}^{d}\;&\geq \frac{K}{NS - \thetan}+h\frac{(NS)^2-\frac{NS}{2} \thetan}{2(NS - \thetan)} - \parenthesis{\frac{K}{NS-\thetad} + \frac{NS +1 +\thetad}{2}h}\notag\\
&= K \cdot \frac{\thetan-\thetad}{(NS-\thetan)(NS-\thetad)} + \frac{h}{2} \frac{NS\parenthesis{\frac{\thetan}{2}-\thetad-1}+\thetan+\thetan \thetad}{(NS-\thetan)}\notag\\
&\geq  K \cdot \frac{\thetan-\thetad}{(NS)^2} + \frac{h}{2} \frac{NS\parenthesis{\frac{\thetan}{2}-\thetad-1}}{NS}\notag \\
&= K \cdot \frac{\thetan-\thetad}{(NS)^2} + \frac{h}{2} \cdot {\parenthesis{\frac{\thetan}{2}-\thetad-1}}.\label{eq:inv_comparison}
\end{align}
The first equality above comes from rearranging the $K$ and $h$ terms; the second inequality comes from upper bounding the $NS-\thetan$ and $NS-\thetad$ terms in the denominators by $NS$, while also lower bounding $NS\parenthesis{\frac{\thetan}{2}-\thetad-1}+\thetan+\thetan \thetad$ by $NS\parenthesis{\frac{\thetan}{2}-\thetad-1}$ in the numerator.

We now bound the difference in the revenues achieved by each policy. Recall, $Rev^{nf} = \hat{p}$. Applying \eqref{eq:rev_tight} from the proof of \Cref{thm:dynamic_objective} to $Rev^d$, we have: 
\begin{equation}\label{eq:rev_temp}
    {Rev}^{nf} - {Rev}^{d} = \delta \cdot q_o \frac{\mathbb{E}[T - \Tstar]}{NS - \thetad}.
\end{equation}
We now bound $\mathbb{E}[T-\Tstar]$. {Recall, in order to prove \Cref{thm:ball_semi2}, we provided an explicit bound on the bound the time that any semi-dynamic policy that starts flexing, given by $T_a = \inf\left\{\Gap^{nf}(t) \geq \frac{a(T-t)q}{N}\right\}$, where $a > 0$ is a constant. Applying this bound (see \eqref{eq:tight_analysis}) to our semi-dynamic policy here, with $a = \cstsemi$ and $T_a = \Tstar$, we have:
}
\begin{align*}
    \mathbb{E}[T-T^\star] &\leq\parenthesis{N \frac{3\sqrt{N} \cdot \parenthesis{1 - \frac{3}{N} + \frac{4}{N^2} - \frac{2}{N^3}}}{\parenthesis{1-\frac{1}{N}}^{1.5}} \sqrt{\frac{\cstsemi q_o + N}{\cstsemi q_o}} + N \frac{\sqrt{N-1}}{\cstsemi q_o} \parenthesis{1+ \sqrt{\frac{1}{2 \pi}}}} \cdot \sqrt{T}\\
    &\leq \frac{N\sqrt{T}}{q_o} \cdot \parenthesis{\frac{3\sqrt{N} \cdot \parenthesis{1 - \frac{3}{N} + \frac{4}{N^2} - \frac{2}{N^3}}}{\parenthesis{1-\frac{1}{N}}^{1.5}} \sqrt{1+5N^3} + 7N^{2.5}},
\end{align*}
where the second inequality uses the fact that $c_d = \frac{1}{10{N\choose 2}}$ and \mbox{$q_o \leq 1$} when factoring out $1/q_o$. Noting that, for all $N \geq 2$, $1-\frac3N+\frac4N^2-\frac2N^3\leq 2$, $\left(1-\frac1N\right)^{1.5} \geq 1/4$, and \mbox{$\sqrt{1+5N^3} \leq \sqrt{6N^3}$}, we obtain:
\begin{align*}
\mathbb{E}[T-\Tstar]&\leq \frac{N\sqrt{T}}{q_o} \left(24\sqrt{6}N^2 + 7N^{2.5}\right) \leq \frac{49N^{3.5}\sqrt{T}}{q_o},
\end{align*}
where the final inequality follows from the fact that $24\sqrt{6}N^2 \leq 42N^{2.5}$ for all $N \geq 2$. Plugging the upper bound on $\mathbb{E}[T-T^\star]$ into \eqref{eq:rev_temp}, and noting that $T\leq NS$, we obtain:
\begin{equation}\label{eq:rev_comparison}
    {Rev}^{nf} - {Rev}^{d} \leq \delta \cdot \frac{49 N^{3.5} \sqrt{NS}}{NS - \thetad}.
\end{equation}

Using the lower bound on the difference in inventory costs derived in \eqref{eq:inv_comparison}, we obtain the following lower bound on the profit difference:
\begin{align*}
    {Pr}^{d} - {Pr}^{nf} &\geq K \cdot \underbrace{\frac{\thetan-\thetad}{(NS)^2}}_{(I)} + \frac{h}{2} \underbrace{{\parenthesis{\frac{\thetan}{2}-\thetad-1}}}_{(II)} - \delta \cdot \frac{49 N^{3.5} \sqrt{NS}}{NS - \thetad}\\
    &=\thetan\cdot \frac{5\psi}{4NS} - \thetad\cdot \frac{3\psi}{2NS} - \delta \cdot \frac{49 N^{3.5} \sqrt{NS}}{NS - \thetad}-\frac{\psi}{2NS},
\end{align*}
where the equality uses the assumption that $K = \psi\cdot NS$ and $h = \psi/(NS)$. Using the fact that $\thetad \in \mathcal{O}(1)$ by \Cref{thm:dynamic_objective}, there exists $\theta_1 \in \mathcal{O}\left(\frac1S\right)$ such that:
\begin{align*}
    {Pr}^{d} - {Pr}^{nf}
    &=\thetan\cdot \frac{5\psi}{4NS} - \delta \cdot \frac{49 N^{3.5} \sqrt{NS}}{NS - \thetad}-\theta_1\\
    &= \frac{1}{NS}\left(\frac{5\psi\thetan}{4}-\delta\cdot\frac{49N^{3.5}\sqrt{NS}}{1-\thetad/(NS)}\right)-\theta_1\\
    &= \frac{1}{NS}\cdot \frac{5\psi\thetan(1-\thetad/NS)-196\delta N^4\sqrt{S}}{4(1-\thetad/NS)}-\theta_1.
\end{align*}
Therefore, it suffices to show the existence of a constant $\psi > 0$ such that \[{5\psi\thetan(1-\thetad/NS)-196\delta N^4\sqrt{S}} \in \Omega(\sqrt{S}),\]
since $0 <4(1-\thetad/NS) \leq 4$ for large enough $S$, in the denominator. Equivalently, it suffices to show the existence of a constant $\psi$ such that:
\begin{align}\label{eq:lb-thetan}
5\psi\frac{\thetan}{\sqrt{S}}(1-\thetad/NS)-196\delta N^4 \geq \theta_2,
\end{align}
for some $\theta_2 \in \Omega(1)$. Recall, we derived tight bounds on $\mathbb{E}[R^{nf}]$ in the proof of \Cref{prop:nf_objective}. In particular, by \eqref{eq:tight_nf_bound}, we had that, for large enough $S$: 
\begin{align*}
&N \cdot \parenthesis{\sqrt{S \log(N)/2 - \parenthesis{\log(N)/4}^2} - \log(N)/4}\left(1-o(1)\right)\\& \geq \frac12\cdot N \parenthesis{\sqrt{S \log(N)/2} -\frac{\parenthesis{\log(N)/4}^2}{2\sqrt{S \log(N)/2 - \parenthesis{\log(N)/4}^2}} - \log(N)/4} \\
&\geq \frac12 N\sqrt{S\log(N)/2}-\theta_3,
\end{align*}
for some $\theta_3 \in \Theta(1)$.
Plugging this lower bound into the left-hand side of \eqref{eq:lb-thetan}, and again using the fact that $\thetad\in\mathcal{O}(1)$, we obtain:
\begin{align*}
5\psi\frac{\thetan}{\sqrt{S}}(1-\thetad/NS)-196\delta N^4 \geq 5\psi\cdot \left(\frac{N\sqrt{\log(N)/2}}{2}-\frac{\theta_3}{\sqrt{S}}-\theta_4\right)-196\delta N^4,
\end{align*}
for $\theta_4 \in \mathcal{O}\left(\frac1S\right)$. For large enough $S$, this term is strictly positive for all $\psi$ satisfying:
\begin{align*}
\psi &> \frac{196\delta N^4}{(5/4)N\sqrt{\log(N)/2}}.
\end{align*}
Noting that the right-hand side is upper bounded by $267N^3\delta$, and letting $C = 267N^3$, we obtain the result.
\end{proof}

%% file: balls-into-bins-numerics.tex
\section{Additional Opaque Selling Experiments}\label{apx:numerical-results-balls-into-bins}

{Throughout these experiments, we let $\mu = 0.1, \delta = 0.05$, $\underline{v} = 0.6$, and randomize over $K, h$ and $\alpha$ as in \cref{sec:opaque_extensions}. All results are averaged over $100$ replications.}

{
\paragraph{Sensitivity analysis with respect to $\underline{c}$.} We now investigate the robustness of our insights to the marginal cost $\underline{c}$. We present our results in \Cref{tab:c}, where we report the average marginal production cost per period for the three policies of interest in Columns 2-4, and the relative profit improvement of the semi-dynamic policy over $\pi^{nf}, \pi^a, \max\{\pi^{nf},\pi^a\}$ in Columns 5-7. 

Production costs dampen the benefits of increased sales, since additional units sold now incur additional cost. We observe that the always-flex policy incurs the highest marginal cost per period due to the greater sales volume induced by the opaque option, {followed by the semi-dynamic and the no-flex policies. Moreover, the absolute difference in production costs between the semi-dynamic and always-flex policy increases with $\underline{c}$.} {As a result, as $\underline{c}$ increases, the semi-dynamic policy yields a substantial improvement over the always-flex policy (varying from 8.4\% up to 26.1\%), while maintaining robust performance relative to the no-flex policy (always exceeding 5\%). Across all tested values of $\underline{c}$, the semi-dynamic policy consistently achieves a relative profit improvement of more than 2.4\% over the better of the two benchmarks.} Overall, these results illustrate that late-stage opaque selling is most valuable in settings where the inventory cost savings outweigh the production cost incurred by the lift in sales volume.}


\begin{table}
\centering
\begin{tabular}{c|ccc|ccc}
\toprule
& \multicolumn{3}{c|}{Marginal Cost per Period} & \multicolumn{3}{c}{Relative Profit Improvement (\%)} \\
\midrule
$\underline{c}$ & $\pi^{nf}$ & $\pi^{a}$ & $\pi^{d}$ & $\pi^{nf}$ & $\pi^{a}$ & $\max\{\pi^{nf},\pi^{a}\}$ \\
\midrule
$0$ & 0 & 0 & 0 & 5.9 & 8.4 & 2.4 \\
$0.05$ & 0.043 & 0.046 & 0.044 & 6.1 & 9.8 & 2.6 \\
$0.10$ & 0.084 & 0.092 & 0.088 & 6.9 & 13.5 & 3.9 \\
$0.15$ & 0.125 & 0.138 & 0.131 & 6.8 & 18.5 & 4.1 \\
$0.20$ & 0.164 & 0.184 & 0.173 & 5.0 & 26.1 & 3.5 \\
\bottomrule
\end{tabular}
\caption{Impact of $\underline{c}$ on semi-dynamic policy performance.}\label{tab:c}
\end{table}

{
\paragraph{Dependence on $K$ and $h$.} Finally, we examine how the relative performance of the semi-dynamic policy varies with inventory cost parameters $K$ and $h$. Consistent with \Cref{cor:comparison}, \Cref{fig:pattern-a} shows that the semi-dynamic policy underperforms the no-flex policy only when $h$ and $K$ are very small, and its performance gains become significant as $h$ and $K$ increase. In contrast, the semi-dynamic policy maintains a consistent advantage over the always-flex policy across all tested regions of $h$ and $K$. As shown in \Cref{fig:pattern-b}, the relative improvement is larger when $h$ and $K$ are large, given the fact that the always-flex policy holds higher inventory on average (as discussed earlier), thereby incurring higher inventory costs overall.}

\begin{figure}
    \centering
    \subfloat[\centering Relative profit improvement (\%) over $\pi^{nf}$]{{\includegraphics[width=0.45\textwidth]{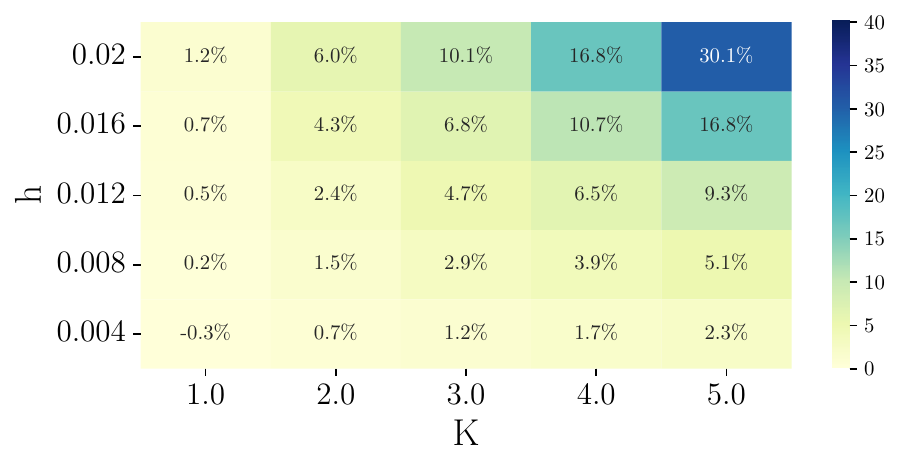}}\label{fig:pattern-a}}%
    \quad 
    \subfloat[\centering Relative profit improvement (\%) over $\pi^{a}$]{{\includegraphics[width=0.45\textwidth]{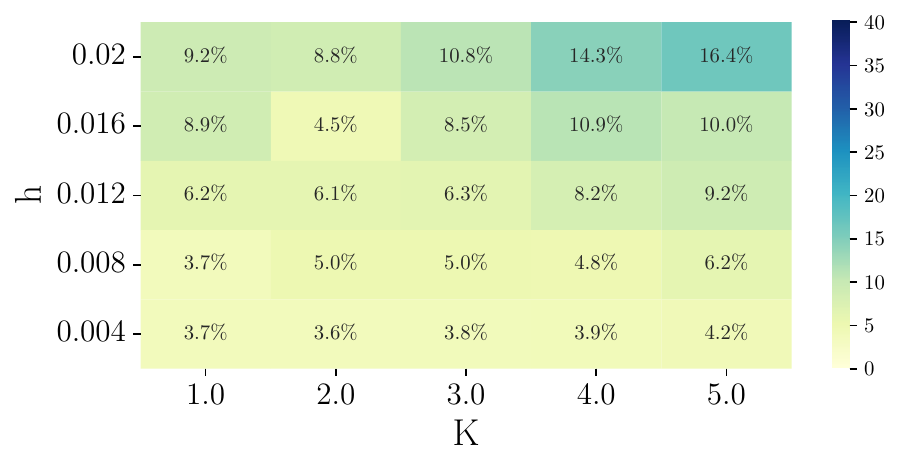}}\label{fig:pattern-b}}%
    \caption{\centering Heatmaps for the average relative profit improvement of $\pid$ over $\pi^{nf}$ and $\pia$.}
    \label{fig:pattern}
\end{figure}

%% file: main.bbl
\begin{thebibliography}{47}
\providecommand{\natexlab}[1]{#1}
\providecommand{\url}[1]{\texttt{#1}}
\expandafter\ifx\csname urlstyle\endcsname\relax
  \providecommand{\doi}[1]{doi: #1}\else
  \providecommand{\doi}{doi: \begingroup \urlstyle{rm}\Url}\fi

\bibitem[{Amazon}(2025{\natexlab{a}})]{amazonFBA}
{Amazon}.
\newblock {Get Started with Fulfillment by Amazon (FBA)}, 2025{\natexlab{a}}.
\newblock URL \url{https://sell.amazon.com/start/fba}.
\newblock Accessed: 2025-04-11.

\bibitem[{Amazon}(2025{\natexlab{b}})]{amazonFlexibleJobs}
{Amazon}.
\newblock {Flexible Schedule Jobs at Amazon}, 2025{\natexlab{b}}.
\newblock URL \url{https://hiring.amazon.com/search/flexible-schedule-jobs}.
\newblock Accessed: 2025-04-11.

\bibitem[{Amazon Science}(2023)]{oneill2023amazon}
{Amazon Science}.
\newblock {How Amazon reworked its fulfillment network to meet customer demand}, July 2023.
\newblock URL \url{https://www.amazon.science/news-and-features/how-amazon-reworked-its-fulfillment-network-to-meet-customer-demand}.
\newblock Accessed: 2025-04-12.

\bibitem[Azar et~al.(1994)Azar, Broder, Karlin, and Upfal]{azar1994balanced}
Yossi Azar, Andrei~Z Broder, Anna~R Karlin, and Eli Upfal.
\newblock Balanced allocations.
\newblock In \emph{Proceedings of the Twenty-Sixth Annual ACM Symposium on Theory of Computing}, pages 593--602, 1994.

\bibitem[Chung and Lu(2006)]{chung2006concentration}
Fan Chung and Linyuan Lu.
\newblock Concentration inequalities and martingale inequalities: a survey.
\newblock \emph{Internet Mathematics}, 3\penalty0 (1):\penalty0 79--127, 2006.

\bibitem[Dubhashi and Panconesi(2009)]{dubhashi2009concentration}
Devdatt~P Dubhashi and Alessandro Panconesi.
\newblock \emph{Concentration of measure for the analysis of randomized algorithms}.
\newblock Cambridge University Press, 2009.

\bibitem[Durrett(2019)]{durrett2019probability}
Rick Durrett.
\newblock \emph{Probability: theory and examples}, volume~49.
\newblock Cambridge University Press, 2019.

\bibitem[Elmachtoub and Hamilton(2021)]{elmachtoub2021power}
Adam~N Elmachtoub and Michael~L Hamilton.
\newblock The power of opaque products in pricing.
\newblock \emph{Management Science}, 2021.

\bibitem[Elmachtoub et~al.(2015)Elmachtoub, Wei, and Zhou]{elmachtoub2015retailing}
Adam~N Elmachtoub, Yehua Wei, and Yeqing Zhou.
\newblock Retailing with opaque products.
\newblock \emph{Available at SSRN 2659211}, 2015.

\bibitem[Elmachtoub et~al.(2019)Elmachtoub, Yao, and Zhou]{elmachtoub2019value}
Adam~N Elmachtoub, David Yao, and Yeqing Zhou.
\newblock The value of flexibility from opaque selling.
\newblock \emph{Available at SSRN 3483872}, 2019.

\bibitem[Elmachtoub et~al.(2021)Elmachtoub, Lei, and Zhou]{elmachtoub2021customerflex}
Adam~N Elmachtoub, Xiao Lei, and Yeqing Zhou.
\newblock The value of consumer flexibility in scheduled service systems, 2021.

\bibitem[Escallon-Barrios and Smilowitz(2025)]{escallon2025sharing}
Mariana Escallon-Barrios and Karen Smilowitz.
\newblock Sharing the schedule: Nonprofit staffing for volunteers and employees.
\newblock In \emph{Nonprofit Operations and Supply Chain Management: Theory and Practice}, pages 99--131. Springer, 2025.

\bibitem[Fay and Xie(2008)]{fay2008probabilistic}
Scott Fay and Jinhong Xie.
\newblock Probabilistic goods: A creative way of selling products and services.
\newblock \emph{Marketing Science}, 27\penalty0 (4):\penalty0 674--690, 2008.

\bibitem[Fay and Xie(2015)]{fay2015timing}
Scott Fay and Jinhong Xie.
\newblock Timing of product allocation: Using probabilistic selling to enhance inventory management.
\newblock \emph{Management Science}, 61\penalty0 (2):\penalty0 474--484, 2015.

\bibitem[Feldman et~al.(2025)Feldman, Huang, and Chen]{feldmanapproximation}
Jacob Feldman, Yukai Huang, and Xingxing Chen.
\newblock Approximation schemes for dynamic pricing with opaque products.
\newblock 2025.

\bibitem[Freund and van Ryzin(2021)]{freund2021pricing}
Daniel Freund and Garrett van Ryzin.
\newblock Pricing fast and slow: Limitations of dynamic pricing mechanisms in ride-hailing.
\newblock \emph{Available at SSRN 3931844}, 2021.

\bibitem[Fu et~al.(2025)Fu, Li, and Rujeerapaiboon]{fu2025optimal}
Mingyang Fu, Xiaobo Li, and Napat Rujeerapaiboon.
\newblock Optimal competitive ratio in opaque sales.
\newblock \emph{Available at SSRN 5130478}, 2025.

\bibitem[Gallego and Phillips(2004)]{gallego2004revenue}
Guillermo Gallego and Robert Phillips.
\newblock Revenue management of flexible products.
\newblock \emph{Manufacturing \& Service Operations Management}, 6\penalty0 (4):\penalty0 321--337, 2004.

\bibitem[Golrezaei and Yao(2021)]{golrezaei2021online}
Negin Golrezaei and Evan Yao.
\newblock Online resource allocation with time-flexible customers.
\newblock \emph{arXiv preprint arXiv:2108.03517}, 2021.

\bibitem[Grimmett and Stirzaker(2020)]{grimmett2020probability}
Geoffrey Grimmett and David Stirzaker.
\newblock \emph{Probability and random processes}.
\newblock Oxford University Press, 2020.

\bibitem[Harris(1990)]{harris1990many}
Ford~W Harris.
\newblock How many parts to make at once.
\newblock \emph{{Operations Research}}, 38\penalty0 (6):\penalty0 947--950, 1990.

\bibitem[Housni et~al.(2025)Housni, Elmachtoub, Sheth, and Shi]{housni2025price}
Omar~El Housni, Adam~N Elmachtoub, Harsh Sheth, and Jiaqi Shi.
\newblock Price and assortment optimization under the multinomial logit model with opaque products.
\newblock \emph{arXiv preprint arXiv:2502.08124}, 2025.

\bibitem[Hssaine et~al.(2024)Hssaine, Topaloglu, and van Ryzin]{hssaine2024target}
Chamsi Hssaine, Huseyin Topaloglu, and Garrett van Ryzin.
\newblock Target-following online resource allocation using proxy assignments.
\newblock \emph{arXiv preprint arXiv:2412.12321}, 2024.

\bibitem[Ignall(1969)]{ignall1969optimal}
Edward Ignall.
\newblock Optimal continuous review policies for two product inventory systems with joint setup costs.
\newblock \emph{Management Science}, 15\penalty0 (5):\penalty0 278--283, 1969.

\bibitem[Jerath et~al.(2010)Jerath, Netessine, and Veeraraghavan]{jerath2010revenue}
Kinshuk Jerath, Serguei Netessine, and Senthil~K Veeraraghavan.
\newblock Revenue management with strategic customers: Last-minute selling and opaque selling.
\newblock \emph{Management Science}, 56\penalty0 (3):\penalty0 430--448, 2010.

\bibitem[Jiang(2007)]{jiang2007price}
Yabing Jiang.
\newblock Price discrimination with opaque products.
\newblock \emph{Journal of Revenue and Pricing Management}, 6\penalty0 (2):\penalty0 118--134, 2007.

\bibitem[Jordan and Graves(1995)]{jordan1995principles}
William~C Jordan and Stephen~C Graves.
\newblock Principles on the benefits of manufacturing process flexibility.
\newblock \emph{Management Science}, 41\penalty0 (4):\penalty0 577--594, 1995.

\bibitem[McFadden(1973)]{mcfadden1973conditional}
Daniel McFadden.
\newblock Conditional logit analysis of qualitative choice behavior.
\newblock 1973.

\bibitem[Mitzenmacher(1996)]{mitzenmacher1996power}
M~Mitzenmacher.
\newblock The power of two choices in randomized load balancing.
\newblock \emph{PhD thesis, University of California at Berkeley}, 1996.

\bibitem[Mitzenmacher(2001)]{mitzenmacher2001power}
Michael Mitzenmacher.
\newblock The power of two choices in randomized load balancing.
\newblock \emph{IEEE Transactions on Parallel and Distributed Systems}, 12\penalty0 (10):\penalty0 1094--1104, 2001.

\bibitem[Peres et~al.(2010)Peres, Talwar, and Wieder]{peres2010}
Yuval Peres, Kunal Talwar, and Udi Wieder.
\newblock The (1+ $\beta$)-choice process and weighted balls-into-bins.
\newblock In \emph{Proceedings of the twenty-first annual ACM-SIAM symposium on Discrete Algorithms}, pages 1613--1619. SIAM, 2010.

\bibitem[Raab and Steger(1998)]{raab1998balls}
Martin Raab and Angelika Steger.
\newblock {“Balls into bins”—A simple and tight analysis}.
\newblock In \emph{International Workshop on Randomization and Approximation Techniques in Computer Science}, pages 159--170. Springer, 1998.

\bibitem[Ren and Huang(2022)]{ren2022opaque}
Hang Ren and Tingliang Huang.
\newblock Opaque selling and inventory management in vertically differentiated markets.
\newblock \emph{Manufacturing \& Service Operations Management}, 24\penalty0 (5):\penalty0 2543--2557, 2022.

\bibitem[Richa et~al.(2001)Richa, Mitzenmacher, and Sitaraman]{richa2001power}
Andrea~W Richa, M~Mitzenmacher, and R~Sitaraman.
\newblock The power of two random choices: A survey of techniques and results.
\newblock \emph{Combinatorial Optimization}, 9:\penalty0 255--304, 2001.

\bibitem[Salop(1979)]{salop1979monopolistic}
Steven~C Salop.
\newblock Monopolistic competition with outside goods.
\newblock \emph{The Bell Journal of Economics}, pages 141--156, 1979.

\bibitem[Silver(1965)]{silver1965some}
Edward~A Silver.
\newblock Some characteristics of a special joint-order inventory model.
\newblock \emph{Operations Research}, 13\penalty0 (2):\penalty0 319--322, 1965.

\bibitem[Tan et~al.(2024)Tan, Qin, Chen, Liu, and Yu]{tan2024online}
Zheng Tan, Hanzhang Qin, Yan Chen, Lindong Liu, and Yugang Yu.
\newblock Online bin packing with load-balancing.
\newblock \emph{Available at SSRN}, 2024.

\bibitem[{Too Good To Go}(2025)]{toogoodtogo}
{Too Good To Go}.
\newblock Too good to go: Save food from local restaurants and stores, 2025.
\newblock URL \url{https://www.toogoodtogo.com/en-us}.
\newblock Accessed: 2025-04-11.

\bibitem[Tsitsiklis and Xu(2013)]{tsitsiklis2013power}
John~N Tsitsiklis and Kuang Xu.
\newblock On the power of (even a little) resource pooling.
\newblock \emph{Stochastic Systems}, 2\penalty0 (1):\penalty0 1--66, 2013.

\bibitem[Tsitsiklis and Xu(2017)]{tsitsiklis2017flexible}
John~N Tsitsiklis and Kuang Xu.
\newblock Flexible queueing architectures.
\newblock \emph{Operations Research}, 65\penalty0 (5):\penalty0 1398--1413, 2017.

\bibitem[Wallace and Whitt(2005)]{wallace2005staffing}
Rodney~B Wallace and Ward Whitt.
\newblock A staffing algorithm for call centers with skill-based routing.
\newblock \emph{Manufacturing \& Service Operations Management}, 7\penalty0 (4):\penalty0 276--294, 2005.

\bibitem[Wang et~al.(2021)Wang, Wang, and Zhang]{wang2021review}
Shixin Wang, Xuan Wang, and Jiawei Zhang.
\newblock A review of flexible processes and operations.
\newblock \emph{Production and Operations Management}, 30\penalty0 (6):\penalty0 1804--1824, 2021.

\bibitem[Xiao and Chen(2014)]{xiao2014evaluating}
Yongbo Xiao and Jian Chen.
\newblock Evaluating the potential effects from probabilistic selling of similar products.
\newblock \emph{Naval Research Logistics (NRL)}, 61\penalty0 (8):\penalty0 604--620, 2014.

\bibitem[Zhang et~al.(2024)Zhang, Liu, and Long]{zhang2024less}
Hailun Zhang, Jie Liu, and Zhenghua Long.
\newblock Less is more: a general framework of opaque selling.
\newblock \emph{Available at SSRN 4817131}, 2024.

\bibitem[Zhang et~al.(2015)Zhang, Joseph, and Subramaniam]{zhang2015probabilistic}
Zelin Zhang, Kissan Joseph, and Ramanathan Subramaniam.
\newblock Probabilistic selling in quality-differentiated markets.
\newblock \emph{Management Science}, 61\penalty0 (8):\penalty0 1959--1977, 2015.

\bibitem[Zhou(2021)]{zhou2021supply}
Yeqing Zhou.
\newblock \emph{Supply Chain and Service Operations with Demand-Side Flexibility}.
\newblock Columbia University, 2021.

\bibitem[Zhu and Topaloglu(2024)]{zhu2022performance}
Wenchang Zhu and Huseyin Topaloglu.
\newblock Performance guarantees for network revenue management with flexible products.
\newblock \emph{Manufacturing \& Service Operations Management}, 26\penalty0 (1):\penalty0 252--270, 2024.

\end{thebibliography}
